\documentclass{amsart}

\usepackage[T1]{fontenc}
\usepackage[utf8]{inputenc}
\usepackage[english]{babel}

\usepackage{hyperref}
\usepackage{amsthm, amssymb}
\usepackage{amsmath}
\usepackage{xifthen}
\usepackage{enumitem}
\usepackage{url}

\newtheorem{theorem}{Theorem}[section]
\newtheorem{proposition}[theorem]{Proposition}
\newtheorem{lemma}[theorem]{Lemma}
\newtheorem{corollary}[theorem]{Corollary}

\theoremstyle{definition}
\newtheorem{definition}{Definition}[section]

\theoremstyle{remark}
\newtheorem{remark}{Remark}[section]

\numberwithin{equation}{section}

\newenvironment{bigcases}{\left\{\begin{aligned}}{\end{aligned}\right.}
\newcommand{\edintertext}[1]{%
  \noalign{%
    \vskip\belowdisplayshortskip
    \vtop{\hsize=\linewidth#1\par
    \expandafter}%
    \expandafter\prevdepth\the\prevdepth
  }%
}

\newcommand\bpr[1]{\left(#1\right)}
\newcommand\bsq[1]{\left[#1\right]}

\newcommand{\wto}{\rightharpoonup}
\newcommand{\xto}[1]{\xrightarrow{#1}}
\newcommand{\sub}{\subset}
\newcommand{\subeq}{\subseteq}

\newcommand{\emb}{\hookrightarrow}

\newcommand\R{\mathbb{R}}

\newcommand\Z{\mathbb{Z}}
\newcommand\N{\mathbb{N}}
\newcommand{\Bal}[2]{B_{#1}(#2)}
\newcommand{\cBal}[2]{\overline{B_{#1}(#2)}}
\newcommand\Sob[1][k]{H^{#1}}
\newcommand{\hSob}[1][k,2]{D^{#1}}
\newcommand{\Cct}{C_c^\infty}

\newcommand\abs[1]{\left\lvert #1 \right\rvert}
\newcommand\tabs[1]{\lvert #1 \rvert}
\newcommand\Babs[1]{\Big\lvert#1\Big\rvert}
\newcommand{\norm}[1]{\left\| #1 \right\|}
\newcommand{\dprod}[2]{\left\langle #1,#2 \right\rangle}
\newcommand{\vprod}[3][g]{\langle #2,#3 \rangle_{#1}}

\newcommand\Snorm[3][(M)]{\norm{#3}_{\Sob[#2]#1}}
\newcommand\Hnorm[3][(\R^n)]{\norm{#3}_{\hSob[#2]#1}}
\newcommand{\Lnorm}[3][]{\norm{#3}_{L^{#2}#1}}
\newcommand{\inorm}[2][]{\Lnorm[#1]{\infty}{#2}}

\newcommand{\dist}[2][g]{\operatorname{dist}_{#1}(#2)}
\newcommand\at[1]{\Big\lvert_{#1}}

\newcommand\bigO{\mathcal{O}}
\newcommand{\intM}[2][M]{\int_{#1} #2\, dv_g}
\newcommand{\intMg}[3]{\int_{#2} #3~ dv_{#1}}
\newcommand{\dg}[2][g]{d_{#1}{(#2)}}
\newcommand{\spanned}[1]{\operatorname{span}\left\{#1\right\}}
\newcommand{\Vol}[1][g]{\operatorname{Vol}_{#1}}
\newcommand{\such}{\,:\,}
\newcommand{\tsum}{\textstyle\sum\limits}

\renewcommand\div{\operatorname{div}}
\newcommand{\Dsur}[2]{\frac{d#1}{d#2}}
\newcommand\dsur[2]{\frac{\partial #1}{\partial #2}}
\newcommand{\D}{\Delta}
\newcommand{\Dg}[1][g]{\D_{#1}}

\renewcommand{\epsilon}{\varepsilon}

\renewcommand{\phi}{\varphi}
\renewcommand{\a}{\alpha}
\newcommand{\inj}{i_g}
\newcommand{\sqa}{\sqrt{\a}}
\newcommand{\deu}{{2^\sharp}}
\newcommand{\deuk}[1]{2_{#1}^\sharp}

\newcommand{\pk}{\rho_{n,k}}
\newcommand{\pa}{\nu}
\newcommand{\zm}[1][]{z_{#1}, \mu_{#1}}
\newcommand{\zma}[1][]{\a_{#1}, \pa_{#1}}
\newcommand{\param}[1][]{\mathcal{P}(\tau_{#1},\a_{#1})}
\newcommand{\parama}[1][\tau_\a]{\mathcal{P}(#1,\a)}
\newcommand{\mani}[1][]{\mathcal{B}_{\tau_{#1},\a_{#1}}}
\newcommand{\Set}{\mathcal{S}}
\newcommand{\Bub}{\operatorname{B}}
\newcommand{\Ker}[1][\zma]{\mathcal{K}_{\ifthenelse{\isempty{#1}}{\Bub}{#1}}}
\newcommand{\Keri}[1][i]{\mathcal{K}_{\zma[#1]}}
\newcommand{\Kera}[1][\pa_\a]{\mathcal{K}_{\a,#1}}
\newcommand{\proKe}[1][\zma]{\Pi_{\Ker[#1]^\perp}}
\newcommand{\proKa}[1][\pa_\a]{\Pi_{\Kera[#1]^\perp}}
\newcommand{\proKi}[1][i]{\Pi_{\Keri[#1]^\perp}}

\newcommand{\TBub}[1][]{V_{\zma[#1]}}
\newcommand{\TBuba}[1][\pa_{\a}]{V_{\a,#1}}
\newcommand{\TB}[1]{V_{#1}}
\newcommand{\Ze}[1]{Z^{#1}}
\newcommand{\Zed}[2][]{Z^{#2}_{\zma[#1]}}
\newcommand{\Zeda}[2][\pa_{\a}]{Z^{#2}_{\a,#1}}
\newcommand{\Gga}[1][]{\operatorname{G}_{g,\a_{#1}}}
\newcommand{\Lia}[1][]{L_{\zma[#1]}}
\newcommand{\sqai}[1][i]{\sqrt{\a_{#1}}}
\newcommand{\err}[2][]{R^{#2}_{\zma[#1]}}
\newcommand{\erra}[2][\pa_\a]{R^{#2}_{\a,#1}}
\newcommand{\rad}{\eta}

\newcommand{\Zdif}[1][j]{\mu\dsur{}{z_{#1}}\Theta_\a(z,\cdot) \Bub_{\zm}}
\newcommand{\U}[1][i]{U_{#1}}

\newcommand{\Tu}{\Tilde{u}}
\newcommand{\bu}{\Bar{u}}
\newcommand{\Tg}{\Tilde{g}}
\newcommand{\Tp}{\Tilde{\phi}}
\newcommand{\Tps}{\Tilde{\psi}}
\newcommand{\Cphi}{\check{\phi}}
\newcommand{\Hps}{\hat{\psi}}

\newcommand{\Tx}{\Tilde{x}}
\newcommand{\bz}{\Bar{z}}
\newcommand{\bmu}{\Bar{\mu}}
\newcommand{\bpa}{\Bar{\pa}}
\newcommand{\Tmu}{\Tilde{\mu}}
\newcommand{\Tz}{\Tilde{z}}
\newcommand{\Tpa}{\Tilde{\pa}}

\begin{document}

\title[Optimal constant for Sobolev inequalities]{Attaining the optimal constant for higher-order Sobolev inequalities on manifolds via asymptotic analysis}

\author{Lorenzo Carletti}
\date{August 2024}
\thanks{This publication is supported by the French Community of Belgium as part of the funding of a FRIA grant.}
\address{Lorenzo Carletti, Université Libre de Bruxelles, Service d'Analyse, Boulevard du Triomphe - Campus de la Plaine, 1050 Bruxelles, Belgique}
\email{\url{lorenzo.carletti@ulb.be}}


\begin{abstract}
Let $(M,g)$ be a closed Riemannian manifold of dimension $n$, and $k\geq 1$ an integer such that $n>2k$. 
We show that there exists $B_0>0$ such that for all $u \in H^{k}(M)$, 
\[\|u\|_{L^{2^\sharp}(M)}^2 \leq K_0^2 \int_M |\Delta_g^{k/2} u|^2 \,dv_g + B_0 \|u\|_{H^{k-1}(M)}^2,\] 
where $2^\sharp = \frac{2n}{n-2k}$ and $\Delta_g = -\operatorname{div}_g(\nabla\cdot)$. Here $K_0$ is the optimal constant for the Euclidean Sobolev inequality $\big(\int_{\mathbb{R}^n} |u|^{2^\sharp}\big)^{2/2^\sharp} \leq K_0^2 \int_{\mathbb{R}^n} |\nabla^k u|^2$ for all $u \in C_c^\infty(\mathbb{R}^n)$.  
This result is proved as a consequence of the pointwise blow-up analysis for a sequence of positive solutions $(u_\alpha)_\alpha$ to polyharmonic critical non-linear equations of the form $(\Delta_g + \alpha)^k u = u^{2^\sharp-1}$ in $M$. We obtain a pointwise description of $u_\alpha$, with explicit dependence in $\alpha$ as $\alpha\to \infty$. 
\end{abstract}
\maketitle

\section{Introduction and statement of the results}
Let $(M,g)$ be a smooth compact Riemannian manifold without boundary, of dimension $n\geq 3$. Let also $k\geq 1$ be an integer such that $n>2k$. We define the Sobolev space $\Sob(M)$ as the closure of $C^\infty(M)$ with respect to the norm 
\begin{equation}\label{def:Snorm}
    \Snorm{k}{u}^2 := \sum_{l=0}^k \intM{\abs{\Dg^{l/2}u}^2},
\end{equation}
where 
\[  \abs{\Dg^{l/2} u} := \begin{cases}
    \abs{\Dg^m u} &\text{if $l=2m$ is even}\\
    \abs{\nabla\Dg^m u}_g &\text{if $l=2m+1$ is odd}
\end{cases}.
    \]
Here $\Dg := -\div_g(\nabla~\cdot)$ is the Laplace-Beltrami operator on $M$. 
\par Thanks to a Bochner-Lichnerowicz-Weitzenböck-type formula, as showed in \cite{Rob11}, the norm \eqref{def:Snorm} is equivalent to the usual Sobolev norm 
\[  \bpr{\sum_{\abs{\beta}\leq k} \intM{\abs{\nabla^\beta u}_g^2}}^{1/2},
    \] 
where $\beta$ is now a multi-index. On a compact manifold, the Sobolev space $\Sob(M)$ embeds continuously into $L^\deu(M)$, where $\deu := \frac{2n}{n-2k}$ is the critical exponent. That is, there exist $A,\,B>0$ such that 
\begin{equation}\label{eq:sobineq}  \Lnorm[(M)]{\deu}{u}^2 \leq A \Lnorm[(M)]{2}{\Dg^{k/2}u}^2 + B \Snorm{k-1}{u}^2 \qquad \forall~u\in \Sob(M).
    \end{equation}
This paper is concerned with the optimization of the constant $A$. We are motivated by the following question: What is the smallest possible $A>0$ for which there exists $B>0$ such that \eqref{eq:sobineq} holds?
\par In the Euclidean setting, we define $K_0>0$ to be the optimal constant for the Sobolev inequality
\begin{equation}\label{eq:optSob}
    \Lnorm[(\R^n)]{\deu}{u}^2 \leq K_0^2 \Lnorm[(\R^n)]{2}{\Dg[\xi]^{k/2}u}^2 \qquad \forall~u\in \Cct(\R^n). 
    \end{equation}
Here $\xi$ is the flat metric on $\R^n$, and $\Dg[\xi] = -\tsum_{i=1}^n \partial_i^2$.
This constant has an explicit expression depending only on $n$ and $k$, which was computed in \cite{Swa92}, and a variational characterization:
\[  \frac{1}{K_0^2} := \inf_{u \in \Cct(\R^n)\setminus\{0\}} \frac{\displaystyle\int_{\R^n}\abs{\Dg[\xi]^{k/2}u}^2 dy}{\bpr{\int_{\R^n}\abs{u}^\deu dy}^{2/\deu}}.
    \]

    \par We introduce the following norm, for $ u \in \Cct(\R^n)$,
\[  \Hnorm{k,2}{u}^2 := \int_{\R^n} \abs{\Dg[\xi]^{k/2}u}^2 dy. 
    \]
We also define the homogeneous Sobolev space $\hSob(\R^n)$ as the closure of $\Cct(\R^n)$ with respect to $\Hnorm{k,2}{\cdot}$. The optimal Sobolev inequality \eqref{eq:optSob} holds for all $u \in \hSob(\R^n)$, by density. The extremals for this inequality were studied in \cite{Swa92}, and are all rescalings and translations of the so-called \emph{standard bubble}.
\begin{definition}[the Euclidean bubble]
    We let $\Bub \in \hSob(\R^n)$ be defined as 
    \begin{equation}\label{def:bub}
        \Bub(y) := \big(1+\pk\abs{y}^2\big)^{\frac{n-2k}{2}} \qquad y\in \R^n,
    \end{equation}
    where $\pk  := \bpr{\prod\limits_{l=-k}^{k-1} (n+2l)}^{-1/k}$ is chosen such that $\Bub$ satisfies 
    \begin{equation}\label{eq:bub}
        \Dg[\xi]^k \Bub = \Bub^{\deu-1} \qquad \text{in }\R^n.
    \end{equation}         
    We have $\Hnorm{k,2}{\Bub} = K_0^{-\frac{n}{2k}}$ and 
    \[  \Lnorm[(\R^n)]{\deu}{\Bub} = K_0 \Hnorm{k,2}{\Bub}.
        \]
\end{definition}
As showed in \cite{WeiXu99}, all the positive solutions of \eqref{eq:bub} are given by the transformations of $\Bub$, for $\mu>0$ and $z\in \R^n$,
\begin{equation}\label{eq:transfo}
    \mu^{-\frac{n-2k}{2}}\Bub\Big(\frac{y-z}{\mu}\Big) = \bpr{\frac{\mu}{\mu^2 + \pk \abs{y-z}^2}}^{\frac{n-2k}{2}}.
\end{equation}

\par In this article we show that, in the inhomogeneous context of a closed Riemannian manifold, we can choose $A=K_0^2$ in \eqref{eq:sobineq}. Our main result is as follows.
\begin{theorem}\label{prop:mainest}
    Let $(M,g)$ be a smooth compact Riemannian manifold without boundary of dimension $n$, and let $k\geq 1$ be an integer such that $n>2k$. There exists $B_0>0$ such that for all $u\in \Sob(M)$,
    \begin{equation}\label{eq:SoboptHk}
        \Lnorm[(M)]{\deu}{u}^2 \leq K_0^2 \Lnorm[(M)]{2}{\Dg^{k/2}u}^2 + B_0 \Snorm{k-1}{u}^2.
    \end{equation}
\end{theorem}
Hebey and Vaugon (1996) proved the case $k=1$, 
they considered the case of a compact manifold \cite{HebVau96}, and also proved the result on a complete Riemannian manifold \cite{HebVau95}. Subsequently, Hebey considered the case $k=2$ on a compact manifold \cite{Heb03}. Theorem \ref{prop:mainest} extends these results for an arbitrary order $k\geq 1$. Note that for the case $k=1$, and for $p\geq 1$, the analog embedding of the Sobolev space $\Sob[1,p](M)$ in the Lebesgue space $L^{\frac{np}{n-p}}(M)$ has been largely studied, see for instance \cite{Dru99}. We also refer to \cite{Aub75, Tal76} for the expression of the optimal constant $K_0$ depending on $n$ and $p$ in this case, and the corresponding extremals on $\R^n$. See also \cite{Heb99,DruHeb02} for a broader discussion about the optimal constants for Sobolev inequalities on manifolds. 
Finally, we refer to \cite{DjHeLe00} where similar optimal inequalities with a norm involving the celebrated Paneitz-Branson operator were considered, for $k=2$. 

\par In the case of $k\geq 1$, Mazumdar \cite{Maz16} proved, with a covering argument, the following two facts:
\begin{enumerate}
    \item Any pair of constants $A,B>0$ for which \eqref{eq:sobineq} holds satisfies $A\geq K_0^2$;
    \item For all $\epsilon >0$, there exists a constant $B_\epsilon >0$ such that 
    \begin{equation}\label{eq:sobepineq}
        \Lnorm[(M)]{\deu}{u}^2 \leq \big(K_0^2 + \epsilon\big)\Lnorm[(M)]{2}{\Dg^{k/2}u}^2 + B_\epsilon \Snorm{k-1}{u}^2.
    \end{equation} 
\end{enumerate}
In particular, \eqref{eq:sobepineq} shows that $\inf\{A>0\,:\, \exists~B>0 \text{ for which \eqref{eq:sobineq} holds}\} = K_0^2$. We show in this paper that the optimal constant is in fact attained, that is that one can choose $\epsilon=0$ in \eqref{eq:sobepineq}.
\par We briefly describe the strategy of the proof of Theorem \ref{prop:mainest}, which is detailed in section \ref{sec:proof1}. As in \cite{HebVau96,Heb03}, the proof will proceed by contradiction. Assuming that $A=K_0^2$ is not reached is equivalent to considering that for all $\alpha>0$,
\begin{equation}\label{tmp:introx}
    \inf_{u\in \Sob(M)\setminus\{0\}} \frac{\displaystyle\intM{(\Dg+\a)^k u~ u}}{\bpr{\intM{\abs{u}^\deu}}^{2/\deu}} < \frac{1}{K_0^2}.
\end{equation}
Assumption \eqref{tmp:introx} implies that for each $\a >0$, there exists a positive function $u_\a \in \Sob(M)$ solution to the critical equation 
\begin{equation}\label{eq:critbub}
    (\Dg+\a)^k u = u^{\deu-1} \qquad \text{in }M,
\end{equation}
satisfying $\Snorm{k}{u_\a} \leq K_0^{-\frac{n}{2k}}$. The main part of this article is devoted to the study of such sequences of solutions: we obtain a precise pointwise description of the blow-up behavior of $u_\a$ as $\a \to \infty$, that we use to reach a contradiction.
\par As we will detail in section \ref{sec:prelim}, for a sequence $\pa_\a = (z_\a, \mu_\a) \in M \times (0,+\infty)$ such that $\a\mu_\a^2 \to 0$ as $\a \to \infty$, we let
\[  \TBuba(x) = \Theta_\a(z_\a,x) \bpr{\frac{\mu_\a}{\mu_\a^2+ \pk\dg{z_\a,x}^2}}^{\frac{n-2k}{2}} \qquad x \in M.
    \]
Here $\Theta_\a(z_\a,\cdot)$ is a smooth function (see definition \ref{def:defTBub}) which is equal to 1 in a ball of radius $1/\sqa$ around the center $z_\a$, and then decays as $e^{-\sqa\dg{z_\a, \cdot}/2}$, with support contained inside a ball of radius $\inj$, the injectivity radius of the compact manifold $M$.
The function $\TBuba$ is then an almost solution to \eqref{eq:critbub}, it satisfies
\[  (\Dg+\a)^k \TBuba = \TBuba^{\deu-1} + o(1),
    \] 
where $o(1)\to 0$ in $\Sob[-k](M)$ and comes with explicit pointwise estimates (see Proposition \ref{prop:estRa}).  

\par We will obtain the required contradiction, and thus the proof of Theorem \ref{prop:mainest}, as a consequence of the following result. Define
\begin{equation}\label{def:rad}
    \rad(t) := \begin{cases}
        t\big(1+\abs{\log t}\big) & \text{if } n = 2k+1\\
        t^{3/2} & \begin{aligned}[t]
            &\text{if } k=1, \,n\geq 4\\
            &~\text{or } k\geq 2,\, n=2k+2
        \end{aligned}\\
        t^2 & \text{if } k\geq 2,\, n\geq 2k+3 
    \end{cases} \qquad t \geq 0.
\end{equation}
\begin{theorem}\label{prop:mainprinc}
    Let $(u_\a)_\a$ be a sequence of positive functions in $\Sob(M)$, such that $u_\a$ is solution to \eqref{eq:critbub} for each $\a$. Assume that there exist a sequence $(\epsilon_\a)_\a$ such that $\epsilon_\a \to 0$ as $\a \to \infty$, and 
    \begin{equation}\label{eq:assump}
        \Snorm{k}{u_\a} \leq K_0^{-\frac{n}{2k}}+\epsilon_\a \qquad \forall~\a.
    \end{equation}
    Then, there exist a sequence $(z_\a)_\a$ of points in $M$, a sequence $(\mu_\a)_\a$ of positive numbers such that $\a\mu_\a^2 \to 0$ as $\a\to \infty$, and a constant $C>0$, such that the following holds. Up to a subsequence, writing $\pa_\a = (\zm[\a])$, we have for all $x\in M$ and $l=0,\ldots 2k-1$ that 
    \begin{multline*}  
        (\mu_\a + \dg{z_\a,x})^l \abs{\nabla^l \bpr{u_\a - \TBuba}(x)}_g \leq C \bpr{\frac{\mu_\a}{\mu_\a^2 + \pk \dg{z_\a,x}^2}}^{\frac{n-2k}{2}}\\ \times \begin{cases}
            \rad\big(\sqa (\mu_\a + \dg{z_\a, x})\big) & \text{when } \sqa\dg{z_\a,x} \leq 1\\
            \a^k \dg{z_\a,x}^{2k} e^{-\sqa\dg{z_\a,x}/2} & \text{when } \sqa\dg{z_\a,x} \geq 1
        \end{cases},
    \end{multline*} 
    where $\TBuba$ is as defined in \eqref{def:TBub}. 
\end{theorem}
\par As a consequence of this theorem, we have immediately
\[  u_\a - \TBuba \to 0 \qquad \text{in } \Sob(M) \quad \text{as } \a\to \infty.
    \] 
Energy decomposition results of this type for sequences of solutions to polyharmonic critical equations are known to hold for the case of bounded coefficients, see \cite{HebRob01,Maz17}. 
The main novelty of Theorem \ref{prop:mainprinc} is twofold. First, we obtain strong pointwise estimates on $u_\a-\TBuba$ and its derivatives, and not just an energy decomposition. Second, we obtain an explicit dependence in $\a$, as $\a \to \infty$, in the a priori estimates. This part is shown in section \ref{sec:pointwise}, following the strategy of \cite{Pre24}.
In the context of Theorem \ref{prop:mainprinc}, the low energy of the sequence $(u_\a)_\a$ only allows the presence of a single bubble. While this simplifies some arguments compared to the bubble-tree case of \cite{Pre24}, the arbitrary order $k$ of the polyharmonic operator $(\Dg+\a)^k$, and more importantly the presence of the diverging coefficient $\a$ in the lower-order terms 
\begin{equation}\label{eq:lowerterm}
    \a^{k-l}\Dg^l u \qquad l=0,\ldots k-1
    \end{equation}
required us to perform significant and novel modifications to the theory. In the course of the proof, a precise pointwise description of the Green's function for the operator $(\Dg+\a)^k$ on $M$ is required to obtain the estimates of Theorem \ref{prop:mainprinc}, which was studied by the author in a previous work \cite{Car24}.
\par With the use of this pointwise decomposition, the proof of Theorem \ref{prop:mainest} reduces to explicit computations, which are carried out in section \ref{sec:proof1}.
Our approach is therefore completely different from the approach in \cite{Heb03}, since we see Theorem \ref{prop:mainest} as a direct consequence of Theorem \ref{prop:mainprinc}. We believe however that Theorem \ref{prop:mainprinc} has a strong interest in itself, as it serves as a toy-model for general critical polyharmonic equations with diverging coefficients on manifolds. In this case, the proof provides a very precise understanding of the obstructions that do not allow such low-energy solutions to exist as $\a \to \infty$. 
The analytical machinery developed here is quite flexible and will be adapted in a second step to study more general sequences of solutions to polyharmonic equations of this type. Other examples where equations of the form \eqref{eq:critbub} are considered are found for instance in \cite{Wan95,AdPaYa95,Rey02} for the case $k=1$, where the possible concentration phenomena of sequences of solutions to $-\D u+\lambda u = u^{\frac{n+2}{n-2}}$ are investigated on an open set of $\R^n$ with Neumann boundary conditions, as $\lambda\to \infty$. In the case $k=2$, we refer to \cite{FeHeRo05}, where the authors studied the energy of solutions to \eqref{eq:critbub} as $\a \to \infty$. As this article was being reviewed, we have learned that Zeitler \cite{Zei24} has independently obtained Theorem \ref{prop:mainest}, and was able to adapt the method of \cite{Heb03} to the higher-order case.
\par This article is structured as follows: Sections \ref{sec:prelim}-\ref{sec:pointwise} focus on the proof of Theorem \ref{prop:mainprinc}, while Theorem \ref{prop:mainest} is proved in section \ref{sec:proof1}. In section \ref{sec:prelim}, we provide the necessary definitions and computations, and prove technical results concerning the function $\TBub$. Sections \ref{sec:Hktheory} and \ref{sec:pointwise} are devoted to the construction of solutions to a more general version of \eqref{eq:critbub}, as perturbations around the modified bubble $\TBub$. We first develop in section \ref{sec:Hktheory} the linear and non-linear theory leading to existence and uniqueness of such solutions in $\Sob(M)$. In a second step, we extend the theory to obtain pointwise estimates in section \ref{sec:pointwise}, proving Theorem \ref{prop:mainprinc}. This is the core of the analysis of the paper.
Finally, in section \ref{sec:proof1}, we use the strong decomposition of Theorem \ref{prop:mainprinc} to prove Theorem \ref{prop:mainest}. 

\section{Preliminary results}\label{sec:prelim}
In this section we gather some basic definitions and results which will be needed to develop the full description of the blow-up for sequences of solutions to \eqref{eq:critbub}. Remark first that the operator $(\Dg+\a)^k$ satisfies the maximum principle since it is in factorized form. Moreover, it is coercive, and
\begin{equation} \label{eq:Pacoerc}
    \Snorm{k}{u}^2 \leq \dprod{(\Dg+\a)^k u}{u}_{\Sob[-k],\Sob}, \qquad u \in \Sob(M)
    \end{equation}
for all $\a\geq 1$. Here and in the following, unless specified otherwise, all constants depend only on $(M,g),\,n,k$. They are denoted $C$, and their explicit value can vary from line to line, sometimes even in the same line.
\subsection{Riemannian bubbles}
This subsection is devoted to the definition of Riemannian versions of the minimizing functions $\Bub$ for the critical Sobolev inequality defined in \eqref{def:bub}.
\begin{definition}
    We define the concentrated version of the bubble, of center $z\in M$ and with weight $\mu>0$, as 
    \begin{equation}\label{def:bubzm}
        \Bub_{\zm}(x) := \bpr{\frac{\mu}{\mu^2 + \pk\dg{z,x}^2}}^{\frac{n-2k}{2}},
    \end{equation}
    where $\pk$ is as in \eqref{def:bub}.
\end{definition}
In general, for any function $\psi \in \hSob(\R^n)$ we can define a corresponding concentrated version on $M$.
\begin{definition}\label{def:psizm}
    Let $\psi \in \hSob(\R^n)$, $z\in M$ and $\mu>0$, let also $\inj/2 < \varrho < \inj$ and $\chi_\varrho \in \Cct(\R^+)$ be a cut-off function such that $0\leq \chi_\varrho\leq 1$, $\chi_\varrho \equiv 1$ on $[0,\inj/2)$ and $\chi_\varrho \equiv 0$ on $(\varrho, + \infty)$. We define 
    \begin{equation}\label{def:conc}
        \psi_{\zm}(x) = \chi_\varrho(\dg{z,x})\mu^{-\frac{n-2k}{2}}\psi\big(\tfrac{1}{\mu}\exp_z^{-1}(x)\big) \qquad \forall~x\in M.
    \end{equation} 
\end{definition}
Let us fix once and for all $\inj/2 < \varrho < \inj$, so that when $\psi\in \hSob(\R^n)\cap C^\infty(\R^n)$, we have $\psi_{\zm}\in C^\infty(M)$. Note also that if $\psi \in \Cct(\R^n)$, there exists $\mu_0>0$ such that for all $\mu\leq\mu_0$, $\mu\abs{y} < \inj/2$ for all $y\in \R^n$ in the support of $\psi$.
\begin{lemma}\label{prop:psizmbded}
    Let $\psi\in \hSob(\R^n)$, there exists $\mu_0>0$ and a constant $C>0$ such that, for all $\mu\leq \mu_0$, $z\in M$, we have $\psi_{\zm} \in \Sob(M)$ with 
    \[  \Snorm{k}{\psi_{\zm}}\leq C\Hnorm{k,2}{\psi}.
        \]
\end{lemma} 
We refer to Appendix \ref{sec:intconc} for a proof of this Lemma. 
\begin{remark}\label{rem:intpsizmR}
    This still holds when the norm is only taken in a subset of $M$. For instance, for a fixed $R>0$ we obtain
    \[  \Snorm[(M\setminus \Bal{z}{R\mu})]{k}{\psi_{\zm}} \leq C \Hnorm[(\R^n \setminus \Bal{0}{R})]{k,2}{\psi},
        \]
    where $C>0$ is independent of $R$.
\end{remark}
\par Comparing \eqref{def:bubzm} with \eqref{eq:transfo}, $\Bub_{\zm}$ can be seen as a good candidate to solve 
\[  \Dg^k \Bub_{\zm} \simeq \Bub_{\zm}^{\deu-1} \qquad \text{in } M
    \]
up to lower-order terms coming from the geometry of the manifold. In equation \eqref{eq:critbub}, because of the lower-order terms \eqref{eq:lowerterm} with $\a \to \infty$, we will need to work with a modified version of the bubble in $M$.

\begin{definition}\label{def:fcth}
    Fix $\a\geq 1$, and let $\chi_\varrho$ be a cut-off function chosen as in definition \ref{def:psizm}. Let also $h : \R^n \to \R^+$ be a $C^{2k+1}(\R^n)$ function such that 
    \[  h(y) \equiv 1 \qquad \text{when } \abs{y} \leq 1,
        \]
    and such that there exists $C_l >0$ for $l=0,\ldots 2k+1$ and 
    \[  \abs{\nabla^l h(y)} \leq C_l e^{-\abs{y}/2} \qquad \text{when } \abs{y} \geq 1.
        \]
    Let $z\in M$, we define 
    \[  \Theta_\a(z,x) := \chi_\varrho(\dg{z,x})\, h(\sqa\exp_z^{-1}(x)) \qquad \forall~x\in M.
        \]
\end{definition}
It is then immediate to obtain that, for $l=1,\ldots 2k+1$,
\begin{equation}\label{eq:estdThet}
    \begin{aligned}
        \nabla_g^l \Theta_\a(z,x) &= 0 & &\text{when } \sqa\dg{z,x} \leq 1,\\
        \abs{\nabla_g^l \Theta_\a(z,x)} &\leq \a^{l/2} e^{-\sqa\dg{z,x}/2} & &\text{when } \sqa\dg{z,x} \geq 1.
    \end{aligned}
\end{equation}
\begin{definition}\label{def:defTBub}
    Let $\a\geq 1$ be such that $1/\sqa < \inj/2$, and let $z\in M$ and $\mu>0$, we write $\pa = (\zm)$. We define, for $x\in M$,
    \begin{equation}\label{def:TBub}
        \TBub(x) := \Theta_\a(z,x) \Bub_{\zm}(x).
    \end{equation}
\end{definition}
\begin{remark}
    This new rescaled version of $\Bub$ is modelled after the function 
    \begin{equation}\label{tmp:betbub}
        \chi_\varrho(\dg{z,x}) c_{n,k}^{-1}\dg{z,x}^{n-2k} \Gga(z,x) \Bub_{\zm}(x),
    \end{equation}
    where $\Gga$ is the Green's function for the operator $(\Dg+\a)^k$ in $M$, and $\frac{c_{n,k}}{\abs{x-y}^{n-2k}}$ is the Green's function for the poly-Laplacian $\Dg[\xi]^k$ on $\R^n$. We refer to \cite{Car24} for the construction and pointwise estimates of $\Gga$ and its derivatives, that are similar to \eqref{eq:estdThet}. Of course any compactly supported function $h$ such that $h\equiv 1$ in $\Bal{0}{1}$ would satisfy the requirements of definition \ref{def:fcth}. We choose here to work with a more general function, with only the minimum requirements to make the proof work. This is done to cover the cases where the bubble needs to be modified precisely as in the expression \eqref{tmp:betbub} (when $\alpha$ is fixed this was done for instance in \cite{Heb14,EsPiVe14}). In these cases, which will be treated in future work, the best possible decay would indeed be exponential, as showed by the pointwise estimates of $\Gga$ in \cite{Car24}.
\end{remark}
As we will see, $\TBub$ is a better candidate to solve the critical equation \eqref{eq:critbub}, when $\a \to \infty$ and $\a\mu^2 \to 0$, in a precise sense given by Proposition \ref{prop:estRa} below.
\begin{definition}
    Let $\a\geq 1$ and $\tau\leq 1$, we define the parameter set 
    \begin{equation}\label{def:param}
        \param := \{\pa = (\zm) \in M\times (0,+\infty) \such \a\mu^2< \tau\}.
        \end{equation}
\end{definition}
\begin{lemma}\label{prop:Thetto1}
    Let $(\tau_i)_i$, $(\a_i)_i$ be sequences of positive numbers such that $\tau_i \to 0$ and $1/\sqai < \inj/2$ for all $i$, and let $\pa_i = (\zm[i]) \in \param[i]$ for all $i$. Then 
    \[  \Theta_{\a_i}(z_i,\exp_{z_i}(\mu_i \cdot)) \to 1 \qquad \text{in } C^\infty_{loc}(\R^n), \text{ as } i\to \infty.
        \]
\end{lemma}
\begin{proof}
    Fix $K\sub \R^n$ a compact subset. Then, since $\a_i \mu_i^2 < \tau_i$ for all $i$, there is $i_0\in \N$ such that for all $i \geq i_0$, $\sqai \mu_i \abs{y} \leq 1$ for all $y \in K$. Thus, for all $i\geq i_0$, we have
    \[  \Theta_{\a_i}(z_i,\exp_{z_i}(\mu_i y)) = \chi_\varrho(\mu_i \abs{y}) h(\sqai\mu_i y) \equiv 1
        \]
    for all $y \in K$.
\end{proof}

\subsection{Almost solutions to the linearized equation}
The purpose of the next sections is to find approximated solutions to \eqref{eq:critbub} of the form $u_\a = \TBub + \phi_\a$, where $\phi_\a$ is small in comparison with $\TBub$ in a specific sense that will be explained below (see Proposition \ref{prop:uniqest}). It is then natural to spend some time studying the linearized version  around $\TBub$ of the critical equation \eqref{eq:critbub}, namely
\begin{equation}\label{eq:critlin}
    (\Dg+\a)^k u = (\deu-1)\TBub^{\deu-2} u \qquad \text{in } M.
\end{equation}
In this subsection, we define functions of the same form as \eqref{def:TBub}, but related to solutions to the linearized critical equation in $\R^n$.
\begin{definition}
    In the Euclidean setting, we consider the equation 
    \begin{equation}\label{eq:linbub}
        \Dg[\xi]^k u = (\deu-1)\Bub^{\deu-2}u \qquad \text{in }\R^n,
    \end{equation}
    for $n>2k$. We then define 
    \begin{equation}\label{def:ker}
        \Ker[] := \{v \in \hSob(\R^n) \such \Dg[\xi]^k v = (\deu-1)\Bub^{\deu-2} v \quad \text{in the weak sense}\},
    \end{equation}
    the set of solution to \eqref{eq:bub} linearized around $\Bub$ as given by \eqref{def:bub}.
\end{definition}
\begin{remark}
    If $v\in \Ker[]$, then 
    \[  \vprod[\hSob(\R^n)]{\Bub}{v} = \int_{\R^n} \vprod[\xi]{\Dg[\xi]^{k/2}\Bub}{\Dg[\xi]^{k/2}v} dy = (\deu-1)\int_{\R^n} \Bub^{\deu-1} v\, dy
        \]
    by definition of $\Ker[]$. Since $\Bub$ solves \eqref{eq:bub}, we also have
    \[  \int_{\R^n} \Bub^{\deu-1} v\, dy = \int_{\R^n} \Dg[\xi]^k \Bub v \, dy = \int_{\R^n} \vprod[\xi]{\Dg[\xi]^{k/2}\Bub}{\Dg[\xi]^{k/2}v} dy. 
        \]
    We conclude, for all $n>2k$, that $\vprod[\hSob(\R^n)]{\Bub}{v}=0$, so that $\Bub \in \Ker[]^\perp$ the orthogonal to $\Ker[]$ in $\hSob(\R^n)$. 
\end{remark}

It has been shown in \cite{BarWetWil03} that Equation \eqref{eq:linbub} possesses $n+1$ linearly independent solutions in $\hSob(\R^n)$. These correspond to the partial derivatives of $\Bub$ with respect to the parameters $\zm$ of the transformation \eqref{eq:transfo} under which \eqref{eq:bub} is invariant.
\begin{definition}
    We define 
    \begin{equation}\label{def:Z}
    \begin{aligned}
        \Ze{0}(y) &:= y \cdot \nabla \Bub(y) + \frac{n-2k}{2}\Bub(y),\\
        \Ze{j}(y) &:= \dsur{\Bub}{y_j}(y) & &\qquad j=1,\ldots n.
    \end{aligned}
    \end{equation}
    These form an orthogonal basis of $\Ker[]$ in $\hSob(\R^n)$. In particular, they are solutions of \eqref{eq:linbub} in $\R^n$. 
\end{definition}

\par We now define the rescaled versions of $\Ze{j}$ on $M$ by analogy, using $\TBub$ as defined in \eqref{def:TBub}.
\begin{definition}
    Let $\a\geq 1$ be such that $1/\sqa < \inj/2$, $\tau\leq 1$, and $\pa_0 = (z_0,\mu_0)\in \param$. We define
    \begin{equation}\label{def:tZ}
    \begin{aligned}
        \Zeda[\pa_0]{0} &:= \mu_0 \bpr{\dsur{\TBub}{\mu}}\at{z=z_0,\mu=\mu_0},\\
        \Zeda[\pa_0]{j} &:= \mu_0 \bpr{\dsur{\TBub}{z_j}}\at{z=z_0,\mu=\mu_0} & &\qquad j=1\ldots n,
    \end{aligned}
    \end{equation}
    where the partial derivative along the direction $z_j$ is defined in \eqref{def:dzj} below. 
\end{definition}
\begin{remark}
    The multiplication by $\mu_0$ ensures that $\Zed{j}$ is equal to first order to the same rescaling transformation of $\Ze{j}$ that $\TBub$ is for $\Bub$. Indeed, for all $x\in M$, we have
    \begin{align*}
        \Zed{0}(x) &= \Theta_\a(z,x) \big(\Ze{0}\big)_{\zm}(x)\\
        \Zed{j}(x) &= \Theta_\a(z,x) \big(\Ze{j}\big)_{\zm}(x) + \mu\dsur{}{z_j}\Theta_\a(z,x) \Bub_{\zm}(x) & &\qquad j=0,\ldots n,
    \end{align*}
    where $\big(\Ze{j}\big)_{\zm}$ is the rescaling defined in \eqref{def:conc}, for $\dg{z,x} < \inj$, 
    \begin{equation}\label{def:Zzm}
        \big(\Ze{j}\big)_{\zm}(x) := \mu^{-\frac{n-2k}{2}}\Ze{j}\big(\tfrac{1}{\mu}\exp_{z}^{-1}(x)\big).
    \end{equation}
\end{remark}
\begin{remark}
    The map $(\zm) \mapsto \TBub$ is differentiable in its variables $z\in M$, $\mu>0$, and the partial derivative along the variable $z$, near a center $z_0$, can be defined through the local charts. For $j=0,\ldots n$, we let $v_j\in\R^n$ be the $j$-th vector of the canonical basis of $\R^n$, then we define 
    \begin{align*}
        \gamma_j : (-\epsilon, \epsilon) & \longrightarrow M\\
            t & \longmapsto \exp_{z_0}(tv_j)
    \end{align*}
    for a small enough $\epsilon >0$. The partial derivative along $z_j$ is defined as
    \begin{equation}\label{def:dzj}
        \dsur{\TBub}{z_j}\at{\pa_0} := \Dsur{}{t}\big(\TBuba[(\gamma_j(t),\mu)]\big)\at{t = 0}  
    \end{equation}
\end{remark}
By analogy with $\Ker[]$, we also define the following.
\begin{definition}
    Let $\a\geq 1$ be such that $1/\sqa < \inj/2$, $\tau\leq 1$, and $\pa = (\zm) \in \param$. We define the set
    \begin{equation}\label{def:kerzm}
        \Ker := \spanned{\Zed{j} \such 0\leq j\leq n} \sub \Sob(M).
    \end{equation}
\end{definition}
As we will show below (see Corollary \ref{prop:Zedorth}), $\{\Zed{j}\such 0\leq j\leq n\}$ forms an almost orthogonal basis of $\Ker$.

\subsection{Properties of the rescaled functions}
In this subsection, we study the properties of the rescaled functions defined earlier. Since they have very similar pointwise behavior, we will write in the following $X$ to be either $\Bub$ or $\Ze{j}$ for some $j\in \{0,\ldots n\}$. Similarly, we write respectively $X_{\zma}$ either as $X_{\zma} =\TBub$ defined in \eqref{def:TBub}, or $X_{\zma} = \Zed{j}$ defined in \eqref{def:tZ}. In this subsection, we show that indeed $\TBub$ and $\Zed{j}$ are \emph{almost solutions} for the Equations \eqref{eq:critbub} and \eqref{eq:critlin}, respectively.

\begin{proposition}\label{prop:estdifX}
    Let $\a \geq 1$ be such that $1/\sqa < \inj/2$, $\tau\leq 1$, and let $\pa = (\zm)\in \param$. For $l=0,\ldots 2k$, there is a constant $C_l>0$ independent of $\a, \tau, \pa$ such that the following holds.
    \begin{itemize}
        \item For all $x\in M$ with $\sqa\dg{z,x} \leq 1$, 
        \begin{equation}\label{eq:estdifXsqa}
            \abs{\nabla_g^l X_{\zma}(x) - \mu^{-\frac{n-2k+2l}{2}}\big(\nabla_\xi^l X\big)\big(\tfrac{1}{\mu}\exp_z^{-1}(x)\big)}_\xi\leq C_l (\mu+\dg{z,x})^{2-l}\Bub_{\zm}(x).
        \end{equation}
        \item For all $x\in M$ with $\sqa\dg{z,x} \geq 1$,
        \begin{equation}\label{eq:estdifXexp}
            \abs{\nabla_g^l X_{\zma}(x)}_g \leq C_l \a^{l/2} \mu^{\frac{n-2k}{2}}\dg{z,x}^{2k-n} e^{-\sqa \dg{z,x}/2}.
        \end{equation}
    \end{itemize}
\end{proposition}
\begin{proof}
    Let $\inj/2 <\varrho <\inj$ be as defined in definition \ref{def:psizm}. Computing the quantities in local coordinates at $z\in M$, for a smooth $f\in C^\infty(M)$, we have for all $y \in \Bal{0}{\varrho}$, $l=2,\ldots 2k$,
    \begin{multline}\label{eq:nabexp}
        \big(\nabla_g^l f\big)_{i_1,\ldots i_l} (\exp_{z}(y)) = \big(\nabla_\xi^l (f\circ \exp_z) \big)_{i_1,\ldots i_l}(y)\\ + \bigO\bpr{\abs{y}\abs{\nabla_\xi^{l-1}(f\circ \exp_z)(y)}} + \bigO\bpr{\tsum_{m=1}^{l-2}\abs{\nabla_\xi^m (f\circ \exp_z)(y)}}. 
    \end{multline}
    We first consider the case $\dg{z,x} \leq 1/\sqa$. Then, since $\Theta_\a(z,x) \equiv 1$, we have
    \begin{align*}
        \nabla_g^l \TBub(x) &= \nabla_g^l \Bub_{\zm}(x) \\
        \nabla_g^l \Zed{j}(x) &= \nabla_g^l \big(\Ze{j}\big)_{\zm}(x),\qquad j=0,\ldots n
    \end{align*}
    for $l=0,\ldots 2k$.
    For all $y \in \Bal{0}{1/\sqa} \sub \R^n$, and $l=2,\ldots 2k$, we have 
    \begin{multline*}
        \nabla_g^l X_{\zma}(\exp_{z}(y)) = \mu^{-\frac{n-2k}{2}-l} \big(\nabla_\xi^l X\big)\bpr{\frac{y}{\mu}}\\ + \bigO\bpr{\abs{y}\mu^{-\frac{n-2k}{2}-l+1} \abs{\big(\nabla_\xi^{l-1} X\big)\bpr{\frac{y}{\mu}}}}\\ + \bigO\bpr{\tsum_{m=1}^{l-2} \mu^{-\frac{n-2k}{2}-m}\abs{\big(\nabla_\xi^m X\big)\bpr{\displaystyle\frac{y}{\mu}}}}.
    \end{multline*}
    Note that the derivatives of $\Bub$ and $\Ze{j}$ have a known behavior : Using their explicit expression, we obtain for all $y \in M$ and $l=0,\ldots 2k$,
    \begin{align}
        \label{eq:controlB} &(1+\abs{y})^{l} \abs{\nabla^l \Bub(y)} \leq C (1+\abs{y})^{2k-n}\\
        \label{eq:estdZ}    &(1+\abs{y})^l \abs{\nabla^l \Ze{j}(y)} \leq C (1+\abs{y})^{2k-n}
    \end{align}
    for some $C>0$.
    We then obtain, for $y \in \Bal{0}{1/\sqa}$,
    \begin{align*}
        \nabla_g^l X_{\zma}(\exp_z(y)) &= \mu^{-\frac{n-2k+2l}{2}} \big(\nabla_\xi^l X\big) \bpr{\frac{y}{\mu}} + \bigO\Big(\abs{y}(\mu+\abs{y})^{-l+1}\Bub_{\zm}(\exp_z(y))\Big)\\
            &\qquad\qquad +\bigO\bpr{\tsum_{m=1}^{l-2} (\mu+\abs{y})^{-m} \Bub_{\zm}(\exp_z(y))}\\
            &= \mu^{-\frac{n-2k+2l}{2}} \big(\nabla_\xi^l X\big)\bpr{\frac{y}{\mu}} + \bigO\Big((\mu+\abs{y})^{2-l}\Bub_{\zm}(\exp_z(y))\Big),
    \end{align*}
    where we used that $\abs{y} \leq (\mu+\abs{y}) \leq \frac{1+\tau}{\sqa} < \inj$ for all $\pa \in \param$. This proves \eqref{eq:estdifXsqa}.
    \par Let now $x\in M$ be such that $1/\sqa \leq \dg{z,x} \leq \varrho$, where $\varrho$ was chosen in definition \ref{def:psizm}, and write $y=\exp_z^{-1}(x)$, we compute
    \begin{equation}\label{tmp:forV}
    \begin{aligned}
        \abs{\nabla_g^l \TBub(x)}_g &\leq C\sum_{m=0}^l \abs{\nabla_g^{l-m} \Theta_\a(z,x)}_g\abs{\nabla_g^m \Bub_{\zm}(x)}_g\\
            &\leq C\Bigg[ \mu^{-\frac{n-2k}{2}}\Bub\big(\tfrac{y}{\mu}\big)\a^{l/2}e^{-\sqa\abs{u}/2}\\
            &\qquad\qquad +\sum_{m=1}^{l}\sum_{p=1}^m \mu^{-\frac{n-2k}{2}-p}\abs{\big(\nabla_\xi^p \Bub\big)\big(\tfrac{y}{\mu}\big)}\a^{\frac{l-m}{2}}e^{-\sqa\abs{y}/2}  \Bigg]
    \end{aligned}
    \end{equation}
    using \eqref{eq:estdThet}. As before, we use \eqref{eq:controlB} and \eqref{eq:estdZ}, together with the fact that $\abs{y} > \mu$ for all $y \in \Bal{0}{\varrho}\setminus\Bal{0}{1/\sqa}$, and obtain
    \begin{align*}  
        \abs{\nabla_g^l \TBub(\exp_z(y))}_g &\leq C\mu^{\frac{n-2k}{2}}\abs{y}^{2k-n} \a^{l/2}e^{-\sqa\abs{y}/2}\bsq{1+ \sum_{m=1}^l \sum_{p=1}^m \a^{-m/2}\abs{y}^{-p}}\\
            &\leq C \mu^{\frac{n-2k}{2}} \abs{y}^{2k-n}\a^{l/2} e^{-\sqa\abs{y}/2}
    \end{align*}
    for all $y \in \Bal{0}{\varrho}\setminus\Bal{0}{1/\sqa}$. Now for $\Zed{j}$, we have 
    \begin{multline*}  
        \abs{\nabla_g^l \Zed{j}(x)}_g \leq C\sum_{m=0}^l \abs{\nabla_g^{l-m} \Theta_\a(z,x)}\abs{\nabla_g^m \big(\Ze{j}\big)_{\zm}(x)}\\ + \sum_{m=0}^l \abs{\nabla_g^l \bpr{\Zdif}(x)}.
    \end{multline*}
    With the same arguments as in \eqref{tmp:forV}, using \eqref{eq:estdZ}, we obtain
    \[  \sum_{m=0}^l \abs{\nabla_g^{l-m} \Theta_\a(z,x)}\abs{\nabla_g^m \big(\Ze{j}\big)_{\zm}(x)} \leq C \mu^{\frac{n-2k}{2}}\dg{z,x}^{2k-n}\a^{l/2} e^{-\sqa \dg{z,x}/2}
        \]
    for all $x \in \Bal{z}{\varrho}\setminus \Bal{z}{1/\sqa}$. Now using \eqref{eq:estdPhi} in Lemma \ref{prop:Phito0} below, we have
    \begin{align*}
        \sum_{m=0}^l \abs{\nabla_g^l \bpr{\Zdif}(x)} \leq C\sqa\mu \mu^{\frac{n-2k}{2}}\dg{z,x}^{2k-n}\a^{l/2} e^{-\sqa \dg{z,x}/2}.
    \end{align*}
    Realizing that when $\dg{z,x}\geq \varrho$, $\abs{\nabla_g^l X_{\zma}} = 0$, we obtain \eqref{eq:estdifXexp} which concludes the proof of Proposition \ref{prop:estdifX}.
\end{proof}

\begin{proposition}\label{prop:Tconc}
    Let $(\tau_i)_i$ and $(\a_i)_i$ be sequences of positive numbers such that $1/\sqai \leq \inj/2$ for all $i$, $\tau_i\to 0$ as $i \to \infty$, and let $\pa_i = (\zm[i]) \in \param[i]$ for all $i$. Let $X$ denote either $\Bub$ of $\Ze{j}$ for some $j=0,\ldots n$, and $X_{\zma[i]}$ denote respectively either $\TBub[i]$ or $\Zed[i]{j}$, we have the following.
    \begin{enumerate}
        \item $\displaystyle \mu_i^{\frac{n-2k}{2}} X_{\zma[i]}(\exp_{z_i}(\mu_i \cdot)) \xto{i\to \infty} X$ in $C^\infty_{loc}(\R^n)$;
        \item For $l=0,\ldots k$, 
            \begin{equation}\label{eq:ii}
                \intM{\abs{\Dg^{l/2}X_{\zma[i]}}^{\deuk{k-l}}} \xto{i\to \infty} \int_{\R^n} \abs{\Dg[\xi]^{l/2}X}^{\deuk{k-l}} dy,
            \end{equation}
            where $\deuk{l} := \frac{2n}{n-2l}$;
        \item We have 
            \begin{equation}\label{eq:iv}
                X_{\zma[i]} \wto 0 \qquad \text{ in $\Sob(M)$, as $i\to \infty$},
            \end{equation}
            and $\Snorm{k}{X_{\zma[i]}} \to \Hnorm{k,2}{X}$;
        \item For any $\psi \in \hSob(\R^n)$, 
            \begin{equation}\label{eq:iii}
                \vprod[\Sob(M)]{X_{\zma[i]}}{\psi_{\zm[i]}} \xto{i\to \infty} \vprod[\hSob(\R^n)]{X}{\psi},
            \end{equation}
            where $\psi_{\zm[i]}$ is the concentrated version of $\psi$ defined in \eqref{def:conc}.
    \end{enumerate}
\end{proposition}
\begin{proof}
    \par\textbf{Proof of (1):} For $X_{\zma[i]} = \TBub[i]$, the conclusion follows immediately from the definition of $\TBub[i]$, and Lemma \ref{prop:Thetto1}. For $X_{\zma[i]} = \Zed[i]{j}$, $j=0,\ldots n$, as in Lemma \ref{prop:Thetto1}, for all compact set $K\sub \R^n$, there is $i_0>0$ such that for all $i\geq i_0$, $\mu_i\abs{y} \leq 1/\sqai$ for all $y\in K$. Therefore, 
    \[  \Zed[i]{j}(\exp_{z_i}(\mu_i y)) = \mu_i^{-\frac{n-2k}{2}}\Ze{j}(y) \qquad \forall~ y \in K,
        \]  
    and we conclude.
    \par\textbf{Proof of (2):} For all $i$, we have
    \begin{multline}\label{tmp:twoterm}
        \intM{\abs{\Dg^{l/2}X_{\zma[i]}}^{\deuk{k-l}}} = \intM[\Bal{z_i}{1/\sqai}]{\abs{\Dg^{l/2}X_{\zma[i]}}^{\deuk{k-l}}}\\ + \intM[\Bal{z_i}{\varrho}\setminus\Bal{z_i}{1/\sqai}]{\abs{\Dg^{l/2}X_{\zma[i]}}^{\deuk{k-l}}}.
    \end{multline}
    We compute the first term in the right-hand side of \eqref{tmp:twoterm} using \eqref{eq:estdifXsqa}, in local coordinates we have with a change of variables $y = \tfrac{1}{\mu_i}\exp_{z_i}^{-1}(x)$,
    \begin{multline}\label{tmp:Xdeukl}
        \intM[\Bal{z_i}{1/\sqai}]{\abs{\Dg^{l/2}X_{\zma[i]}}^{\deuk{k-l}}}\\ = \int_{\Bal{0}{\frac{1}{\sqai\mu_i}}} \abs{\Dg[\xi]^{l/2} X(y)}^{\deuk{k-l}}\big(1+\bigO(\mu_i^2 \abs{y}^2)\big) dy\\ + \bigO\bpr{\int_{\Bal{0}{\frac{1}{\sqai\mu_i}}} \mu_i^{\frac{4n}{n-2k+2l}} \bsq{(1+\abs{y})^{2-l}\Bub(y)}^{\deuk{k-l}}dy}
    \end{multline}
    We then observe that for all $y\in \R^n$, using \eqref{eq:controlB} and \eqref{eq:estdZ},
    \begin{align*}
        (1+\abs{y})^{2-l} \Bub(y) &\leq C (1+\abs{y})^{2k-n+2-l}\\
        \abs{y}^2 \abs{\Dg[\xi]^{l/2}X(y)}^{\deuk{k-l}} &\leq C \abs{y}^2 (1+\abs{y})^{-2n \frac{n-2k+l}{n-2k+2l}}, 
    \end{align*}
    so that, with straightforward computations, \eqref{tmp:Xdeukl} gives
    \begin{equation}\label{tmp:decompX}
        \intM[\Bal{z_i}{1/\sqai}]{\abs{\Dg^{l/2}X_{\zma[i]}}^{\deuk{k-l}}} = \int_{\R^n} \abs{\Dg[\xi]^{l/2} X(y)}^{\deuk{k-l}} dy + o(1)
    \end{equation}
    as $i \to \infty$, since $X \in \hSob(\R^n) \emb \hSob[l,\deuk{k-l}](\R^n)$ by Sobolev embedding. We compute the second term in the right-hand side of \eqref{tmp:twoterm} using \eqref{eq:estdifXexp}, and we obtain with a change of variables $y = \sqai \exp_{z_i}^{-1}(x)$,
    \begin{multline}\label{tmp:Tconcxx}
        \intM[\Bal{z_i}{\varrho}\setminus\Bal{z_i}{1/\sqai}]{\abs{\Dg^{l/2}X_{\zma[i]}}^{\deuk{k-l}}} \\
        \begin{aligned}    
            &\leq C \int_{\Bal{0}{\varrho\sqai}\setminus\Bal{0}{1}} \a_i^{-\frac{n}{2}}\a_i^{\frac{n}{2}} \mu_i^{\frac{n-2k}{2}\deuk{k-l}} \bsq{\abs{y}^{2k-n}e^{-\abs{y}/2}}^{\deuk{k-l}} dy\\
            &\leq C \mu_i^{\frac{n-2k}{2}\deuk{k-l}} = o(1) 
        \end{aligned}
    \end{multline}
    as $i\to \infty$. This concludes the proof of \eqref{eq:ii}. 
    \par\textbf{Proof of (3):} We use the continuous embedding $L^{\frac{2n}{n+2k}}(M) \emb \Sob[-k](M)$, so that we need to show
    \[  \intM{\abs{X_{\zma[i]}}^{\frac{2n}{n+2k}}} = o(1) \qquad \text{as } i \to \infty.
        \]
    We split the domain of the integral between $\Bal{z_i}{1/\sqai}$ and $\Bal{z_i}{\varrho}\setminus\Bal{z_i}{1/\sqai}$ as before. We have, by straightforward computations,
    \[
        \intM[\Bal{z_i}{1/\sqai}]{\abs{X_{\zma[i]}}^{\frac{2n}{n+2k}}} \leq C \int_{\Bal{0}{\frac{1}{\sqai\mu_i}}} \mu^{\frac{4kn}{n+2k}} (1+\abs{y})^{-2n\frac{n-2k}{n+2k}} dy = o(1).
    \]
    We also get
    \begin{multline*}  
        \intM[\Bal{z_i}{\varrho}\setminus\Bal{z_i}{1/\sqai}]{\abs{X_{\zma[i]}}^{\frac{2n}{n+2k}}}\\ \leq C (\sqai\mu_i)^{\frac{n(n-6k)}{n+2k}}\mu_i^{\frac{4kn}{n+2k}}\int_{\R^n\setminus\Bal{0}{1}} \bsq{\abs{y}^{2k-n}e^{-\abs{y}/2}}^{\frac{2n}{n+2k}} dy = o(1) 
    \end{multline*}
    as $i\to \infty$. By the compactness of the inclusion $\Sob(M) \emb \Sob[k-1](M)$, we then have $X_{\zma[i]} \to 0$ in $\Sob[k-1](M)$, and 
    \[  \Snorm{k}{X_{\zma[i]}}^2 = \Lnorm[(M)]{2}{\Dg^{k/2}X_{\zma[i]}}^2 + \Snorm{k-1}{X_{\zma[i]}}^2 = \Hnorm{k,2}{X}^2 + o(1) 
        \]
    as $i\to \infty$, using \eqref{eq:ii}. 
    \par\textbf{Proof of (4):} As showed before, $\Snorm{k-1}{X_{\zma[i]}} \to 0$, so that, since $\psi_{\zm[i]}$ is a bounded sequence in $\Sob(M)$ with Lemma \ref{prop:psizmbded}, we obtain
    \begin{equation}\label{eq:vprodphiZ}
        \vprod[\Sob(M)]{X_{\zma[i]}}{\psi_{\zm[i]}} = \intM{\vprod{\Dg^{k/2}X_{\zma[i]}}{\Dg^{k/2}\psi_{\zm[i]}}} + o(1).
    \end{equation}
    Fix any $R>0$, we split the domain of the integral between $\Bal{z_i}{R\mu_i}$ and $M\setminus\Bal{z_i}{R\mu_i}$. On the one hand, we define the metric $\Tg_i(y) := \exp_{z_i}^*g(\mu_i y)$ on $\Bal{0}{R}\sub \R^n$, and we use the fact that $\mu_i^{\frac{n-2k}{2}} X_{\zma[i]}(\exp_{z_i}(\mu_i \cdot)) \to X$, and $\Tg_i \to \xi$ in $C^\infty_{loc}(\R^n)$ to obtain 
    \begin{multline*}
        \lim_{i\to \infty} \intM[\Bal{z_i}{R\mu_i}]{\vprod{\Dg^{k/2}X_{\zma[i]}}{\Dg^{k/2}\psi_{\zm[i]}}}\\
        \begin{aligned}
            &= \int_{\Bal{0}{R}} \vprod[\xi]{\Dg[\xi]^{k/2}X}{\Dg[\xi]^{k/2}\psi}dy\\
            &= \int_{\R^n} \vprod[\xi]{\Dg[\xi]^{k/2}X}{\Dg[\xi]^{k/2}\psi}dy + \bigO\bpr{\Hnorm[(\R^n \setminus \Bal{0}{R})]{k,2}{\psi}},
        \end{aligned}
    \end{multline*}
    since $X,\psi \in \hSob(M)$. On the other hand, we have by Hölder's inequality, and since $(X_{\zma[i]})_i$ is bounded in $\Sob(M)$, that
    \[
        \intM[M\setminus\Bal{z_i}{R\mu_i}]{\vprod{\Dg^{k/2}X_{\zma[i]}}{\Dg^{k/2}\psi_{\zm[i]}}} \leq C \Hnorm[(\R^n \setminus \Bal{0}{R})]{k,2}{\psi}
    \]
    using Remark \ref{rem:intpsizmR}. Finally, for all $R>0$, we have showed that 
    \[  \lim_{i\to \infty} \vprod[\Sob(M)]{X_{\zma[i]}}{\psi_{\zm[i]}} = \int_{\R^n} \vprod[\xi]{\Dg[\xi]^{k/2}X}{\Dg[\xi]^{k/2}\psi} dy + \bigO\bpr{\Hnorm[(\R^n \setminus\Bal{0}{R})]{k,2}{\psi}}.
        \]
    Letting $R\to \infty$, since $\psi \in \hSob(\R^n)$, we conclude.
\end{proof}

\begin{corollary}\label{prop:Zedorth}
    Let $(\tau_i)_i$, $(\a_i)_i$, $(\pa_i)_i$ be sequences as in Proposition \ref{prop:Tconc}, then for $j,j' \in \{0,\ldots n\}$, we have
    \[  \lim_{i\to \infty}\vprod[\Sob(M)]{\Zed[i]{j}}{\Zed[i]{j'}} = \begin{cases}
        \Hnorm{k,2}{\Ze{j}}^2 & \text{if } j = j'\\
        0 & \text{otherwise}
    \end{cases}.
        \]
\end{corollary}
\begin{proof}
    By \eqref{eq:iv}, we have $\Zed[i]{j} \to 0$ and $\Zed[i]{j'} \to 0$ in $\Sob[k-1](M)$, so that 
    \[  \vprod[\Sob(M)]{\Zed[i]{j}}{\Zed[i]{j'}} = \intM{\vprod{\Dg^{k/2}\Zed[i]{j}}{\Dg^{k/2}\Zed[i]{j'}}} + o(1).
        \]
    Now using \eqref{tmp:Tconcxx} with $\deuk{k-k} = 2$, we also obtain 
    \[  \intM{\vprod{\Dg^{k/2}\Zed[i]{j}}{\Dg^{k/2}\Zed[i]{j'}}} = \intM[\Bal{z_i}{1/\sqai}]{\vprod{\Dg^{k/2}\Zed[i]{j}}{\Dg^{k/2}\Zed[i]{j'}}} + o(1).
        \]
    It follows from the same computation as in \eqref{tmp:decompX} that
    \begin{multline*}
        \intM[\Bal{z_i}{1/\sqai}]{\vprod{\Dg^{k/2}\Zed[i]{j}}{\Dg^{k/2}\Zed[i]{j'}}}\\ = \int_{\Bal{0}{\frac{1}{\sqai\mu_i}}} \vprod[\xi]{\Dg[\xi]^{k/2}\Ze{j}}{\Dg[\xi]^{k/2}\Ze{j'}} dy + o(1)
    \end{multline*}
    as $i\to \infty$. Since $\{\Ze{j}\such 0\leq j\leq n\}$, form an orthogonal basis of $\Ker[]$, we thus get
    \[  \intM[\Bal{z_i}{1/\sqai}]{\vprod{\Dg^{k/2}\Zed[i]{j}}{\Dg^{k/2}\Zed[i]{j'}}} = \delta_{j\,j'}\Hnorm{k,2}{\Ze{j}}^2 + o(1)
        \]
    as $i\to \infty$, which gives the conclusion.
\end{proof}

The following result quantifies how far $\TBub$ and $\Zed{j}$ are, respectively, from solving \eqref{eq:critbub} and \eqref{eq:critlin}.
\begin{proposition}\label{prop:estRa}
    Write for $x\in M$,
    \begin{align*}
        \err{B}(x) &:= (\Dg+\a)^k \TBub(x) - \TBub^{\deu-1}(x)\\
        \err{\Ze{j}}(x) &:= (\Dg+\a)^k \Zed{j} - (\deu-1) \TBub^{\deu-1}(x) \Zed{j}(x) & &\qquad j=0,\ldots n.
    \end{align*}
    There exists $C>0$ such that, for all $\a\geq 1$ such that $1/\sqa < \inj/2$, all $\tau\leq 1$, and $\pa = (\zm) \in \param$, we have the following:
    \begin{itemize}
        \item For all $x\in M$ with $\sqa\dg{z,x}\leq 1$,
            \[  \abs{\err{X}(x)} \leq C\a(\mu+\dg{z,x})^{2-2k}\Bub_{\zm}(x).
                \]
        \item For all $x\in M$ with $\sqa\dg{z,x} \geq 1$,
            \[  \abs{\err{X}(x)}\leq C\a^{k} \mu^{\frac{n-2k}{2}}\dg{z,x}^{2k-n} e^{-\sqa\dg{z,x}/2}.
                \]
    \end{itemize}
    Moreover, 
    \begin{equation}\label{eq:RainH-k}
        \Snorm{-k}{\err{X}} \leq C (\sqa \mu)^{1/2}. 
    \end{equation} 
\end{proposition}
\begin{proof}
    Let $X$ be either $B$ or $\Ze{j}$ for some $j\in\{0,\ldots n \}$, and $X_{\zma}$ be respectively either $X_{\zma} = \TBub$ or $X_{\zma} = \Zed{j}$. Assume first that $\sqa\dg{z,x} \leq 1$, then with \eqref{eq:estdifXsqa}, we have for all $y\in \Bal{0}{1/\sqa}\sub \R^n$,
    \begin{multline}\label{tmp:Rax}
        (\Dg+\a)^k X_{\zma}(\exp_z(y)) = \mu^{-\frac{n+2k}{2}} \big(\Dg[\xi]^k X\big)\big(\tfrac{y}{\mu}\big) + \bigO\bpr{(\mu+\abs{y})^{2-2k}\Bub_{\zm}(\exp_z(y))}\\
            +\sum_{l=0}^{k} \tbinom{k}{l} \a^{k-l} \Big[\bigO\bpr{\mu^{-\frac{n-2k+4l}{2}} \abs{\nabla_\xi^{2l}X}\big(\tfrac{y}{\mu}\big)} + \bigO\bpr{(\mu+\abs{y})^{2-2l}\Bub_{\zm}(\exp_z(y))}\Big].
    \end{multline}
    Now observe that
    \begin{equation*}
    \begin{aligned}
        \mu^{-\frac{n+2k}{2}}\big(\Dg[\xi]^{k} \Bub\big)\big(\tfrac{y}{\mu}\big) &= \bpr{\mu^{-\frac{n-2k}{2}} \Bub\big(\tfrac{y}{\mu}\big)}^{\deu-1} = \TBub^{\deu-1}(\exp_z(y))\\
        \mu^{-\frac{n+2k}{2}}\big(\Dg[\xi]^{k} \Ze{j}\big)\big(\tfrac{y}{\mu}\big) &= \mu^{-\frac{n+2k}{2}} \Bub^{\deu-2}\big(\tfrac{y}{\mu}\big)\Ze{j}\big(\tfrac{y}{\mu}\big) = \TBub^{\deu-2}(\exp_z(y))\Zed{j}(\exp_z(y))\\
    \end{aligned}
    \end{equation*}
    for all $y\in \Bal{0}{1/\sqa}$, by definition of $\TBub,\Zed{j}$. With \eqref{eq:controlB} and \eqref{eq:estdZ}, and using that $\a(\mu+\abs{y})^2 \leq \tau+1$ for all $y\in\Bal{0}{1/\sqa}$, \eqref{tmp:Rax} becomes
    \[  \abs{\err{X}(\exp_z(y))} \leq C \a(\mu+\abs{y})^{2-2k} \Bub_{\zm}(\exp_z(y)).
        \]
    In the case $\sqa\dg{z,x} \geq 1$, we use \eqref{eq:estdifXexp}, and we obtain
    \begin{align*}  
        \abs{(\Dg+\a)^k X_{\zma}(x)} &\leq C \sum_{l=0}^{k} \a^{k-l}\a^l \mu^{\frac{n-2k}{2}} \dg{z,x}^{2k-n}e^{-\sqa\dg{z,x}/2}\\
            &\leq C\a^k \mu^{\frac{n-2k}{2}}\dg{z,x}^{2k-n}e^{-\sqa\dg{z,x}/2}.
    \end{align*}
    It remains only to prove \eqref{eq:RainH-k}, this follows from Lemma \ref{prop:estintRX} below with $\sigma=1$.
\end{proof}

\begin{corollary}~
    \begin{enumerate}
        \item There exists $C>0$ such that, for all $\a\geq 1$ such that $1/\sqa < \inj/2$, all $\tau\leq 1$, and $\pa = (\zm)\in \param$,
            \begin{equation}\label{eq:bdedDgaZ}
                \Snorm{-k}{(\Dg+\a)^k \Zed{j}} \leq C \qquad j=0,\ldots n;
            \end{equation}
        \item If $(\a_i)_i$, $(\tau_i)_i$, $(\pa_i)_i$ are sequences taken as in Proposition \ref{prop:Tconc}, then for $j,j' \in \{0,\ldots n\}$,
            \begin{equation}\label{eq:DgaZZ}
                \dprod{(\Dg+\a_i)^k\Zed[i]{j}}{\Zed[i]{j'}}_{\Sob[-k],\Sob} \xto{i\to \infty} \delta_{j\,j'}\Hnorm{k,2}{\Ze{j}}^2.
            \end{equation}
    \end{enumerate}
\end{corollary}
\begin{proof}
    The first claim is a direct consequence of Proposition \ref{prop:estRa} and the fact that $\Zed{j},\TBub$ are uniformly bounded in $L^\deu(M)$ by \eqref{eq:ii}. Indeed, using the continuous embedding $L^{\frac{2n}{n+2k}}(M) \emb \Sob[-k](M)$ and Hölder's inequality, we have
    \begin{align*}  
        \Snorm{-k}{\TBub^{\deu-2}\Zed{j}} &\leq C \Lnorm[(M)]{\frac{2n}{n+2k}}{\TBub^{\deu-2}\Zed{j}} \leq C \Lnorm[(M)]{\deu}{\Zed{j}} \Lnorm[(M)]{\deu}{\TBub}^{\frac{4k}{n-2k}}\\
            &\leq C \Hnorm{k,2}{\Ze{j}}\Hnorm{k,2}{\Bub}^{\frac{4k}{n-2k}} \leq C,
    \end{align*}
    by \eqref{eq:ii}. Moreover, with \eqref{eq:RainH-k}, we have 
    \[  \Snorm{-k}{(\Dg+\a)^k \Zed{j}} \leq \Snorm{-k}{\TBub^{\deu-2}\Zed{j}} +o(1) \leq C.
        \]
    For the second claim, test $\err[i]{\Ze{j}}\in \Sob[-k](M)$ against $\Zed[i]{j'} \in \Sob(M)$, we have 
    \[
        \abs{\dprod{(\Dg+\a_i)^k \Zed[i]{j}-(\deu-1)\TBub[i]^{\deu-2}\Zed[i]{j}}{\Zed[i]{j'}}_{\Sob[-k],\Sob}} = o(1)
    \] 
    as $i\to \infty$ by \eqref{eq:RainH-k}. Therefore,
    \begin{multline*}
        \dprod{(\Dg+\a_i)^k\Zed[i]{j}}{\Zed[i]{j'}}_{\Sob[-k],\Sob} = \intM{(\deu-1)\TBub[i]^{\deu-2}\Zed[i]{j}\Zed[i]{j'}} + o(1)\\
            = (\deu-1) \intM[\Bal{z_i}{1/\sqai}]{\TBub[i]^{\deu-2}\Zed[i]{j}\Zed[i]{j'}} \\
            + (\deu-1) \intM[\Bal{z_i}{\varrho}\setminus\Bal{z_i}{1/\sqai}]{\TBub[i]^{\deu-2}\Zed[i]{j}\Zed[i]{j'}} + o(1). 
    \end{multline*}
    Now as in \eqref{tmp:decompX}, 
    \begin{align*}
        \intM[\Bal{z_i}{1/\sqai}]{\TBub[i]^{\deu-2}\Zed[i]{j}\Zed[i]{j'}} &= \int_{\Bal{0}{\frac{1}{\sqai\mu_i}}} \Bub^{\deu-2}\Ze{j}\Ze{j'}\big(1+\bigO(\mu_i^2 \abs{y})\big) dy\\
            &=\int_{\R^n} \Bub^{\deu-2}\Ze{j}\Ze{j'} dy + o(1)
    \end{align*}
    as $i\to \infty$. We also have 
    \begin{multline*}
        \abs{\intM[\Bal{z_i}{\varrho}\setminus\Bal{z_i}{1/\sqai}]{\TBub[i]^{\deu-2}\Zed[i]{j}\Zed[i]{j'}}}\\ \leq \bpr{\intM[\Bal{z_i}{\varrho}\setminus\Bal{z_i}{1/\sqai}]{\TBub[i]^\deu}}^{\frac{2k}{n}} \Lnorm[(M)]{\deu}{\Zed[i]{j}}\Lnorm[(M)]{\deu}{\Zed[i]{j'}} = o(1)
    \end{multline*}
    by \eqref{tmp:Tconcxx}. Thus, we obtained
    \begin{align*}
        \lim_{i\to \infty}\dprod{(\Dg+\a_i)^k\Zed[i]{j}}{\Zed[i]{j'}}_{\Sob[-k],\Sob} &= \int_{\R^n}(\deu-1)\Bub^{\deu-2}\Ze{j}\Ze{j'}dy \\
            &= \delta_{j\,j'} \Hnorm{k,2}{\Ze{j}}^2,
    \end{align*}
    since $\Ze{j}$ solves \eqref{eq:linbub} in $\R^n$.
\end{proof}

\subsection{Finding the best bubble}
The purpose of this subsection is to show that when $u\in\Sob(M)$ is close to the set of all possible bubbles $\{\TBub \such \pa \in \param\}$, there is some choice of parameters $\bpa$ which is optimal, and the corresponding bubble satisfies $u -\TBuba[\bpa] \in \Ker[\a,\bpa]^\perp$. The following are standard considerations, we refer to \cite{BahCor88} for analogous results. 
\begin{lemma}\label{prop:uniqbub}
    Let $(\a_i)_i$, $(\tau_i)_i$ be sequences of positive numbers such that $\a_i \to \infty$, $\tau_i\to 0$ as $i\to \infty$, and let $\pa_i = (\zm[i]), \Tpa_i = (\Tz_i,\Tmu_i) \in \param[i]$, for all $i$, be two sequences such that 
    \[  \Snorm{k}{\TBub[i]-\TB{\a_i,\Tpa_i}} \xto{i\to \infty} 0.
        \]  
    Then, we have 
    \[  \begin{aligned}
        &\frac{\mu_i}{\Tmu_i} \to 1\\
        &\frac{\dg{z_i,\Tz_i}^2}{\mu_i\Tmu_i}\to 0
    \end{aligned} \qquad \text{as } i\to \infty.
        \]
\end{lemma}
\begin{proof}
    We first claim that there exists a constant $C>0$ such that 
    \begin{equation}\label{tmp:mutmubdd}
        \frac{\Tmu_i}{\mu_i} + \frac{\mu_i}{\Tmu_i} + \frac{\dg{\Tz_i,z_i}^2}{\mu_i\Tmu_i} \leq C.
    \end{equation}
    By contradiction, suppose that $\frac{\Tmu_i}{\mu_i} + \frac{\mu_i}{\Tmu_i} + \frac{\dg{\Tz_i,z_i}^2}{\mu_i\Tmu_i} \to \infty$, then observe that 
    \begin{multline*}
        \intM{\vprod{\Dg^{k/2}\TBub[i]}{\Dg^{k/2}\TB{\a_i,\Tpa_i}}}\\ = \intM[\Bal{z_i}{1/\sqai}\cap \Bal{\Tz_i}{1/\sqai}]{\vprod{\Dg^{k/2}\TBub[i]}{\Dg^{k/2}\TB{\a_i,\Tpa_i}}} + o(1) 
    \end{multline*}
    using \eqref{tmp:Tconcxx}. Moreover, 
    \begin{multline*}  
        \intM[\Bal{z_i}{1/\sqai}\cap \Bal{\Tz_i}{1/\sqai}]{\vprod{\Dg^{k/2}\TBub[i]}{\Dg^{k/2}\TB{\a_i,\Tpa_i}}}\\ = \intM[\Bal{z_i}{1/\sqai}\cap \Bal{\Tz_i}{1/\sqai}]{\vprod{\Dg^{k/2}\Bub_{\zm[i]}}{\Dg^{k/2}\Bub_{\Tz_i,\Tmu_i}}} = o(1)
        \end{multline*}  
    as $i\to \infty$, by straightforward computations and with the contradiction assumption. With \eqref{eq:iv}, we then have
    \[  \vprod[\Sob(M)]{\TBub[i]}{\TB{\a_i,\Tpa_i}} = \intM{\vprod{\Dg^{k/2}\TBub[i]}{\Dg^{k/2}\TB{\a_i,\Tpa_i}}} + o(1) = o(1),
        \]
    and thus we compute 
    \begin{align*}  
        \vprod[\Sob(M)]{\TBub[i]}{\TBub[i]-\TB{\a_i,\Tpa_i}} &= \Snorm{k}{\TBub[i]}^2 - \vprod[\Sob(M)]{\TBub[i]}{\TB{\a_i,\Tpa_i}}\\
            &= K_0^{-\frac{n}{k}} + o(1)
    \end{align*}
    using \eqref{eq:ii} and since $\Hnorm{k,2}{\Bub} = K_0^{-\frac{n}{2k}}$. We also obtain
    \[  \vprod[\Sob(M)]{\TBub[i]}{\TBub[i]-\TB{\a_i,\Tpa_i}} \leq \Snorm{k}{\TBub[i]} \Snorm{k}{\TBub[i]-\TB{\a_i,\Tpa_i}} = o(1)
        \]
    by assumption, and we reach a contradiction. This proves \eqref{tmp:mutmubdd}, and thus as a direct consequence there exists $\lambda >0$ and $y_\infty \in \R^n$ such that $\frac{\mu_i}{\Tmu_i} \to \lambda$ and $\tfrac{1}{\mu_i}\exp_{z_i}^{-1}(\Tz_i) \to y_\infty$, up to a subsequence. 
    \par By Sobolev's inequality, we have 
    \[  \Lnorm[(M)]{\deu}{\TBub[i]-\TB{\a_i,\Tpa_i}} = o(1) \qquad \text{as } i\to \infty.
        \]
    Let $R>0$, we compute
    \begin{multline*}
        o(1)= \intM[\Bal{z_i}{R\mu_i}]{\abs{\TBub[i]- \TB{\a_i,\Tpa_i}}^\deu} \\
            \geq \int_{\Bal{0}{R}} \abs{\Bub(y) - \lambda^{\frac{n-2k}{2}}\Bub(\lambda (y-y_\infty))}^{\deu} dy + o(1).
    \end{multline*}
    This can only hold if $\lambda = 1$ and $y_\infty = 0$.
\end{proof}
We now prove that the optimal bubble can be attained, following the proof of \cite[Proposition 7]{BahCor88}.
\begin{definition}
    Let $\a\geq 1$ be such that $1/\sqa<\inj/2$, and $\tau\leq 1$, we define the set of all possible rescaled bubbles,
    \begin{equation}\label{def:Bubmani}
        \mani := \{\TBub \such \pa = (\zm) \in \param\}.
    \end{equation}
\end{definition}
\begin{lemma}\label{prop:minprob}
    There exists $\epsilon_0 >0$, $\a_0\geq 1$ and $\tau_0 >0$ such that, for all $\a\geq \a_0$, $\tau\leq \tau_0$ and $\epsilon \leq \epsilon_0$, for any $u\in \Sob(M)$ such that $\dist[\Sob(M)]{u, \mani}<\epsilon$, the problem
    \begin{equation}\label{eq:minprob}
        \text{minimize}~ \Snorm{k}{u - \TBub} \quad \text{for } \pa \in \parama[4\tau] 
    \end{equation}
    admits a solution $\bpa = (\bz,\bmu) \in \parama[2\tau]$.
\end{lemma}
\begin{proof}
    Let $\epsilon\leq \epsilon_0$, $\tau\leq \tau_0$ to be fixed later, and let $u\in \Sob(M)$ be such that $\dist[\Sob(M)]{u,\mani} < \epsilon$. Consider a minimizing sequence $(\pa_i)_i$, $\pa_i = (\zm[i]) \in \parama[4\tau]$, for \eqref{eq:minprob}. Since $M$ is compact, up to a subsequence we have $z_i \to \bz \in M$, and since $\mu_i^2 < 4\tau/\a$, up to a subsequence $\mu_i \to \bmu \in [0,2\sqrt{\tau}/\sqa]$. We first show that $\bmu \neq 0$. 
    \par Assume by contradiction that $\mu_i \to 0$, then by \eqref{eq:iv}, we have $\TBuba[\pa_i] \wto 0$ in $\Sob(M)$. Hence, by lower semi-continuity of the norm, we have 
    \[  \Snorm{k}{u} \leq \liminf_{i\to \infty} \Snorm{k}{u - \TBuba[\pa_i]} \leq \epsilon_0.
        \]
    By the triangle inequality, we also obtain for all $i$ that 
    \[  \Snorm{k}{u} \geq \Snorm{k}{\TBuba[\pa_i]} - \Snorm{k}{u-\TBuba[\pa_i]},
        \]
    so that 
    \begin{equation}\label{tmp:contraV}
        \Snorm{k}{\TBuba[\pa_i]} \leq 2\epsilon_0 \qquad \forall~i \geq 0.
    \end{equation}
    Proposition \ref{prop:Tconc} gives moreover that $\Snorm{k}{\TBuba[\pa_i]} \to \Hnorm{k,2}{\Bub}= K_0^{-\frac{n}{2k}}$, which is a contradiction with \eqref{tmp:contraV} choosing $\epsilon_0$ small enough. 
    \par We showed that for all $u\in \Sob(M)$ such that $\dist[\Sob(M)]{u,\mani} < \epsilon$, there exists $\bz \in M$ and $\bmu >0$ that minimize the quantity $\Snorm{k}{u-\TBub}$. We now claim that $\a\bmu^2 < 2\tau$ when $\epsilon_0$ and $\tau_0$ are small enough, and $\a_0$ big enough. By contradiction, let $(\tau_i)_i$, $(\epsilon_i)_i$, $(\a_i)_i$ be sequences such that $\tau_i \to 0$, $\epsilon_i\to 0$ and $\a_i \to \infty$, assume that there exists for each $i$ : A parameter $\pa_i = (\zm[i])$ with $z_i \in M$, $\mu_i>0$ and $\a_i \mu_i^2 < \tau_i$, and some $\bpa_i = (\bz_i,\bmu_i)$ with $\bz_i \in M$, $\bmu_i >0$, such that $\TB{\a_i,\bpa_i}$ minimizes \eqref{eq:minprob} and $2\tau_i \leq \a_i \bmu_i^2 \leq 4\tau_i$, satisfying
    \[  \Snorm{k}{\TBub[i]-\TB{\a_i,\bpa_i}} < 2\epsilon_i.
        \]
    By Lemma \ref{prop:uniqbub}, we see that 
    \begin{align}\label{tmp:dot}
        \frac{\bmu_i}{\mu_i} &= 1+o(1) & &\text{and} & \frac{\dg{z_i,\bz_i}^2}{\mu_i\bmu_i} &= o(1),
    \end{align}
    as $i\to \infty$. Observe now that $\bmu_i \geq \frac{\sqrt{2\tau_i}}{\sqai}$, and $\mu_i < \frac{\sqrt{\tau_i}}{\sqai}$, so that for all $i$, $\frac{\bmu_i}{\mu_i} > \sqrt{2}$. For $i$ large enough, this is incompatible with \eqref{tmp:dot}, which concludes the proof.
\end{proof}
\begin{lemma}\label{prop:phiainKa}
    Let $\a_0\geq 1$, $\epsilon_0>0$ and $\tau_0>0$ be given by Lemma \ref{prop:minprob}, and fix $\a\geq \a_0$, $\tau\leq \tau_0$ and $\epsilon\leq \epsilon_0$. Let $u \in \Sob(M)$ be such that $\dist[\Sob(M)]{u,\mani}< \epsilon$, and let $\bpa \in \parama[2\tau]$ be a minimizer of \eqref{eq:minprob}. Then, 
    \[  u-\TBuba[\bpa] \in \Kera[\bpa]^\perp,
        \]
    where $\Kera[\bpa]$ is defined in \eqref{def:kerzm}.
\end{lemma}
\begin{proof}
    Fix $\bpa = (\bz,\bmu)$ be a minimizer for the problem \eqref{eq:minprob} obtained by Lemma \ref{prop:minprob}, the map $(\zm)\mapsto \TBub$ defined in \eqref{def:TBub} is differentiable in its variables $z\in M$, $\mu>0$ in a neighborhood of $(\bz, \bmu)$. Since $\bpa\in M\times (0,+\infty)$ is a minimizer of \eqref{eq:minprob}, that is also an interior point of $\parama[4\tau]$, we have 
    \begin{align*}
        \dsur{}{\mu} \Snorm{k}{u-\TBub}^2\at{\bpa = (\bz,\bmu)} &= 0\\
        \dsur{}{z_j} \Snorm{k}{u-\TBub}^2\at{\bpa = (\bz,\bmu)} &= 0 & &\quad j = 1,\ldots n,
    \end{align*}
    where the partial derivative along the variable $z_j$ is as defined in \eqref{def:dzj}.
    Writing the variable $\mu$ as $z_{0}$ for simplicity, we have 
    \begin{align*}
        0 =\dsur{}{z_j} \Snorm{k}{u-\TBub}^2\at{\bpa} &= \dsur{}{z_j} \vprod[\Sob(M)]{u-\TBub}{u-\TBub}\at{\bpa}\\
            &= 2\vprod[\Sob(M)]{u-\TBuba[\bpa]}{\dsur{}{z_j}\TBub\at{\bpa}}
    \end{align*}
    for $j=0,\ldots n$. By definition of $\Zeda[\bpa]{j}$ in \eqref{def:tZ} and $\Kera[\bpa]$ in \eqref{def:kerzm}, we obtain the result.
\end{proof}

\section{Constructive approach in $H^k(M)$}\label{sec:Hktheory}
\subsection{Uniform invertibility Theorem}\label{sec:unifinv}
In this subsection, we prove existence and uniqueness of solutions to the linearization of equation \eqref{eq:critbub} around $\TBub$, for $\pa \in \param$, 
\[  (\Dg+\a)^k u - (\deu-1)\TBub^{\deu -2} u= f \qquad \text{in $M$},
    \] 
up to terms that belong to $\Ker$, as defined in \eqref{def:kerzm}. We prove that for such a solution, $\Snorm{k}{u}$ is controlled by $\Snorm{-k}{f}$ uniformly with respect to $\a$ and $\tau$ as $\a \to \infty, \tau \to 0$. We write in the following $\proKe$ the projection onto the set $\Ker^\perp$ in $\Sob(M)$.

\begin{definition}
    Let $\a\geq 1$ be such that $1/\sqa < \inj/2$, $\tau\leq 1$, and $\pa = (\zm) \in \param$ as defined in \eqref{def:param}. We define the linear operator $\Lia : \Ker^\perp \to \Ker^\perp$ as
    \begin{equation}\label{def:Lzm}
        \Lia \phi := \proKe\bsq{\phi - (\Dg+\a)^{-k}\bpr{(\deu-1)\TBub^{\deu-2}\phi}}.
    \end{equation}
\end{definition}

This subsection is devoted to the proof of the following result.
\begin{proposition}\label{prop:lininv}
    There exists $\a_0\geq 1$, $\tau_0 >0$, and $C_0 >0$ such that, for all $\a \geq \a_0$, $\tau\leq \tau_0$, and $\pa =(\zm)\in \param$, $\Lia$ is invertible and the following holds:
    \begin{equation}\label{eq:lininv}
        \frac{1}{C_0} \Snorm{k}{\phi} \leq \Snorm{k}{\Lia\phi} \leq C_0 \Snorm{k}{\phi} \qquad \forall~\phi\in \Ker^\perp.
        \end{equation}
\end{proposition}
The proof of Proposition \ref{prop:lininv} is split in two parts: We first show the right-hand side of the inequality, i.e. that $\Lia$ is a bounded operator, uniformly with respect to $\a,\tau$, and $\pa \in \param$; In a second step, we prove the uniform invertibility of $\Lia$.
\begin{proof}[Proof of Proposition \ref{prop:lininv}, part 1: $\Lia$ is bounded]
    First observe that since $(\Dg+\a)^k$ is coercive, with \eqref{eq:Pacoerc}, we have for all $f \in \Sob[-k](M)$
    \begin{equation}\label{eq:Dgabded}
        \Snorm{k}{(\Dg+\a)^{-k}f} \leq \Snorm{-k}{f}.
    \end{equation} 
    With that in mind, let $\phi \in \Ker^\perp$, $\a \geq \a_0$, $\tau \leq \tau_0$ to be chosen later, and $\pa \in \param$. We have 
    \begin{align*}  
        \Snorm{k}{\Lia\phi} &\leq \Snorm{k}{\phi - (\Dg+\a)^{-k}\bpr{(\deu-1)\TBub^{\deu-2}\phi}}\\
            &\leq \Snorm{k}{\phi} + (\deu-1) \Snorm{-k}{\TBub^{\deu-2}\phi}
        \end{align*} 
    since $\proKe$ is a projection. Use the continuous embedding of the space $\Sob[-k](M)$ into $L^{\frac{2n}{n+2k}}(M)$, by Hölder's inequality we obtain
    \[  \Snorm{-k}{\TBub^{\deu-2}\phi} \leq C \Lnorm[(M)]{\frac{2n}{n+2k}}{\TBub^{\deu-2}\phi} \leq C \Lnorm[(M)]{\deu}{\phi} \bpr{\intM{\TBub^{\deu}}}^{\frac{2k}{n}}.
        \]
    Thanks to \eqref{eq:ii}, we then conclude that for a big enough $\a_0$, small enough $\tau_0$, there exists $C>0$ such that for all $\a \geq \a_0$, $\tau\leq \tau_0$, $\pa \in \param$,
    \begin{equation*}
        \Snorm{k}{\Lia\phi} \leq \Snorm{k}{\phi} + C \Lnorm[(M)]{\deu}{\phi} \leq C \Snorm{k}{\phi}.
    \end{equation*} 
    \end{proof}
The second part of the proof of Proposition \ref{prop:lininv} is analog to \cite[Proposition 3.1]{Pre24}, but in the polyharmonic setting. The presence of the diverging coefficient $\a$ needs to be taken care of in the computations, this explains why we use the modified bubble $\TBub$. 
\begin{proof}[Proof of Proposition \ref{prop:lininv}, part 2: Uniform invertibility]
    We prove the left-hand side of \eqref{eq:lininv} by contradiction. Suppose there exists sequences $(\a_i)_i, (\tau_i)_i$ such that $\a_i \to \infty$, $\tau_i \to 0$ as $i\to \infty$, $\pa_i = (\zm[i]) \in \param[i]$ for all $i$, and a sequence $(\phi_i)_i$ in $\Keri^\perp$ such that
    \begin{equation}\label{eq:normphi}
        \Snorm{k}{\phi_i} = 1 \qquad \forall~i\geq 0
    \end{equation} 
    and $\Lia[i]\phi_i\to 0$ in $\Sob(M)$ as $i\to \infty$. By definition of $\Lia[i]$ in \eqref{def:Lzm}, for all $i$ there exist real numbers $\lambda_i^j \in \R$, $j=0,\ldots n$, such that
    \begin{equation}\label{eq:forphi}
        (\Dg+\a_i)^k \phi_i - (\deu-1)\TBub[i]^{\deu-2}\phi_i = o(1) + \sum_{j=0}^n \lambda_i^j (\Dg+\a_i)^k \Zed[i]{j},
    \end{equation}
    where $o(1) \to 0$ in $\Sob[-k](M)$ as $i\to \infty$. Since $\phi_i \in \Keri^\perp$, we have 
    \begin{equation}\label{eq:prodphiZ}
        \vprod[\Sob(M)]{\phi_i}{\Zed[i]{j}} = 0 \qquad \forall~i, \quad j=0,\ldots n.
    \end{equation}
    Moreover, with \eqref{eq:normphi}, there exists $\phi_0 \in \Sob(M)$ such that, up to a subsequence, $\phi_i \wto \phi_0$ in $\Sob(M)$. Finally, since $\pa_i \in \param[i]$ with $\a_i \to \infty$ and $\tau_i \to 0$, by \eqref{eq:iv} we get that $\TBub[i] \wto 0$ and $\Zed[i]{j}\wto 0$ in $\Sob(M)$ as $i\to \infty$.
    \par The idea of the proof is to test equation \eqref{eq:forphi} against various functions, to obtain that $\phi_i \to 0$ in $\Sob(M)$ which will give the contradiction.

    \par\textbf{Step 1.} We prove that $\tsum_{j=0}^n \tabs{\lambda_i^j} \to 0$. Fix $j_0 \in \{0,\ldots n\}$, and test equation \eqref{eq:forphi} against $\Zed[i]{j_0}$ for all $i$. By integration by parts, and the fact that $\Zed[i]{j_0}$ is bounded in $\Sob(M)$, we get
    \[  
        \intM{\err[i]{\Ze{j_0}}\, \phi_i} = o(1) + \sum_{j=0}^k \lambda_i^j \intM{(\Dg+\a_i)^k \Zed[i]{j}\,\Zed[i]{j_0}}.
    \]
    Since we have $\err[i]{\Ze{j_0}} \to 0$ in $\Sob[-k](M)$ as $i\to \infty$ by Proposition \ref{prop:estRa}, with \eqref{eq:DgaZZ} we obtain
    \[  o(1) = \lambda_i^{j_0} + o\Big(\sum_{j=0}^n \tabs{\lambda_i^j}\Big) \qquad \text{ as $i\to \infty$}.
        \]   
    This holds for $j_0 = 0,\ldots n$, we conclude that $\tsum_{j=0}^n \tabs{\lambda_i^j} \to 0$ as $i\to \infty$.

    \par\textbf{Step 2.} Let $\psi \in \Cct(\R^n)$ with $\Hnorm{k,2}{\psi} = 1$, and define $\psi_{\zm[i]} \in \Sob(M)$ by \eqref{def:conc}. We now test equation \eqref{eq:forphi} against $\psi_{\zm[i]}$,
    \begin{multline*}  
        \sum_{l=0}^k \tbinom{k}{l} \a_i^{k-l} \intM{\vprod{\Dg^{l/2}\phi_i}{\Dg^{l/2}\psi_{\zm[i]}}} - (\deu-1) \intM{\TBub[i]^{\deu-2} \phi_i \psi_{\zm[i]}}\\ = o(1) + \sum_{j=0}^k \lambda_i^j \intM{(\Dg+\a_i)^k \Zed[i]{j}\,\psi_{\zm[i]}}.
    \end{multline*}
    First, by \eqref{eq:bdedDgaZ} and Lemma \ref{prop:psizmbded}, since $\tsum_{j=0}^n \tabs{\lambda_i^j} \to 0$ we have 
    \begin{equation}\label{tmp:DgaZpsi}
        \sum_{j=0}^k \lambda_i^j \intM{(\Dg+\a_i)^k \Zed[i]{j}\,\psi_{\zm[i]}}\to 0 \qquad \text{as } i\to \infty.
    \end{equation}
    Define, for all $i$ and for $y \in \Bal{0}{\varrho_0/\mu_i} \sub \R^n$, $\varrho_0 < \inj$ given by Lemma \ref{prop:intcomp} below, 
    \[  \Cphi_{i}(y) = \chi_{\varrho_0}(\mu_i \abs{y}) \mu_i^{\frac{n-2k}{2}} \phi_i(\exp_{z_i}(\mu_i y)), 
        \]
    where $\chi_{\varrho_0} \in \Cct(\R^+)$ is a cut-off function such that $\chi_{\varrho_0} \equiv 1$ on $[0,\varrho_0/2)$ and $\chi_{\varrho_0} \equiv 0$ on $(\varrho_0,+\infty)$. Using Lemma \ref{prop:intcomp}, we obtain $\Hnorm{k,2}{\Cphi_i} \leq C \Snorm{k}{\phi_i}$, and thus $\Cphi_i$ converges weakly, up to a subsequence, to some $\Cphi_0 \in \hSob(\R^n)$ as $i \to \infty$. For $R>0$, we can write by the change of variables $y = \tfrac{1}{\mu_i }\exp_{z_i}^{-1}(x)$,
    \[  \intM[\Bal{z_i}{R\mu_i}]{\TBub[i]^{\deu-2}\phi_i \psi_{\zm[i]}} = \int_{\Bal{0}{R}} \Bub^{\deu-2} \Cphi_i \psi\, dy + o(1)
        \]
    using that $\mu_i^{\frac{n-2k}{2}}\TBub[i](\exp_{z_i}(\mu_i \cdot)) \to \Bub$ and $\exp_{z_i}^* g(\mu_i \cdot) \to \xi$ in $C^\infty_{loc}(\R^n)$ as $i\to \infty$. Since $\psi \in \Cct(\R^n)$, we can now choose $R_\psi>0$ such that the support of $\psi$ is contained in $\Bal{0}{R_\psi}$, this gives up to a subsequence
    \begin{equation}\label{tmp:TBphipsi}
    \begin{aligned}
        \intM{\TBub[i]^{\deu-2}\phi_i \psi_{\zm[i]}} &= \int_{\Bal{0}{R_\psi}} \Bub^{\deu-2}\Cphi_i \psi\, dy + o(1)\\
            &= \int_{\R^n} \Bub^{\deu-2}\Cphi_0 \psi\, dy + o(1)
    \end{aligned}
    \end{equation}
    using the embedding $\hSob(\R^n) \emb L^{\deu}(\R^n)$ and standard integration theory. We now claim that, similarly,
    \begin{equation}\label{tmp:Dgaphipsi}
        \sum_{l=0} \tbinom{k}{l} \a_i^{k-l} \intM{\vprod{\Dg^{l/2}\phi_i}{\Dg^{l/2}\psi_{\zm[i]}}} = \int_{\R^n} \vprod[\xi]{\Dg[\xi]^{k/2}\Cphi_0}{\Dg[\xi]^{k/2}\psi}dy + o(1)
        \end{equation}
    as $i \to \infty$. To prove \eqref{tmp:Dgaphipsi}, fix $R_\psi >0$ as above, it follows that 
    \[  \intM{\vprod{\Dg^{l/2}\phi_i}{\Dg^{l/2}\psi_{\zm[i]}}} = \intM[\Bal{z}{R_\psi \mu_i}]{\vprod{\Dg^{l/2}\phi_i}{\Dg^{l/2}\psi_{\zm[i]}}}
        \]   
    for $l = 0,\ldots k$. For the term $l=k$, by a change of variables we have as before letting $i\to \infty$, up to a subsequence,
    \begin{align*}  
        \intM[\Bal{z_i}{R_\psi\mu_i}]{\vprod{\Dg^{k/2}\phi_i}{\Dg^{k/2}\psi_{\zm[i]}}} &= \int_{\Bal{0}{R_\psi}} \vprod[\xi]{\Dg[\xi]^{k/2}\Cphi_i}{\Dg[\xi]^{k/2}\psi} dy + o(1)\\
            &= \int_{\R^n} \vprod[\xi]{\Dg[\xi]^{k/2}\Cphi_0}{\Dg[\xi]^{k/2}\psi} dy + o(1),
        \end{align*}
    since $\chi_\varrho(\mu \abs{\cdot}) \to 1$ and $\exp_{z_i}^* g(\mu_i \cdot) \to \xi$ in $C^\infty_{loc}(\R^n)$. For $l=0,\ldots k-1$, we have 
    \begin{multline*}  
        \abs{\intM[\Bal{z_i}{R_\psi \mu_i}]{\vprod{\Dg^{l/2}\phi_i}{\Dg^{l/2}\psi_{\zm[i]}}}}\\ \leq \bpr{\intM[\Bal{z_i}{R_\psi\mu_i}]{\abs{\Dg^{l/2}\phi}^2}}^{1/2} \bpr{\intM[\Bal{z_i}{R_\psi \mu_i}]{\abs{\Dg^{l/2}\psi_{\zm[i]}}^2}}^{1/2}.
        \end{multline*}
    Using Hölder's inequality and the Sobolev embedding $\Sob(M) \emb \Sob[l,\deuk{k-l}](M)$, where $\deuk{l} :=\tfrac{2n}{n-2l}$, we deduce that
    \begin{equation}\label{eq:Dglphimu}
    \begin{aligned}
        \intM[\Bal{z_i}{R_\psi\mu_i}]{\abs{\Dg^{l/2}\phi_i}^2} &\leq C \big(R_\psi^n \mu_i^n\big)^{\frac{2(k-l)}{n}} \Lnorm[(M)]{\deuk{k-l}}{\Dg^{l/2}\phi_i}^2\\
            &\leq C \mu_i^{2(k-l)} \Snorm{k}{\phi_i}^2 \leq C \mu^{2(k-l)}.
    \end{aligned}
    \end{equation}
    With a change of variables, we also have
    \begin{equation}\label{eq:Dglpsizm}
    \begin{aligned} 
        \intM[\Bal{z_i}{R_\psi \mu_i}]{\abs{\Dg^{l/2}\psi_{\zm[i]}}^2} &= \mu_i^{2(k-l)}\bpr{\int_{\Bal{0}{R_\psi}}\abs{\Dg[\xi]^{l/2}\psi}^2 dy + o(1)}\\
            &\leq C \mu_i^{2(k-l)} R_\psi^{2(k-l)}\bpr{\int_{\R^n} \abs{\Dg[\xi]^{l/2}\psi}^{\deuk{k-l}}}^{2/\deuk{k-l}}\\
            &\leq C \mu_i^{2(k-l)}\Hnorm{k,2}{\psi}
    \end{aligned}   
    \end{equation}
    using again Hölder's inequality and the embedding $\hSob(\R^n) \emb \hSob[l,\deuk{k-l}](\R^n)$. We thus obtain in the end 
    \[  \sum_{l=0}^k \tbinom{k}{l} \a_i^{k-l} \intM{\vprod{\Dg^{l/2}\phi_i}{\Dg^{l/2}\psi_{\zm[i]}}} = \int_{\R^n} \vprod[\xi]{\Dg[\xi]^{k/2}\Cphi_0}{\Dg[\xi]^{k/2}\psi} dy + \bigO(\a_i \mu_i^2).
        \]
    Recalling that $\a_i \mu_i^2 \to 0$ by assumption, this concludes the proof of \eqref{tmp:Dgaphipsi}. Putting \eqref{tmp:DgaZpsi}, \eqref{tmp:TBphipsi}, and \eqref{tmp:Dgaphipsi} together and letting $i\to \infty$, we have showed that for all $\psi \in \Cct(\R^n)$, 
    \[  \int_{\R^n} \vprod[\xi]{\Dg[\xi]^{k/2}\Cphi_0}{\Dg[\xi]^{k/2}\psi} dy - (\deu-1) \int_{\R^n} \Bub^{\deu-2} \Cphi_0 \psi\, dy = 0.
        \]  
    In other words, $\Cphi_0 \in \hSob(\R^n)$ is a weak solution to \eqref{eq:linbub}, and by definition of $\Ker[]$ in \eqref{def:ker}, we have $\Cphi_0 \in \Ker[]$. Using \eqref{eq:vprodphiZ}, as in the proof of \eqref{eq:iii}, we obtain for $j=0,\ldots n$,
    \[  \vprod[\Sob(M)]{\phi_i}{\Zed[i]{j}} = \intM{\vprod{\Dg^{k/2}\phi_i}{\Dg^{k/2}\Zed[i]{j}}} + o(1).
        \]
    We now have, for a fixed $R>0$,
    \begin{multline}\label{tmp:lininvx}  
        \intM[\Bal{z_i}{R\mu_i}]{\vprod{\Dg^{k/2}\phi_i}{\Dg^{k/2}\Zed[i]{j}}} \\
        \begin{aligned}
            &= \int_{\Bal{0}{R}} \vprod[\xi]{\Dg[\xi]^{k/2}\Cphi_i}{\Dg[\xi]^{k/2}Z_j} dy + o(1)\\
            &= \int_{\R^n} \vprod[\xi]{\Dg[\xi]^{k/2}\Cphi_i}{\Dg[\xi]^{k/2}\Ze{j}} dy + \bigO\big(\Hnorm[(\R^n \setminus \Bal{0}{R})]{k,2}{\Ze{j}}\big) + o(1)\\
            &= \int_{\R^n} \vprod[\xi]{\Dg[\xi]^{k/2}\Cphi_0}{\Dg[\xi]^{k/2}\Ze{j}} dy+ \bigO\big(\Hnorm[(\R^n \setminus \Bal{0}{R})]{k,2}{\Ze{j}}\big) + o(1)
        \end{aligned}
    \end{multline}
    since $\Cphi_i \wto \Cphi_0$ weakly in $\hSob(\R^n)$ as $i\to \infty$, up to a subsequence. Moreover,
    \begin{multline}\label{tmp:lininvxx}
        \abs{\intM[M\setminus \Bal{z_i}{R\mu_i}]{\vprod{\Dg^{k/2}\phi_i}{\Dg^{k/2}\Zed[i]{j}}}}\\
        \begin{aligned}
            &\leq \Snorm{k}{\phi_i} \bpr{\intM[M\setminus \Bal{z_i}{R\mu_i}]{\abs{\Dg^{k/2}\Zed[i]{j}}^2}}^{1/2}\\
            &\leq C \Hnorm[(\R^n \setminus \Bal{0}{R})]{k,2}{\Ze{j}}.
        \end{aligned} 
    \end{multline}
    Hence, letting $R\to \infty$ in \eqref{tmp:lininvx}, \eqref{tmp:lininvxx}, and using assumption \eqref{eq:prodphiZ} we conclude that 
    \[  \vprod[\hSob(\R^n)]{\Cphi_0}{\Ze{j}} = 0 \qquad j = 0,\ldots n,
        \]
    and thus $\Cphi_0 \in \Ker[]^\perp$. We then have $\Cphi_0 = 0$ and $\Cphi_i \wto 0$ in $\hSob(\R^n)$, up to a subsequence.

    \par\textbf{Step 3: Contradiction.} We now test \eqref{eq:forphi} against $\phi_i \in \Sob(M)$. Since $(\phi_i)$ is bounded in $\Sob(M)$, with \eqref{eq:bdedDgaZ}, and since that $\lambda_i^j \to 0$ for $j= 0,\ldots n$, we obtain 
    \[  \sum_{l=0}^k \tbinom{k}{l} \a_i^{k-l} \intM{\abs{\Dg^{l/2}\phi_i}^2} - (\deu-1)\intM{\TBub[i]^{\deu-2}\phi_i^2} = o(1)
        \]
    as $i\to \infty$. Observe that, as a consequence,
    \[  \Snorm{k}{\phi_i}^2 = \sum_{l=0}^k \intM{\abs{\Dg^{l/2}\phi_i}^2} \leq C \intM{\TBub[i]^{\deu-2}\phi_i^2} + o(1)
        \]
    as $i\to \infty$, and we estimate this last quantity. Fix some $R>0$, reasoning as before,
    \begin{align*}  
        \intM[\Bal{z_i}{R\mu_i}]{\TBub[i]^{\deu-2}\phi_i^2} &= \int_{\Bal{0}{R}} \Bub^{\deu-2} \Cphi_i^2 \,dy + o(1)\\
            &= \int_{\R^n} \Bub^{\deu-2} \Cphi_i^2 \,dy + \bigO\big(\Lnorm[(\R^n \setminus \Bal{0}{R})]{\deu}{\Bub}^{\deu-2}\big)+ o(1).
        \end{align*}
    Now since we have $\Cphi_i \wto 0$ in $\hSob(\R^n)$ up to a subsequence, using again standard integration theory we get $\Cphi_i \wto 0$ in $L^\deu(\R^n)$. Since $\Bub \in L^{\deu}(\R^n)$, this implies 
    \[  \int_{\R^n} \Bub^{\deu-2} \Cphi_i^2\,dy \to 0 \qquad \text{as $i\to \infty$}.
        \]
    We also obtain by Hölder's inequality and the Sobolev embedding 
    \begin{align*}  
        \intM[M\setminus \Bal{z_i}{R\mu_i}]{\TBub[i]^{\deu-2} \phi_i^2} &\leq C \bpr{\intM[M\setminus\Bal{z_i}{R\mu_i}]{\TBub[i]^\deu}}^{\frac{2k}{n}} \Snorm{k}{\phi_i}^2\\
            &\leq C \Lnorm[(\R^n \setminus \Bal{0}{R})]{\deu}{\Bub}^{\deu-2}.
    \end{align*}
    Letting $R\to \infty$, we have showed that
    \[  \Snorm{k}{\phi_i}^2 \to 0 \qquad \text{as $i\to \infty$}
        \]
    up to a subsequence, which contradicts \eqref{eq:normphi}.
\end{proof}

\subsection{Non-linear procedure in energy space}
The objective of this section is to construct a solution $u_{\zma} \in \Sob(M)$, satisfying \eqref{eq:critbub} up to terms that belong to $\Ker$, as a perturbation around $\TBub$. We prove existence and uniqueness of such a $u_{\zma} = \TBub + \phi_{\zma}$, for small $\phi_{\zma}$ in $\Sob(M)$. In the following we write $f_+(x) := \max(0,f(x))$ the positive part of $f$.
\begin{proposition}\label{prop:uniqHk}
    There exists $\delta>0$, $\a_0 \geq 1$ and $\tau_0 >0$ such that for all $\a \geq \a_0$, $\tau \leq \tau_0$, and $\pa = (\zm) \in \param$, there exists a unique solution 
    \[  \phi_{\zma} \in \Ker^\perp \cap \{\phi \in \Sob(M) \such \Snorm{k}{\phi}<\delta\}
        \]
    to 
    \begin{equation}\label{eq:nlinphi}
        \proKe\bsq{\TBub + \phi_{\zma} - (\Dg+\a)^{-k}\bpr{(\TBub+\phi_{\zma})_+^{\deu-1}}}= 0,
    \end{equation}
    writing $\proKe$ the projection in $\Sob(M)$ onto the orthogonal to $\Ker$ as defined in \eqref{def:kerzm}.
\end{proposition}
The proof follows from a fixed-point argument. We refer to \cite{MusPis02,RobVet13} for instances where a similar method is developed in the case $k=1$ and for operators with bounded coefficients. 
We will need to estimate the following two quantities.
\begin{definition}
    Let $\a \geq 1$ be such that $1/\sqa < \inj/2$, $\tau\leq 1$, and $\pa = (\zm)\in \param$. We define 
    \begin{equation}\label{def:EaGa}
        \begin{aligned}
            E_{\zma} &:= \proKe\bsq{\TBub - (\Dg+\a)^{-k} \TBub^{\deu-1}}\\
            G_{\zma}(\phi) &:= (\TBub+\phi)_+^{\deu-1} - \TBub^{\deu-1} - (\deu-1) \TBub^{\deu-2}\phi \qquad \forall~\phi \in \Ker^\perp.
        \end{aligned}
    \end{equation}
\end{definition}
\begin{lemma}\label{prop:EaHk}
    For all $\delta>0$, there exists $\a_0\geq 1$ and $\tau_0> 0$ such that, for all $\a \geq \a_0$, $\tau\leq \tau_0$, and $\pa \in \param$, we have 
    \[  \Snorm{k}{E_{\zma}} \leq \delta.
        \]
\end{lemma}
\begin{proof}
    This is an immediate consequence of Proposition \ref{prop:estRa}. Since $\proKe$ is a projection, and by \eqref{eq:Dgabded}, we have
    \begin{align*}  
        \Snorm{k}{E_{\zma}} &\leq \Snorm{k}{\TBub - (\Dg+\a)^{-k}\TBub^{\deu-1}}\\
            &\leq \Snorm{-k}{(\Dg+\a)^k \TBub - \TBub^{\deu-1}} \leq \tau^{1/4}.
    \end{align*}
\end{proof}
\begin{lemma}\label{prop:GaHk}
    There exist $0<\theta < \min(1, \deu-2)$, $\a_0 \geq 1$, $\tau_0\leq 1$ and a constant $C>0$ such that, for all $\a \geq \a_0$, $\tau\leq \tau_0$, and $\pa \in \param$, the following holds. Let $\phi_0, \phi_1, \phi_2 \in L^\deu(M)$ be such that $\Lnorm[(M)]{\deu}{\phi_j} \leq 1$, $j=0,1,2$, then 
    \begin{align}
    \label{tmp:Ga1}    &\Lnorm[(M)]{\frac{2n}{n+2k}}{G_{\zma}(\phi_0)} \leq C \Lnorm[(M)]{\deu}{\phi_0}^{1+\theta}\\
    \label{tmp:Ga2}    &\Lnorm[(M)]{\frac{2n}{n+2k}}{G_{\zma}(\phi_1)- G_{\zma}(\phi_2)}\\
        \notag &\qquad\qquad\qquad \leq C \bpr{\Lnorm[(M)]{\deu}{\phi_1}^\theta + \Lnorm[(M)]{\deu}{\phi_2}^\theta} \Lnorm[(M)]{\deu}{\phi_1-\phi_2}.
    \end{align}
\end{lemma}
\begin{proof}
    We use the following identity: There exist a constant $C>0$ and an exponent $0< \theta < \min(1,\deu-2)$ such that for all $b>0$, $a,a_1,a_2 \in \R$, the following holds:
    \begin{equation}\label{eq:estpoint1} 
        \abs{(b+a)_+^{\deu-1} - b^{\deu-1} - (\deu-1)\, b^{\deu-2} a} \leq C \abs{a}\bpr{\abs{a}^{\deu-2} + b^{\deu-2-\theta}\abs{a}^{\theta}}
    \end{equation}
    and 
    \begin{multline}\label{eq:estpoint2}
        \abs{(b+a_1)_+^{\deu-1} - (b+a_2)_+^{\deu-1} - (\deu-1)\, b^{\deu-2}(a_1 - a_2)}\\
            \leq C \abs{a_1 - a_2} \bpr{b^{\deu-2-\theta}\abs{a_1}^\theta + \abs{a_1}^{\deu-2} + b^{\deu-2-\theta}\abs{a_2}^\theta + \abs{a_2}^{\deu-2}}.
    \end{multline}
    Apply \eqref{eq:estpoint1} at each point and integrate, this gives
    \begin{multline}\label{tmp:Gax}
        \intM{\abs{(\TBub + \phi_0)_+^{\deu-1} - \TBub^{\deu-1} - (\deu-1) \TBub^{\deu-2} \phi_0}^{\frac{2n}{n+2k}}}\\
            \begin{aligned}
                &\leq C\intM{\bpr{\abs{\phi_0}^{\deu-1} + \TBub^{\deu-2-\theta}\abs{\phi_0}^{1+\theta}}^{\frac{2n}{n+2k}}}\\
                &\leq C \intM{\abs{\phi_0}^\deu} + C \bpr{\intM{\TBub^\deu}}^{\frac{\deu-2-\theta}{\deu-1}} \bpr{\intM{\abs{\phi_0}^\deu}}^{\frac{1+\theta}{\deu-1}}
            \end{aligned}
    \end{multline}
    by Hölder's inequality. Now using \eqref{eq:ii} we get that, for a big enough $\a_0$, small enough $\tau_0$, there exists $C>0$ such that for all $\a \geq \a_0$, $\tau\leq \tau_0$, and $\pa \in \param$,
    \[  \bpr{\intM{\TBub^\deu}}^{\frac{\deu-2-\theta}{\deu-1}} \bpr{\intM{\abs{\phi_0}^\deu}}^{\frac{1+\theta}{\deu-1}} \leq C \Lnorm[(M)]{\deu}{\phi_0}^{(1+\theta)\frac{2n}{n+2k}}.
        \]
    Therefore, coming back to \eqref{tmp:Gax}, we have showed that
    \[  \Lnorm[(M)]{\frac{2n}{n+2k}}{G_{\zma}(\phi_0)} \leq C \bpr{\Lnorm[(M)]{\deu}{\phi_0}^{\deu-1} + \Lnorm[(M)]{\deu}{\phi_0}^{1+\theta}} \leq C \Lnorm[(M)]{\deu}{\phi_0}^{1+\theta}
        \]
    since $\Lnorm[(M)]{\deu}{\phi_0} \leq 1$. 
    \par Using the same arguments with \eqref{eq:estpoint2}, we obtain \eqref{tmp:Ga2}.
\end{proof}
\begin{proof}[Proof of Proposition \ref{prop:uniqHk}]
    Let $\a_0$ and $\tau_0$ be given as the respective maxima and minima of the values obtained in Proposition \ref{prop:lininv} and Lemma \ref{prop:GaHk}, and let $\a \geq \a_0$, $\tau\leq \tau_0$, and $\pa = (\zm) \in \param$. Observe that $\phi \in \Ker^\perp$ is a solution to \eqref{eq:nlinphi} if and only if it satisfies
    \begin{equation}\label{tmp:eqphix}
        E_{\zma} - \proKe \circ (\Dg+\a)^{-k} \big(G_{\zma}(\phi)\big) + \Lia\phi = 0,
    \end{equation}
    where $E_{\zma}$ and $G_{\zma}$ are defined in \eqref{def:EaGa}. By the choice of $\a_0$, $\tau_0$, $\Lia$ is uniformly invertible, so that defining the operator $T_{\zma} : \Ker^\perp \to \Ker^\perp$ as 
    \[  T_{\zma}(\phi) := -\Lia^{-1}\big(E_{\zma}\big) + \Lia^{-1} \circ \proKe \circ (\Dg+\a)^{-k} \big(G_{\zma}(\phi)\big),
        \]
    equation \eqref{tmp:eqphix} is equivalent to 
    \[  \phi = T_{\zma}(\phi).
        \]
    We first show that $T_{\zma}$ stabilizes the set
    \[  \Set_\delta := \{ \phi \in \Ker^\perp \such \Snorm{k}{\phi} < \delta\},
        \]
    for a small enough $0< \delta\leq 1$ to be chosen later. We have
    \begin{align}
        \notag &\Snorm{k}{\Lia^{-1}(E_{\zma})} \leq C \Snorm{k}{E_{\zma}}\\
        \label{tmp:estGaxx} &\Snorm{k}{\Lia^{-1}\circ \proKe \circ (\Dg+\a)^{-k}\big(G_{\zma}(\phi)\big)} \leq C \Snorm{-k}{G_{\zma}(\phi)}.
    \end{align}
    Using the continuous embeddings $L^{\frac{2n}{n+2k}}(M) \emb \Sob[-k](M)$ and $\Sob(M) \emb L^\deu(M)$, and with Lemma \ref{prop:GaHk}, we have $G_{\zma}(\phi) \in \Sob[-k](M)$ and 
    \[  \Snorm{-k}{G_{\zma}(\phi)} \leq C \Lnorm[(M)]{\deu}{\phi}^{1+\theta} \leq C \delta^{1+\theta}
        \]
    for all $\phi \in \Set_\delta$, for some $0<\theta<\min(1,\deu-2)$ independent of $\a,\tau$. Using Lemma \ref{prop:EaHk}, up to increasing $\a_0$ and decreasing $\tau_0$, we obtain
    \begin{align*}
        \Snorm{k}{T_{\zma}} &\leq C \bpr{\Snorm{k}{E_{\zma}} + \Snorm{-k}{G_{\zma}(\phi)}}\\
            &\leq \frac{\delta}{2} + C \delta^{1+\theta}.
    \end{align*}
    We now fix $\delta >0$ small enough so that $C\delta^{1+\theta} < \delta/2$, and we see that $T_{\zma}(\Set_\delta) \subeq \Set_\delta$.
    \par Let $\phi_1,\phi_2 \in \Set_\delta$, we have 
    \begin{align*}
        \Snorm{k}{T_{\zma}(\phi_1) - T_{\zma}(\phi_2)} &\leq C \Snorm{-k}{G_{\zma}(\phi_1)- G_{\zma}(\phi_2)}\\
            &\leq C \delta^\theta \Snorm{k}{\phi_1-\phi_2}
    \end{align*}
    using again \eqref{tmp:estGaxx}, the continuous Sobolev embeddings, and Lemma \ref{prop:GaHk}. Up to choosing $\delta>0$ small enough so that $C\delta^\theta < 1$, $T_{\zma}$ is a contraction on $\Set_\delta$. By Banach's fixed-point Theorem, there exists a unique $\phi \in \Set_\delta$ such that $T_{\zma}(\phi)= \phi$.
\end{proof}
\begin{remark}
    Note that nothing imposes the solution $u_{\zma} := \TBub + \phi_{\zma}$ of \eqref{eq:nlinphi} to be positive in $M$, so that we must take this particularity into account when raising it to the critical exponent $\deu \not\in \Z$. This justifies having considered the positive part of $u_{\zma}$ in \eqref{eq:nlinphi}. This equation then corresponds to the projection onto $\Ker^\perp$ of the slightly modified equation
    \begin{equation}\label{eq:modcrit}
        (\Dg+\a)^k u = u_+^{\deu-1}, \qquad u \in \Sob(M).
    \end{equation}
    However, thanks to the factorized form of the polyharmonic operator, if $u_{\zma}$ is a solution to \eqref{eq:modcrit}, we can use the maximum principle and $u_{\zma}$ is non-negative. Therefore, $(u_{\zma})_+ = u_{\zma}$, and $u_{\zma}$ solves \eqref{eq:critbub}. Conversely, if $u_{\zma}$ is a positive solution to \eqref{eq:critbub}, then it also satisfies \eqref{eq:nlinphi}.
\end{remark}

\section{Blow-up analysis}\label{sec:pointwise}
\subsection{Pointwise behavior of solutions to the linearized equation}
In section \ref{sec:unifinv}, we proved that $\Lia$ is a bounded operator, uniformly invertible in $\Ker^\perp$ with respect to $\a\geq \a_0$, $\tau\leq \tau_0$. For any sequence $(R_i)_i$ such that $R_i \to 0$ in $\Sob[-k](M)$, for any $(\a_i)_i$, $(\tau_i)_i$ such that $\a_i \to \infty$, $\tau_i \to 0$ as $i\to \infty$ and $\pa_i\in \param[i]$, we have trivially $\proKi\bpr{(\Dg+\a_i)^{-k} R_i} \in \Keri^\perp$ and Proposition \ref{prop:lininv} applies : There exists $i_0$ such that for all $i\geq i_0$, there is a unique $\phi_i \in \Keri$ and a unique family $\{\lambda_i^j \in \R \such j=0,\ldots n\}$ satisfying
\[  (\Dg+\a_i)^k \phi_i - (\deu-1) \TBub[i]^{\deu-2} \phi_i = R_i + \sum_{j=0}^n \lambda_i^j (\Dg+\a_i)^k \Zed[i]{j}.
    \]
Moreover, the $\Sob(M)$ norm of $\phi_i$ is controlled by the $\Sob[-k](M)$ norm of $R_i$. This control can in fact be improved, we show here below how $\phi_i$ inherits pointwise estimates from $R_i$, when $R_i$ has suitable bounds. 
\begin{definition}
    Let $\a\geq 1$ be such that $1/\sqa < \inj/2$, $\tau\leq 1$, and $\pa = (\zm) \in \param$. We define for all $x \in M$ such that $\dg{z,x} \leq 1/\sqa$,
    \begin{multline}\label{def:ra} 
        r_{\zma}(x) := \\ \begin{cases}
            \sqa (\mu + \dg{z,x})\big(1 + \abs{\log \sqa(\mu+\dg{z,x})}\big) & \text{if } n =2k+1\\
            \a^{3/4} (\mu+ \dg{z,x})^{3/2} & \begin{aligned}[t]
                &\text{if } k=1, \,n\geq 4\\
                &~\text{or } k\geq 2,\, n=2k+2
            \end{aligned}\\
            \a (\mu + \dg{z,x})^2 & \text{if } k \geq 2,\,n\geq 2k+3
        \end{cases},
    \end{multline}
    so that $r_{\zma}(x) = \rad\big(\sqa(\mu + \dg{z,x})\big)$, for $\rad$ as in \eqref{def:rad}.
\end{definition}
This subsection is devoted to the proof of the following result.

\begin{theorem}\label{prop:linest}
    Let $(\tau_i)_i$ and $(\a_i)_i$ be sequences of positive numbers such that $\a_i \to \infty, \tau_i \to 0$ as $i\to \infty$, and let $\pa_i = (\zm[i]) \in \param[i]$ for all $i$. There exists $C_0>0$ and $i_0>0$ such that the following holds :
    \par Let $(R_i)_i$ be a sequence of continuous functions in $M$, satisfying
    \begin{equation}\label{eq:estRiprop}
        \abs{R_i(x)} \leq \Bub_{\zm[i]} \times\begin{cases}
            \a_i^{\sigma_{n,k}} (\mu_i + \dg{z_i, x})^{2\sigma_{n,k}-2k} & \text{when } \sqai \dg{z_i,x} \leq 1\\
            \a_i^k e^{-\sqai \dg{z_i,x}/2} & \text{when } \sqai \dg{z_i,x} \geq 1
        \end{cases}
    \end{equation}
    in $M$, where 
    \begin{equation}\label{def:snk}
        \sigma_{n,k} = \begin{cases}
            1/2 & \text{if } n=2k+1\\
            3/4 & \begin{aligned}[t]
                &\text{if } k=1, \,n\geq 4\\
                &~\text{or } k\geq 2,\, n=2k+2
            \end{aligned}\\
            1   & \text{if } k\geq 2,\, n\geq 2k+3
        \end{cases}.
    \end{equation}
    Let, for all $i\geq i_0$, $\phi_i \in \Keri^\perp$ be the unique solution given by Proposition \ref{prop:lininv} to 
    \begin{equation}\label{eq:linphiR}
        (\Dg+\a_i)^k \phi_i - (\deu-1) \TBub[i]^{\deu-2} \phi_i = R_i + \sum_{j=0}^n \lambda_i^j (\Dg+\a_i)^k \Zed[i]{j},
    \end{equation}
    for some unique choice of $\lambda_i^j \in \R$, $j=0,\ldots n$. Then $\phi_i$ is continuous in $M$ and satisfies 
    \begin{multline}\label{eq:estphi}
        \abs{\phi_i(x)} \leq C_0 \Bub_{\zm[i]}(x)\\ \times \begin{cases}
        r_{\zma[i]}(x) & \text{when } \sqai \dg{z_i,x}\leq 1\\
        \a_i^k \dg{z_i,x}^{2k} e^{-\sqai \dg{z_i,x}/2} & \text{when } \sqai \dg{z_i,x}\geq 1,
    \end{cases}
        \end{multline}
    where $r_{\zma[i]}$ is defined in \eqref{def:ra}.
\end{theorem}
\par By Proposition \ref{prop:estRa}, a sequence $(R_i)_i$ satisfying \eqref{eq:estRiprop} is for instance given by $R_i = \err[i]{\Bub} := (\Dg+\a_i)^k \TBub[i]-\TBub[i]^{\deu-1}$. In that regard, \eqref{eq:linphiR} can be seen as a linearization around $\TBub[i]$ of the equation \eqref{eq:critbub}. This is essentially how we will use Theorem \ref{prop:linest} in the next subsection. As it will be clear here below when using a representation formula to estimate $\phi_i$, we have that $\phi_i$ behaves formally as $(\Dg+\a_i)^{-k}R_i$. 
\par The proof of Theorem \ref{prop:linest} closely follows the arguments of \cite{Pre24}, with some significant adaptations. While the presence of only one bubble simplifies the arguments, the diverging coefficient $\a_i$ creates additional difficulties in the pointwise analysis of the blow-up around $\TBub[i]$. The strategy of the proof is as follows. We first use a representation formula to obtain a naive pointwise estimate on $\phi_i$ that depends on $\Lnorm[(M)]{\infty}{\phi_i}$. In a second step, we obtain a bound on $\Lnorm[(M)]{\infty}{\phi_i}$ in terms of $\a_i$ and $\mu_i$ using a contradiction argument, which will allow to conclude.

\par As in the proof of Proposition \ref{prop:lininv}, we start by observing that $\tsum_{j=0}^n \tabs{\lambda_i^j} \to 0$.
\begin{lemma}\label{prop:limli}
    Let $(\lambda_i^j)_i$, $j=0,\ldots n$, be the coefficients given in \eqref{eq:linphiR}, Then
    \[  \lim_{i\to \infty} \sum_{j=0}^n \tabs{\lambda_i^j} = 0.
        \]
\end{lemma}
\begin{proof}
    Since $R_i$ satisfies the same type of pointwise bounds as the term $\err{X}$ in Proposition \ref{prop:estRa}, by Lemma \ref{prop:estintRX} we have $R_i \to 0$ in $\Sob[-k](M)$. By Proposition \ref{prop:lininv}, there exists $C>0$ such that
    \[  \Snorm{k}{\phi_i} \leq C \Snorm{-k}{R_i} = o(1).
        \]
    We test \eqref{eq:linphiR} against $\Zed[i]{j_0}$ for a fixed $j_0 \in \{0,\ldots n\}$, with Proposition \ref{prop:estRa} and \eqref{eq:DgaZZ} we obtain 
    \[  \lambda_i^{j_0} + o\Big(\sum_{j=0}^n \tabs{\lambda_i^j}\Big) = o(1) \qquad \text{as } i \to \infty,
        \]
    and we conclude.
\end{proof}

We now obtain a first bound on the function $\phi_i$. For the sake of brevity, we write until the end of this subsection $\inorm{\phi_i} := \Lnorm[(M)]{\infty}{\phi_i}$. 
We will use here below the Green's function $\Gga$, for the operator $(\Dg+\a)^k$ in $M$, studied in \cite{Car24}, as well as versions of the so-called Giraud's Lemma, which are recalled in Appendix \ref{sec:gir}. We define 
\begin{equation}\label{def:Fa}
    F_{\zma}(x) := \begin{cases}
        r_{\zma}(x) & \text{when } \sqa\dg{z,x} \leq 1\\
        \a^k \dg{z,x}^{2k}e^{-\sqa\dg{z,x}/2} & \text{when } \sqa\dg{z,x} \geq 1
    \end{cases}
\end{equation}
for all $x\in M$, where $r_{\zma}$ is as in \eqref{def:ra}.

\begin{lemma}\label{prop:1stestphi}
    There exists $i_0>0$ such that for all $i\geq i_0$, the unique solution $\phi_i \in \Keri^\perp$ to \eqref{eq:linphiR} in Theorem \ref{prop:linest} satisfies
    \begin{multline}\label{eq:1stestphi}
        \abs{\phi_i(x) - \sum_{j=0}^n \lambda_i^j \Zed[i]{j}(x)} \leq C F_{\zma[i]}(x) \Bub_{\zm[i]}(x)\\ + C \inorm{\phi_i} \Psi\big(\sqai\dg{z_i,x}\big)\begin{cases}
            \mu_i^{\frac{n-2k}{2}} \Bub_{\zm[i]}(x) & \text{if } n < 4k\\
            \bpr{\frac{\mu_i + \dg{z_i,x}}{\mu_i}}^{-2k+1} & \text{if } n=4k\\
            \bpr{\frac{\mu_i + \dg{z_i,x}}{\mu_i}}^{-2k} & \text{if } n> 4k
        \end{cases},
    \end{multline}
    where
    \begin{equation}\label{def:Psi}
        \Psi(t) := \begin{cases}
            1 & \text{when } t< 1\\
            e^{-t/2} & \text{when } t\geq 1
        \end{cases}
    \end{equation}
    for all $t\geq 0$.
\end{lemma}
\begin{proof}
    First, note that for $\a_0,\tau_0$ given by Proposition \ref{prop:lininv}, there exists $i_0>0$ such that  $\a_i \geq \a_0$, $\tau_i \leq \tau_0$ for all $i\geq i_0$.
    Since $\TBub[i], \Zed[i]{j} \in C^{2k}(M)$ and $R_i \in C^0(M)$, by a standard bootstrap argument, we have $\phi_i \in C^{2k-1,\beta}(M)$ for $\beta \in (0,1)$. Observe that equation \eqref{eq:linphiR} can be re-written as 
    \[  (\Dg+\a_i)^k \bsq{\phi_i -\sum_{j=0}^n \lambda_i^j \Zed[i]{j}} = (\deu-1) \TBub[i]^{\deu-1}\phi_i + R_i \qquad \text{in } M.
        \]
    We then have the following representation formula: For all $x\in M$,
    \begin{multline}\label{eq:reprforphi}
        \phi_i(x) -\sum_{j=0}^n \lambda_i^j \Zed[i]{j}(x) = (\deu-1) \intM{\TBub[i]^{\deu-2}(y) \phi_i(y) \Gga[i](x,y)}(y)\\ + \intM{R_i(y)\Gga[i](x,y)}(y).
    \end{multline}
    We now use the estimates on $\Gga[i]$ in \cite[Theorem 1.1]{Car24}, namely that up to increasing $i_0 >0$, for all $i\geq i_0$ we have 
    \[  \Gga[i](x,y) \leq C \dg{x,y}^{2k-n} \Psi(\sqai\dg{x,y}).
        \]
    With the definition of $\TBub[i]$ in \eqref{def:TBub} and Giraud's Lemma \ref{prop:gir2} below, we obtain the following:
    \begin{itemize}
        \item When $n<4k$,
            \begin{multline*}
                \abs{\intM{\Gga[i](x,y)\TBub[i]^{\deu-2}(y)\phi_i(y)}(y)}\\
                    \leq C\inorm{\phi_i} \Psi\big(\sqai \dg{z_i,x}\big) \mu_i^{\frac{n-2k}{2}}\Bub_{\zm[i]}(x).
            \end{multline*}
        \item When $n=4k$, for all $\theta>0$, there exists $C_\theta>0$ such that 
            \begin{multline*}  
                \abs{\intM{\Gga[i](x,y)\TBub[i]^{\deu-2}(y)\phi_i(y)}(y)}\\ \leq C_\theta \inorm{\phi_i} \Psi\big(\sqai\dg{z_i,x}\big) \bpr{\frac{\mu_i + \dg{z_i,x}}{\mu_i}}^{-2k+\theta}
                \end{multline*}
            using the properties of the function $t \mapsto \log t$ in $(0,+\infty)$.
        \item When $n>4k$, 
            \begin{multline*}  
                \abs{\intM{\Gga[i](x,y)\TBub[i]^{\deu-2}(y)\phi_i(y)}(y)}\\ \leq C \inorm{\phi_i} \Psi\big(\sqai\dg{z_i,x}\big) \bpr{\frac{\mu_i + \dg{z_i,x}}{\mu_i}}^{-2k}.
            \end{multline*}
    \end{itemize}
    We also compute, using Lemma \ref{prop:gir1} with \eqref{eq:estRiprop}, when $\sqai \dg{z_i,x} \leq 1$,
    \begin{multline}\label{eq:girRasqa}
        \abs{\intM{R_i(y)\Gga[i](x,y)}(y)} \\
        \begin{aligned}
            &\leq C\begin{cases}
            \a_i^{1/2}\mu_i^{1/2}\big(1+\abs{\log \sqai(\mu_i + \dg{z_i,x})}\big) & \text{if } n=2k+1\\
            \a_i^{\sigma_{n,k}} \mu_i^{\frac{n-2k}{2}}(\mu_i+\dg{z_i,x})^{2k-n+2\sigma_{n,k}} & \text{if } n\geq 2k+2
            \end{cases}\\
                &\leq C r_{\zma[i]}(x) \Bub_{\zm[i]}(x),
        \end{aligned}
    \end{multline}
    where $r_{\zma}$ is as in \eqref{def:ra}. Finally, when $\sqai \dg{z_i,x} \geq 1$, we get 
    \begin{equation}\label{eq:girRaexp}
    \begin{aligned}
        \abs{\intM{R_i(y)\Gga[i](x,y)}(y)} &\leq C\a_i^k \mu_i^{\frac{n-2k}{2}}\dg{z_i,x}^{4k-n}e^{-\sqai\dg{z_i,x}/2}\\
            &\leq C \a_i^k \dg{z_i,x}^{2k} \Bub_{\zm[i]}(x) e^{-\sqai\dg{z_i,x}/2},
    \end{aligned}
    \end{equation}
    which concludes the proof.
\end{proof}

The following Lemma improves the result of Lemma \ref{prop:limli}. 
\begin{lemma}\label{prop:estlami}
    Let $\lambda_i^j$, $j=0,\ldots n$, be the coefficients appearing in \eqref{eq:linphiR}. Then we have for all $i\geq i_0$, $i_0$ given by Lemma \ref{prop:1stestphi},
    \[   \abs{\lambda_i^j} \leq C \rad(\sqai\mu_i) + C \mu_i^{\frac{n-2k}{2}}\inorm{\phi_i} \qquad j=0,\ldots n.
        \]
\end{lemma}
\begin{proof}
    Recall the definition of $\Bub_{\zm[i]}$ in \eqref{def:bubzm} and $\Zed[i]{j}$ in \eqref{def:tZ}, we have 
    \begin{align*}
        \Bub_{\zm[i]}(z_i) &= \mu_i^{-\frac{n-2k}{2}}, & \abs{\Zed[i]{0}(z_i)} &= \tfrac{n-2k}{2}\mu_i^{-\frac{n-2k}{2}}, & \Zed[i]{j}(z_i) = 0 \quad 1\leq j\leq n.
    \end{align*}
    Evaluating \eqref{eq:1stestphi} at the point $z_i$, we thus obtain for all $i$,
    \begin{align*}
        \abs{\lambda_i^0} \mu_i^{-\frac{n-2k}{2}} &\leq \abs{\phi_i(z_i)} + CF_{\zma[i]}(z_i) \mu_i^{-\frac{n-2k}{2}} + C\inorm{\phi_i} \Psi(0)\\
            &\leq C \inorm{\phi_i} + C r_{\zma[i]}(z_i) \mu_i^{-\frac{n-2k}{2}},
    \end{align*}
    and we have by definition $r_{\zma[i]}(z_i) =\rad(\sqai \mu_i)$.
    \par Let now $j_0 \in \{1,\ldots n\}$ be fixed, and let $\Tilde{\rho}_{n,k} >0$ be a constant such that $\Ze{0}\big((\Tilde{\rho}_{n,k}, 0,\ldots 0)\big) = 0$, its explicit value can be computed using \eqref{def:Z}. Let $v_{j_0}\in\R^n$ be the $j_0$-th vector of the canonical basis of $\R^n$, and define 
    \[  y_i^{j_0} = \exp_{z_i}\big(\mu_i \Tilde{\rho}_{n,k}\, v_{j_0}\big).
    \]
    This way, we have
    \begin{align*}
        \Zed[i]{0}(y_i^{j_0}) &= 0, & \Zed[i]{j}(y_i^{j_0}) = 0 \quad j\neq j_0,
    \end{align*}
    and 
    \begin{align*}
        \Bub_{\zm[i]}(y_i^{j_0}) &= \mu_i^{-\frac{n-2k}{2}} (1+\pk\Tilde{\rho}_{n,k}^2)^{-\frac{n-2k}{2}} & \Zed[i]{j_0}(y_i^{j_0}) &= \mu_i^{-\frac{n-2k}{2}} \tfrac{\pk\Tilde{\rho}_{n,k}^2}{(1+\pk\Tilde{\rho}_{n,k}^2)^{\frac{n-2k}{2}+1}}.
    \end{align*}
    As before, evaluate \eqref{eq:1stestphi} at the point $y_i^{j_0}$, we obtain for all $i$,
    \[  \abs{\lambda_i^{j_0}}\mu_i^{-\frac{n-2k}{2}} \leq C \inorm{\phi_i} + C r_{\zma[i]}(y_i^{j_0}) \mu_i^{-\frac{n-2k}{2}},
        \]
    and we have $r_{\zma[i]}(y_i^{j_0}) \leq C \rad(\sqai \mu_i)$ using that $\dg{z_i,y_i^{j_0}} \leq C \mu_i$.
\end{proof}

We now iteratively improve the bound of Lemma \ref{prop:1stestphi} using the representation formula and Giraud's Lemma.
\begin{proposition}\label{prop:2ndestphi}
    Let $\phi_i\in\Keri^\perp$ be the unique solution to \eqref{eq:linphiR} in Theorem \ref{prop:linest}. There exists $C>0$ such that $\phi_i$ satisfies, for all $i\geq i_0$ and for all $x\in M$,
    \begin{equation}\label{eq:2ndestphi}
        \abs{\phi_i(x)} \leq C F_{\zma[i]}(x) \Bub_{\zm[i]}(x) + C\mu_i^{\frac{n-2k}{2}}\inorm{\phi_i} \Bub_{\zm[i]}(x)\Psi\big(\sqai \dg{z_i,x}\big),
    \end{equation}
    where $i_0$ is given by Lemma \ref{prop:1stestphi}, $F_{\zma[i]}$ is defined in \eqref{def:Fa}, and $\Psi$ is defined in \eqref{def:Psi}.
\end{proposition}
\begin{proof}
    We use the estimate of Lemma \ref{prop:estlami}, with the observation that, by definition, 
    \begin{equation}\label{eq:basicZB}
        \abs{\Zed[i]{j}(x)} \leq C\Bub_{\zm[i]}(x)\Psi\big(\sqai \dg{z_i,x}\big) \quad \forall~x\in M, ~j=0,\ldots n,
    \end{equation}
    and that 
    \begin{equation}\label{eq:raPsitoFa}
        \rad(\sqai\mu_i) \Psi\big(\sqai\dg{z_i,x}\big) \leq F_{\zma[i]}(x) \quad \forall~x\in M.
    \end{equation}
    The bound \eqref{eq:1stestphi} thus becomes 
    \begin{multline}\label{tmp:2ndix}
        \abs{\phi_i(x)} \leq CF_{\zma[i]}(x)\Bub_{\zm[i]}(x)\\ + C\inorm{\phi_i} \Psi\big(\sqai \dg{z_i,x}\big) \begin{cases}
            \mu_i^{\frac{n-2k}{2}}\Bub_{\zm[i]} & \text{if } n<4k\\
            \mu_i^{\frac{n-2k}{2}}\Bub_{\zm[i]} + \bpr{\frac{\mu_i+\dg{z_i,x}}{\mu_i}}^{-2k+1} & \text{if } n=4k\\
            \mu_i^{\frac{n-2k}{2}}\Bub_{\zm[i]} + \bpr{\frac{\mu_i +\dg{z_i,x}}{\mu_i}}^{-2k} & \text{if } n>4k.
        \end{cases}
    \end{multline}
    When $n\geq 4k$, remark that we have $n\geq 2k+2$, we use the estimate \eqref{tmp:2ndix} in the representation formula \eqref{eq:reprforphi}. With Lemma \ref{prop:gir2}, we compute that for all $x\in M$,
    \begin{multline}\label{tmp:2ndx}
        \intM{\Gga[i](x,y)\TBub[i]^{\deu-2}(y)F_{\zma[i]}(y)\Bub_{\zm[i]}(y)}(y)\\
        \begin{aligned}
            &\leq C \a_i^{\sigma_{n,k}}\mu_i^{2\sigma_{n,k}} \mu_i^{\frac{n-2k}{2}} (\mu_i + \dg{z_i,x})^{2k-n} \Psi\big(\sqai\dg{z_i,x}\big)\\
            &\leq CF_{\zma[i]}(x)\Bub_{\zm[i]}(x),
        \end{aligned} 
\end{multline}
    using \eqref{eq:raPsitoFa}. We also obtain with Lemma \ref{prop:gir2} that
    \begin{multline}\label{tmp:2ndxx}
        \intM{\Gga[i](x,y)\TBub[i]^{\deu-2}(y)\Psi\big(\sqai \dg{z_i,y}\big)\Bub_{\zm[i]}(y)}(y)\\ \leq C\Bub_{\zm[i]}(x) \Psi\big(\sqai\dg{z_i,x}\big).
    \end{multline}
    For the last term coming from \eqref{tmp:2ndix}, we have to consider several cases, as in the proof of Lemma \ref{prop:1stestphi}. Using Lemma \ref{prop:gir2}, we obtain the following:
    \begin{itemize}
        \item When $4k\leq n<6k$, we let $\theta = 1$ if $n=4k$ and $\theta =0 $ if $n>4k$. Then
            \begin{multline}\label{tmp:2ndxxx}
                \intM{\Gga[i](x,y)\TBub[i]^{\deu-2}(y)\bpr{\frac{\mu_i+\dg{z_i,y}}{\mu_i}}^{-2k+\theta}\Psi\big(\sqai\dg{z_i,x}\big)}(y)\\
                \begin{aligned}
                    &\leq C\mu_i^{n-2k}(\mu_i+\dg{z_i,x})^{2k-n}\Psi\big(\sqai\dg{z_i,x}\big) \\
                    &\leq C\mu_i^{\frac{n-2k}{2}}\Bub_{\zm[i]}(x)\Psi\big(\sqai\dg{z_i,x}\big).    
                \end{aligned}    
            \end{multline}
        \item When $n=6k$, for all $\Tilde{\theta}>0$, there exists $C_{\Tilde{\theta}}>0$ such that 
            \begin{multline}\label{tmp:2ndimp1}
                \intM{\Gga[i](x,y)\TBub[i]^{\deu-2}(y)\bpr{\frac{\mu_i+\dg{z_i,y}}{\mu_i}}^{-2k}\Psi\big(\sqai\dg{z_i,x}\big)}(y)\\
                    \leq C_{\Tilde{\theta}} \bpr{\frac{\mu_i + \dg{z_i,x}}{\mu_i}}^{-4k+\Tilde{\theta}} \Psi\big(\sqai \dg{z_i,x}\big)
            \end{multline}
            using the properties of the function $t\mapsto \log t$ in $(0,+\infty)$.
        \item When $n>6k$,
        \begin{multline}\label{tmp:2ndimp2}
            \intM{\Gga[i](x,y)\TBub[i]^{\deu-2}(y)\bpr{\frac{\mu_i+\dg{z_i,y}}{\mu_i}}^{-2k}\Psi\big(\sqai\dg{z_i,x}\big)}(y)\\
                \leq C \bpr{\frac{\mu_i + \dg{z_i,x}}{\mu_i}}^{-4k} \Psi\big(\sqai \dg{z_i,x}\big).
        \end{multline}
    \end{itemize}
    In the case $n<6k$, plugging \eqref{tmp:2ndix} in \eqref{eq:reprforphi}, putting \eqref{tmp:2ndx}, \eqref{tmp:2ndxx} and \eqref{tmp:2ndxxx} together, we have proven  \eqref{eq:2ndestphi}. When $n\geq 6k$, we repeat the arguments, using now the improved estimates \eqref{tmp:2ndimp1} and \eqref{tmp:2ndimp2} in \eqref{eq:reprforphi}. After a finite amount of iterations using Lemma \ref{prop:gir2}, we obtain \eqref{eq:2ndestphi}.
\end{proof}

We now obtain an estimate on $\inorm{\phi_i}$, showing that the second term in \eqref{eq:2ndestphi} can be absorbed by the first.
\begin{proposition}\label{prop:estphiinfn}
    Let $\phi_i \in \Keri^\perp$ be the unique solutions to \eqref{eq:linphiR} in Theorem \ref{prop:linest}, and let $\rad$ be as in \eqref{def:rad}. There exists $C>0$ and $i_0>0$ such that for all $i\geq i_0$,
    \[  \mu_i^{\frac{n-2k}{2}}\inorm{\phi_i} \leq C \rad(\sqai \mu_i).
        \]
\end{proposition}
The proof will proceed by contradiction, we will assume 
\begin{equation}\label{eq:contrainfn}
    \frac{\mu_i^{\frac{n-2k}{2}}\inorm{\phi_i}}{\rad(\sqai\mu_i)} \to \infty \qquad \text{as } i \to \infty.
\end{equation}
Let $\inj/2 < \varrho < \inj$ be as in definition \ref{def:psizm}. If $\phi_i \not\equiv 0$, we define for all $y \in \Bal{0}{\varrho/\mu_i} \sub \R^n$,
\begin{equation}\label{def:psii}
    \psi_i(y) := \frac{\phi_i(\exp_{z_i}(\mu_i y))}{\inorm{\phi_i}},
\end{equation}
then $\psi_i \in C^0(\Bal{0}{\varrho/\mu_i})$. 
\par Before proving Proposition \ref{prop:estphiinfn}, we prove a few intermediate results.

\begin{lemma}\label{prop:dgzixi}
    Assume that $\phi_i$ satisfies \eqref{eq:contrainfn}, and let $\Tx_i \in M$ be such that $\abs{\phi_i(\Tx_i)} = \inorm{\phi_i}$. There exists $R_0>0$ and $i_0>0$ such that for all $i\geq i_0$, 
    \begin{equation}\label{eq:dgzixi}
        \dg{z_i,\Tx_i} \leq R_0\mu_i.
    \end{equation}
\end{lemma}
\begin{proof}
    Remark that
    \begin{equation}\label{eq:estFB}
        F_{\zma}(x)\Bub_{\zm}(x) \leq C \mu^{-\frac{n-2k}{2}} \rad(\sqa\mu) \qquad \forall~x\in M.
    \end{equation}
    We evaluate the estimate \eqref{eq:2ndestphi} at $\Tx_i \in M$, and obtain with \eqref{eq:estFB} that
    \begin{equation}\label{tmp:dgzx}
        \abs{\phi_i(\Tx_i)} \leq C \mu_i^{-\frac{n-2k}{2}} \rad(\sqai\mu_i) + C \abs{\phi_i(\Tx_i)} \bpr{\frac{\mu_i}{\mu_i+\dg{z_i,x}}}^{n-2k}.
        \end{equation} 
        With the assumption \eqref{eq:contrainfn}, we see that $\mu_i^{-\frac{n-2k}{2}}\rad(\sqai\mu_i) = o\big(\abs{\phi_i(\Tx_i)}\big)$. If we assume $\frac{\dg{z_i,\Tx_i}}{\mu_i} \to \infty$, we reach the contradiction $\abs{\phi_i(\Tx_i)} \leq o\big(\abs{\phi_i(\Tx_i)}\big)$, so that we have \eqref{eq:dgzixi}.
\end{proof}
\begin{lemma}\label{prop:bdpsii}
    Assume that $\phi_i$ satisfies \eqref{eq:contrainfn}, and let $\psi_i$ be as defined in \eqref{def:psii}. There exists $C>0$ such that the following holds. For all $R\geq 1$, 
    \begin{equation}\label{eq:estpsii}
        \abs{\psi_i(y)} \leq C (1+o(1)) (1+\abs{y})^{2k-n} \qquad \forall~y \in \Bal{0}{R} \sub \R^n,
    \end{equation}
    where $o(1) \to 0$ as $i\to \infty$.
\end{lemma}
\begin{proof}
    Fix $R>0$, and fix $i_0>0$ such that $R\mu_i < 1/\sqai$ for all $i\geq i_0$, this is possible since $\a_i\mu_i^2 \to 0$. Evaluate \eqref{eq:2ndestphi} at $x=\exp_{z_i}(\mu_i y)$ for $y\in\Bal{0}{R}\sub\R^n$, we obtain for all $i\geq i_0$
    \[  \abs{\psi_i(y)} \leq C \bpr{\frac{r_{\zma[i]}(\exp_{z_i}(\mu_i y))}{\inorm{\phi_i}} + \mu_i^{\frac{n-2k}{2}}}\Bub_{\zm[i]}(\exp_{z_i}(\mu_i y)).
        \]
    Now using that $r_{\zma[i]}(\exp_{z_i}(\mu_i y)) \leq CR^2 \rad(\sqai\mu_i)$ for all $\abs{y}<R$ and $R\geq 1$, we get using \eqref{eq:contrainfn} that
    \[  \abs{\psi_i(y)} \leq C\bpr{R^2\frac{\rad(\sqai\mu_i)\mu_i^{-\frac{n-2k}{2}}}{\inorm{\phi_i}} + 1} \Bub(y) \leq C(1+ o(1)) (1+\abs{y})^{2k-n}.
        \]
\end{proof}
\begin{lemma}\label{prop:bdDgapsii}
    Assume that $\phi_i$ satisfies \eqref{eq:contrainfn}, and let $\psi_i$ be as defined in \eqref{def:psii}. For all $R\geq 1$, there exists $C_R>0$ and $i_0\geq 0$ such that for all $i\geq i_0$,
    \[  \Lnorm[(\Bal{0}{R})]{\infty}{(\Dg[\Tg_i]+\a_i\mu_i^2)^k \psi_i} \leq C_R,
        \]
    where $\Tg_i := \exp^*_{z_i}g(\mu_i \cdot)$.
\end{lemma}
\begin{proof}
    Fix $R>0$ and let $i_0>0$ be such that $R\mu_i \leq 1/\sqai$ for all $i\geq i_0$. By definition of $\psi_i$, we compute, for $i\geq i_0$,
    \[  \abs{(\Dg[\Tg_i]+\a_i\mu_i^2)^k \psi_i(y)} = \frac{\mu_i^{2k}}{\inorm{\phi_i}} \abs{(\Dg+\a_i)^k \phi_i(\exp_{z_i}(\mu_i y))}
        \]
    for all $y\in \Bal{0}{R}$. We use equation \eqref{eq:linphiR} satisfied by $\phi_i$, observe that using \eqref{eq:estRiprop}, and with \eqref{eq:contrainfn},
    \begin{align}
        \notag &\mu_i^{2k}\TBub[i]^{\deu-2}(\exp_{z_i}(\mu_i y)) = \mu_i^{2k} \Bub_{\zm[i]}^{\deu-2}(\exp_{z_i}(\mu_i y)) \leq C(1+\abs{y})^{-4k} \leq C\\
        \label{eq:estsupRi} &\frac{\mu_i^{2k}}{\inorm{\phi_i}} \abs{R_i(\exp_{z_i}(\mu_i y))} \leq C_R \frac{\rad(\sqai\mu_i)}{\inorm{\phi_i}}\mu_i^{-\frac{n-2k}{2}} (1+\abs{y})^{-n} = o(1)
    \end{align}
    for all $y\in \Bal{0}{R}$. Moreover, using Proposition \ref{prop:estRa}, we have 
    \begin{multline*}  \frac{\mu_i^{2k}}{\inorm{\phi_i}} \abs{\sum_{j=0}^n \lambda_i^j (\Dg+\a_i)^k \Zed[i]{j}(\exp_{z_i}(\mu_i y))}\\ \leq (\deu-1)\frac{\mu_i^{2k}}{\inorm{\phi_i}}\sum_{j=0}^n \tabs{\lambda_i^j} \abs{\Zed[i]{j}(\exp_{z_i}(\mu_i y))} \Bub_{\zm[i]}^{\deu-2}(\exp_{z_i}(\mu_i y))\\ + \frac{\mu_i^{2k}}{\inorm{\phi_i}}\bigg(\sum_{j=0}^n \tabs{\lambda^j_i}\bigg)\abs{\err[i]{\Ze{j}}(\exp_{z_i}(\mu_i y))}.
        \end{multline*}
    On the one hand, since $\err[i]{\Ze{j}}$ has bounds similar to those of $R_i$, we obtain as in \eqref{eq:estsupRi} that 
    \[  \frac{\mu_i^{2k}}{\inorm{\phi_i}} \abs{\err[i]{\Ze{j}}(\exp_{z_i}(\mu_i y))} \leq C_R \frac{\a_i\mu_i^2}{\inorm{\phi_i}}\mu_i^{-\frac{n-2k}{2}}= o(1) \qquad \text{as } i \to \infty,
        \]
    for all $y \in \Bal{0}{R}$. On the other hand, we use Lemma \ref{prop:estlami} and \eqref{eq:basicZB}, and get for $j=0,\ldots n$, 
    \begin{multline*}
        \frac{\mu_i^{2k}}{\inorm{\phi_i}}\tabs{\lambda_i^j} \abs{\Zed[i]{j}(\exp_{z_i}(\mu_i y))}\Bub_{\zm[i]}^{\deu-2} (\exp_{z_i}(\mu_i y))\\
        \begin{aligned}
            &\leq C_R \frac{\rad(\sqai \mu_i) \mu_i^{-\frac{n-2k}{2}}}{\inorm{\phi_i}}(1+\abs{y})^{-(n+2k)} + C (1+\abs{y})^{-(n+2k)}\\
            &\leq C_R
        \end{aligned}
    \end{multline*}
    since that \eqref{eq:contrainfn} holds. Thus, we obtain in the end
    \[  \abs{(\Dg[\Tg_i]+\a_i\mu_i^2)^k \psi_i(y)} \leq C_R \qquad \forall y \in \Bal{0}{R},
        \]
    for $i \geq i_0$.
\end{proof}
\begin{proposition}\label{prop:eqpsiinf}
    Assume that $\phi_i$ satisfies \eqref{eq:contrainfn}, and let $\psi_i$ be as defined in \eqref{def:psii}. Then, there exists $\psi_\infty$ such that $\psi_i \to \psi_\infty$ in $C^\infty_{loc}(\R^n)$, 
    up to a subsequence. Moreover, $\psi_\infty \in C^\infty(\R^n) \cap \hSob(\R^n)$ and satisfies
    \begin{equation}\label{eq:psiinf}
        \Dg[\xi]^k \psi_\infty = (\deu-1)\Bub^{\deu-2}\psi_\infty \qquad \text{in } \R^n.
    \end{equation}
\end{proposition}
\begin{proof}
    Let $R\geq 1$ be fixed. Since $\Tg_i \to \xi$ in $C^\infty_{loc}(\R^n)$ and $\a_i \mu_i^2 \to 0$, as $i\to \infty$, the operator $(\Dg[\Tg_i]+\a_i\mu_i^2)^k$ is an elliptic operator with bounded coefficients, uniformly with respect to $i$. Thus, by Lemma \ref{prop:bdDgapsii} and by standard elliptic estimates, the sequence $(\psi_i)_i$ is bounded in $C^{2k-1,\beta}(\cBal{0}{R})$ for $\beta \in (0,1)$. By compactness of the embedding $C^{2k-1,\beta}(\cBal{0}{R})\sub C^0(\cBal{0}{R})$, there exists $\psi_\infty \in C^0(\cBal{0}{R})$ such that $\psi_i \to \psi_\infty$ in $C^0(\cBal{0}{R})$, up to a subsequence. Estimate \eqref{eq:estpsii} then passes to the limit, and for all fixed $R>0$, we have 
    \begin{equation}\label{eq:estpsiinf}
        \abs{\psi_\infty(y)} \leq C (1+\abs{y})^{2k-n} \qquad \forall~y\in\Bal{0}{R},
    \end{equation}
    where $C>0$ does not depend on $R$.
    \par We will now show that $\psi_\infty$ satisfies \eqref{eq:psiinf}. Let $\vartheta \in \Cct(\R^n)$ be such that $\Hnorm{k,2}{\vartheta} = 1$, and let $\vartheta_{\zm[i]}$ be as in \eqref{def:conc}. Test equation \eqref{eq:linphiR} against $\vartheta_{\zm[i]}$: For all $i$ we have
    \begin{multline}\label{eq:phivsthet}
        \frac{\mu_i^{-\frac{n-2k}{2}}}{\inorm{\phi_i}}\Bigg[\sum_{l=0}^k\tbinom{k}{l} \a_i^{k-l}\intM{\vprod{\Dg^{l/2}\phi_i}{\Dg^{l/2}\vartheta_{\zm[i]}}}\\ - (\deu-1)\intM{\TBub[i]^{\deu-2}\phi_i\vartheta_{\zm[i]}} -\intM{R_i \vartheta_{\zm[i]}} \\- \sum_{j=0}^n \lambda_i^j \intM{(\Dg+\a_i)^k \Zed[i]{j}\,\vartheta_{\zm[i]}} \Bigg] = 0.
    \end{multline}
    As in \eqref{tmp:TBphipsi}, let $R_\vartheta>0$ be such that the support of $\vartheta$ is contained in $\Bal{0}{R_\vartheta}$, we obtain up to a subsequence as $i \to \infty$, 
    \begin{equation}\label{tmp:infi1}
        \begin{aligned}
            \frac{\mu_i^{-\frac{n-2k}{2}}}{\inorm{\phi_i}} \intM{\TBub[i]^{\deu-2}\phi_i \vartheta_{\zm[i]}} &=\int_{\Bal{0}{R_\vartheta}} \Bub^{\deu-2}\psi_i \vartheta\, dy + o(1)\\
                &= \int_{\R^n} \Bub^{\deu-2}\psi_\infty \vartheta\, dy + o(1),
        \end{aligned}
    \end{equation}
    since $\psi_i\to \psi_\infty$ in $C^0(\Bal{0}{R_\vartheta})$. Similarly, as in \eqref{tmp:Dgaphipsi}, we now show that 
    \begin{equation}\label{tmp:infi2}
        \frac{\mu_i^{-\frac{n-2k}{2}}}{\inorm{\phi_i}} \sum_{l=0}^k \tbinom{k}{l} \a_i^{k-l}\intM{\vprod{\Dg^{l/2}\phi_i}{\Dg^{l/2}\vartheta_{\zm[i]}}} = \int_{\R^n} \psi_\infty \Dg[\xi]^k \vartheta\, dy + o(1)
    \end{equation} 
    as $i\to \infty$. To prove \eqref{tmp:infi2}, we first observe that by integration by parts, and choosing $R_\vartheta>0$ as above, we have
    \begin{align*}  
        \frac{\mu_i^{-\frac{n-2k}{2}}}{\inorm{\phi_i}} \intM{\vprod{\Dg^{k/2}\phi_i}{\Dg^{k/2}\vartheta_{\zm[i]}}} &= \intM[\Bal{z_i}{R_\vartheta\mu_i}]{\frac{\mu_i^{-\frac{n-2k}{2}}}{\inorm{\phi_i}}\phi_i\, \Dg^k \vartheta_{\zm[i]}}\\
            &= \int_{\Bal{0}{R_\vartheta}} \psi_\infty \Dg[\xi]^k\vartheta\, dy + o(1) \\
            &= \int_{\R^n} \psi_\infty \Dg[\xi]^k \vartheta \,dy + o(1)
        \end{align*}
    as $i\to \infty$. For $l=0,\ldots k-1$, we also have, by Hölder's inequality,
    \begin{multline*}
        \frac{\mu_i^{-\frac{n-2k}{2}}}{\inorm{\phi_i}} \a_i^{k-l}\abs{\intM{\vprod{\Dg^{l/2}\phi_i}{\Dg^{l/2}\vartheta_{\zm[i]}}}}\\
        \begin{aligned}
            &\leq \frac{\mu_i^{-\frac{n-2k}{2}}}{\inorm{\phi_i}} \a_i^{k-l} \bpr{\intM[\Bal{z_i}{R_\vartheta \mu_i}]{\abs{\Dg^{l/2}\phi_i}^2}}^{1/2}\\
            &\qquad\qquad \bpr{\intM[\Bal{z_i}{R_\vartheta\mu_i}]{\abs{\Dg^{l/2}\vartheta_{\zm[i]}}^2}}^{1/2}\\
            &\leq C\frac{\mu_i^{-\frac{n-2k}{2}}}{\inorm{\phi_i}} (\a_i\mu_i^2)^{k-l}\Snorm{k}{\phi_i} \leq C\frac{\rad(\sqai\mu_i)\mu_i^{-\frac{n-2k}{2}}}{\inorm{\phi_i}} = o(1)
        \end{aligned}
    \end{multline*}
    with the same computations as in \eqref{eq:Dglphimu} and \eqref{eq:Dglpsizm}, and where we used \eqref{eq:contrainfn}. This concludes the proof of \eqref{tmp:infi2}. With the same arguments, we similarly obtain that
    \begin{multline}\label{tmp:infi3}
        \frac{\mu_i^{-\frac{n-2k}{2}}}{\inorm{\phi_i}}\sum_{j=0}^n \lambda_i^j \intM{(\Dg+\a_i)^k \Zed[i]{j}\,\vartheta_{\zm[i]}}\\ = \sum_{j=0}^n \frac{\mu_i^{-\frac{n-2k}{2}}}{\inorm{\phi_i}}\lambda_j^i \bpr{\int_{\R^n} \Ze{j}\Dg[\xi]^k \vartheta\, dy + o(1)}.
    \end{multline}
    Now using Lemma \ref{prop:estlami} and \eqref{eq:contrainfn}, we have
    \[  \frac{\mu_i^{-\frac{n-2k}{2}}}{\inorm{\phi_i}}\tabs{\lambda_i^j} \leq C \frac{\rad(\sqai\mu_i)\mu_i^{-\frac{n-2k}{2}}}{\inorm{\phi_i}} + C \leq C 
        \]
    for $j=0,\ldots n$, and thus there exists $\lambda^j_\infty \in \R$ such that, up to a subsequence, 
    \begin{equation}\label{tmp:infi3b}
        \frac{\mu_i^{-\frac{n-2k}{2}}}{\inorm{\phi_i}} \lambda_i^j \xto{i\to \infty} \lambda_\infty^j \qquad j=0,\ldots n.
    \end{equation}
    Finally, we have by a change of variables
    \begin{equation}\label{tmp:infi4}
        \begin{aligned}
            \frac{\mu_i^{-\frac{n-2k}{2}}}{\inorm{\phi_i}} \abs{\intM{R_i \vartheta_{\zm[i]}}} &\leq C\frac{\mu_i^{2k}}{\inorm{\phi_i}}\int_{\Bal{0}{R_\vartheta}}\abs{R_i(\exp_{z_i}(\mu_i y))} \abs{\vartheta(y)}dy\\
                &\leq o(1) \bpr{\int_{\Bal{0}{R_\vartheta}}\abs{\vartheta}^{\deu}dy}^{1/\deu},
        \end{aligned}
    \end{equation}
    where we used Hölder's inequality and \eqref{eq:estsupRi}. Thus, for all $\vartheta \in \Cct(\R^n)$, putting \eqref{tmp:infi1}, \eqref{tmp:infi2}, \eqref{tmp:infi3}, \eqref{tmp:infi3b} and \eqref{tmp:infi4} together in \eqref{eq:phivsthet}, and letting $i\to \infty$, we have showed that 
    \[  \int_{\R^n} \psi_\infty \Dg[\xi]^k \vartheta\, dy - (\deu-1) \int_{\R^n}\Bub^{\deu-2}\psi_\infty\vartheta\, dy - \sum_{j=0}^n \lambda_\infty^j \int_{\R^n}\Ze{j}\Dg[\xi]^k \vartheta\, dy = 0.
        \]
    In other words, $\psi_\infty \in C^0(\R^n)$ satisfies
    \begin{equation}\label{eq:1psiinf}  \Dg[\xi]^k \psi_\infty -(\deu-1)\psi_\infty \Bub^{\deu-2} = \sum_{j=0}^n \lambda_\infty^j \Dg[\xi]^k \Ze{j}
        \end{equation}
    in the distributional sense in $\R^n$. Observe that $\psi_\infty \in L^\infty(\R^n)$ thanks to \eqref{eq:estpsiinf}, so that by a standard bootstrap argument $\psi_\infty \in C^\infty(\R^n)$, and $\psi_\infty$ is a classical solution to \eqref{eq:1psiinf}. We use a representation formula for $\psi_\infty$, recalling that the Green's function for the poly-Laplacian operator $\Dg[\xi]^k$ in $\R^n$ is $\operatorname{G}_{\xi,0}(x,y) := c_{n,k}\abs{x-y}^{2k-n}$, where $c_{n,k}$ is an explicit constant (see for instance \cite{GazGruSw10}). 
    For $l=0,\ldots 2k-1$, using \eqref{eq:estpsiinf}, \eqref{def:bub} and \eqref{eq:estdZ}, we obtain with a standard Giraud's Lemma 
    \[  \abs{\nabla^l \psi_\infty(y)} \leq C (1+\abs{y})^{2k-n-l}
        \]
    for all $y \in \R^n$. Thus, we see that $\psi_\infty \in \hSob(\R^n)$. Now integrating \eqref{eq:1psiinf} against $\Ze{j_0}$ for some $j_0\in \{0,\ldots n\}$, we obtain 
    \begin{multline}\label{tmp:intpp}
        \int_{\R^n} \vprod[\xi]{\Dg[\xi]^{k/2}\psi_\infty}{\Dg[\xi]^{k/2}\Ze{j_0}}dy -(\deu-1)\int_{\R^n} \Bub^{\deu-2}\psi_\infty \Ze{j_0}dy\\ = \sum_{j=0}^n \lambda_\infty^j \int_{\R^n} \vprod[\xi]{\Dg[\xi]^{k/2}\Ze{j}}{\Dg[\xi]^{k/2}\Ze{j_0}}dy
        \end{multline}
    by integration by parts. Since the $\Ze{j}$  satisfy \eqref{eq:linbub} and form an orthogonal basis of $\Ker[]$ in $\hSob(\R^n)$, \eqref{tmp:intpp} gives $\lambda_\infty^{j_0}= 0$ for $j_0=0,\ldots n$. Now $\psi_\infty \in C^\infty(\R^n)\cap \hSob(\R^n)$ is a classical solution to \eqref{eq:psiinf}, and this concludes the proof.
\end{proof}

We are now in position to prove Proposition \ref{prop:estphiinfn}.
\begin{proof}[Proof of Proposition \ref{prop:estphiinfn}]
    Let us assume by contradiction that $\phi_i \not\equiv 0$ satisfies \eqref{eq:contrainfn}, then Lemmas \ref{prop:bdpsii}, \ref{prop:bdDgapsii}, and Proposition \ref{prop:eqpsiinf} hold. There exists $\psi_\infty\in C^\infty(\R^n) \cap \hSob(\R^n)$ such that, up to a subsequence, $\psi_i \to \psi_\infty$ in $C^0_{loc}(\R^n)$, 
    where $\psi_i$ is as defined in \eqref{def:psii}. Moreover, $\psi_\infty$ satisfies \eqref{eq:psiinf} and thus $\psi_\infty \in \Ker[]$. Let now $R>R_0$ where $R_0$ is given by Lemma \ref{prop:dgzixi}. Thanks to \eqref{eq:dgzixi}, there exists $y_\infty \in \Bal{0}{R}\sub \R^n$ such that, up to a subsequence, 
    \[  \tfrac{1}{\mu_i}\exp_{z_i}^{-1}(\Tx_i) \to y_\infty,
        \] 
    where $\Tx_i \in M$ is the point where $\abs{\phi_i}$ reaches its maximum. Now by definition of $\psi_i$, we have $\abs{\psi_i\big(\tfrac{1}{\mu_i}\exp_{z_i}^{-1}(\Tx_i)\big)} = 1$ for all $i$, and 
    \[  \psi_i\big(\tfrac{1}{\mu_i}\exp_{z_i}^{-1}(\Tx_i)\big) \xto{i\to \infty} \psi_\infty(y_\infty)
        \]
    up to a subsequence. Thus, by continuity of $\psi_\infty$ we have 
    \begin{equation}\label{eq:psiinfnzer}
        \psi_\infty \not\equiv 0.
    \end{equation}
    \par We now show that $\psi_\infty \in \Ker[]^\perp$, since $\psi_\infty \in \Ker[]$, this will imply that $\psi_\infty \equiv 0$, which is a contradiction. Since $\phi_i \in \Keri^\perp$ for all $i$, we have
    \[  0 =\frac{\mu_i^{-\frac{n-2k}{2}}}{\inorm{\phi_i}} \vprod[\Sob(M)]{\phi_i}{\Zed[i]{j}}= \sum_{l=0}^k \intM{\frac{\mu_i^{-\frac{n-2k}{2}}}{\inorm{\phi_i}} \phi_i \Dg^l \Zed[i]{j}} \qquad j=0,\ldots n
        \]
    by integration by parts. For $l=0,\ldots k-1$, $j=0,\ldots n$, see that
    \begin{multline*}  
        \abs{\intM{\frac{\mu_i^{-\frac{n-2k}{2}}}{\inorm{\phi_i}} \phi_i \Dg^{l} \Zed[i]{j}}} \leq \intM[\Bal{z_i}{1/\sqai}]{\mu_i^{-\frac{n-2k}{2}} \abs{\Dg^l \Zed[i]{j}}}\\ + \intM[M\setminus\Bal{z_i}{1/\sqai}]{\mu_i^{-\frac{n-2k}{2}}\abs{\Dg^l \Zed[i]{j}}}.
        \end{multline*}
    Using \eqref{eq:estdifXsqa}, we have
    \begin{align*}
        \intM[\Bal{z_i}{1/\sqai}]{\mu_i^{-\frac{n-2k}{2}} \abs{\Dg^l \Zed[i]{j}}} &\leq C \mu_i^{2(k-l)}\int_{\Bal{0}{\frac{1}{\sqai\mu_i}}} (1+\abs{y})^{-n-2k+2l} dy\\
            &\leq C\a_i^{-(k-l)},
    \end{align*}
    while using \eqref{eq:estdifXexp}, we have
    \begin{multline}\label{tmp:muDgZbd}
        \intM[M\setminus\Bal{z_i}{1/\sqai}]{\mu_i^{-\frac{n-2k}{2}}\abs{\Dg^l \Zed[i]{j}}}\\
        \begin{aligned}
            &\leq C\a_i^{-(k-l)} \int_{\Bal{0}{\varrho\sqai}\setminus\Bal{0}{1}}\abs{y}^{2k-n}e^{-\abs{y}/2}dy + C \a_i^l e^{-\sqai\varrho/2}\\
            &\leq C\a_i^{-(k-l)}.
        \end{aligned}    
    \end{multline}
    Similarly, for $l=k$, using \eqref{eq:2ndestphi} together with \eqref{eq:estFB}, we obtain
    \begin{multline*}
        \intM[M\setminus\Bal{z_i}{1/\sqai}]{\frac{\mu_i^{-\frac{n-2k}{2}}}{\inorm{\phi_i}}\abs{\phi_i}\abs{\Dg^k \Zed[i]{j}}}\\ \leq C\intM[M\setminus\Bal{z_i}{1/\sqai}]{\frac{\mu_i^{-(n-2k)}\rad(\sqai\mu_i)}{\inorm{\phi_i}} \abs{\Dg^k \Zed[i]{j}}}\\ + C\intM[M\setminus\Bal{z_i}{1/\sqai}]{\mu_i^{\frac{n-2k}{2}}\dg{z_i,x}^{2k-n}e^{-\sqai \dg{z_i,x}/2}\abs{\Dg^k \Zed[i]{j}}}.
    \end{multline*}
    We compute with the same argument as in \eqref{tmp:muDgZbd}
    \[  \intM[M\setminus\Bal{z_i}{1/\sqai}]{\frac{\mu_i^{-(n-2k)}\rad(\sqai\mu_i)}{\inorm{\phi_i}} \abs{\Dg^k \Zed[i]{j}}} \leq C \frac{\mu_i^{-\frac{n-2k}{2}}\rad(\sqai\mu_i)}{\inorm{\phi_i}} = o(1)
        \]
    and 
    \begin{equation*}
        \intM[M\setminus\Bal{z_i}{1/\sqai}]{\mu_i^{\frac{n-2k}{2}}\dg{z_i,x}^{2k-n}e^{-\sqai \dg{z_i,x}/2}\abs{\Dg^k \Zed[i]{j}}} \leq C(\a_i\mu_i^2)^{\frac{n-2k}{2}}
    \end{equation*}
    using \eqref{eq:estdifXexp}. Therefore, we have showed that 
    \begin{equation}\label{tmp:vprodphiZ}
        0 = \intM[\Bal{z_i}{1/\sqai}]{\frac{\mu_i^{-\frac{n-2k}{2}}}{\inorm{\phi_i}} \phi_i \Dg^k \Zed[i]{j}}+ o(1) \qquad j=0,\ldots n,
    \end{equation}
    as $i\to \infty$, since $\a_i\to \infty$, $\a_i\mu_i^2 \to 0$. Fix some $R>0$, we have 
    \begin{equation}\label{tmp:kterm}
    \begin{aligned}
        \frac{\mu_i^{-\frac{n-2k}{2}}}{\inorm{\phi_i}}\intM[\Bal{z_i}{R\mu_i}]{\phi_i \Dg^k \Zed[i]{j}} &= \int_{\Bal{0}{R}} \psi_i \Dg[\xi]^k \Ze{j}dy + o(1)\\
            &= \int_{\Bal{0}{R}} \psi_\infty \Dg[\xi]^k \Ze{j}dy + o(1),
    \end{aligned}
    \end{equation}
    up to a subsequence.
    See that, using \eqref{eq:estpsiinf} and \eqref{eq:estdZ}, by dominated convergence we have 
    \begin{equation}\label{tmp:ktermb}
        \lim_{R\to \infty}\int_{\Bal{0}{R}}\psi_\infty \Dg[\xi]^k \Ze{j} dy = \int_{\R^n}\psi_\infty \Dg[\xi]^k \Ze{j}dy.
    \end{equation}
    Now as in the proof of Proposition \ref{prop:estRa}, we use \eqref{eq:estdifXsqa} to get for all $y \in \Bal{0}{\varrho/\mu_i}$,
    \begin{multline*}
        \Dg^k \Zed[i]{j}(\exp_{z_i}(\mu_i y)) = \mu_i^{-\frac{n+2k}{2}} \Dg[\xi]^k \Ze{j}(y) + \bigO\bpr{\mu_i^2 \mu_i^{-\frac{n+2k}{2}}(1+\abs{y})^{2-n}}\\
            = \mu_i^{-\frac{n+2k}{2}}(\deu-1)\Ze{j}(y)\Bub^{\deu-2}(y) + \bigO\bpr{\mu_i^2 \mu_i^{-\frac{n+2k}{2}} (1+\abs{y})^{2-n}},
    \end{multline*}
    so that 
    \begin{multline*}
        \frac{\mu_i^{-\frac{n-2k}{2}}}{\inorm{\phi_i}}\abs{\intM[\Bal{z_i}{1/\sqai}\setminus\Bal{z_i}{R\mu_i}]{\phi_i \Dg^k \Zed[i]{j}}}\\
        \begin{aligned}
            &\leq \mu_i^{-\frac{n-2k}{2}}\intM[\Bal{z_i}{1/\sqai}\setminus\Bal{z_i}{R\mu_i}]{\abs{\Dg^k \Zed[i]{j}}}\\
            &\leq C \int_{\Bal{0}{\frac{1}{\sqai\mu_i}}\setminus\Bal{0}{R}}\abs{\Ze{j}}\Bub^{\deu-2}dy + C\mu_i^2 \int_{\Bal{0}{\frac{1}{\sqai\mu_i}}}(1+\abs{y})^{2-n}dy.
        \end{aligned} 
    \end{multline*}
    Direct computations using \eqref{eq:controlB} and \eqref{eq:estdZ} give 
    \[  \int_{\Bal{0}{\frac{1}{\sqai\mu_i}}\setminus\Bal{0}{R}}\abs{\Ze{j}}\Bub^{\deu-2}dy \leq C \int_{\R^n\setminus \Bal{0}{R}} (1+\abs{y})^{-(n+2k)} dy = \epsilon(R),
        \]
    where $\epsilon(R) \to 0$ as $R\to \infty$. Finally, see that 
    \[  \mu_i^2 \int_{\Bal{0}{\frac{1}{\sqai\mu_i}}}(1+\abs{y})^{2-n}dy \leq C\a_i^{-1}.
        \]
    In the end, we have thus showed that 
    \begin{equation}\label{tmp:ktermc}
        \frac{\mu_i^{\frac{n-2k}{2}}}{\inorm{\phi_i}}\abs{\intM[\Bal{z_i}{1/\sqai}\setminus\Bal{z_i}{R\mu_i}]{\phi_i \Dg^k \Zed[i]{j}}} \leq \epsilon(R) + o(1) 
    \end{equation}
    as $i\to \infty$.
    Putting \eqref{tmp:kterm}, \eqref{tmp:ktermb}, and \eqref{tmp:ktermc} in \eqref{tmp:vprodphiZ}, letting first $i\to \infty$, up to a subsequence, and then $R\to \infty$, we obtain
    \[  0 = \int_{\R^n} \psi_\infty \Dg[\xi]^k \Ze{j}dy = \int_{\R^n} \vprod[\xi]{\Dg[\xi]^{k/2}\psi_\infty}{\Dg[\xi]^{k/2}\Ze{j}}dy \qquad j= 0\ldots n
        \]
    by integration by parts, since $\psi_\infty \in \hSob(\R^n)$. This gives $\psi_\infty \in \Ker[]^\perp$ and thus $\psi_\infty =0$ which is a contradiction with \eqref{eq:psiinfnzer}.
\end{proof}

We now conclude the proof of Theorem \ref{prop:linest}.
\begin{proof}[Proof of Theorem \ref{prop:linest}]
    Let $\phi_i \in \Keri^\perp$ be the unique solution to \eqref{eq:linphiR}, for $i\geq i_0$ big enough. Then by a standard bootstrap argument, it follows that $\phi_i \in C^0(M)$. Proposition \ref{prop:2ndestphi} together with Proposition \ref{prop:estphiinfn} gives that for all $x\in M$,
    \[  \abs{\phi_i(x)} \leq CF_{\zma[i]}(x) \Bub_{\zm[i]}(x) + C \rad(\sqai\mu_i) \Bub_{\zm[i]}(x)\Psi\big(\sqai \dg{z_i,x}\big),
        \]
    where $F_{\zma[i]}$ is defined in \eqref{def:Fa}, $\rad$ in \eqref{def:rad} and $\Psi$ in \eqref{def:Psi}. Now using \eqref{eq:raPsitoFa}, we obtain \eqref{eq:estphi}.
\end{proof}
\begin{remark}
    A precise inspection of the proof shows that an exponential decay of the form $e^{-\sqai\dg{z_i,x}(1-\epsilon)}$ for $\phi_i$ can be obtained for all $\epsilon\in (0,1)$, without changing the arguments. Similarly, for all $\theta \in (0,1)$, we can set when $k=1, n\geq 2k+2$ or $k\geq 2, n=2k+2$,
    \[  r_{\zma}(x) = \a^{1-\theta} (\mu+\dg{z,x})^{2-2\theta}
        \]
    and still obtain \eqref{eq:estphi} without changing the arguments. Thus, this specific choice of $\epsilon,\theta$ is purely arbitrary, and as we will see later, it will be sufficient to conclude in section \ref{sec:decom}.
\end{remark}
\subsection{Non-linear procedure with explicit pointwise bounds}
In this subsection, we refine the non-linear arguments of Proposition \ref{prop:uniqHk}. We prove here the uniqueness of a solution $u_{\zma} = \TBub + \phi_{\zma}$ to \eqref{eq:critbub}, up to terms that belong to $\Ker$, when $\phi_{\zma}$ is small in some weighted space, with precise pointwise estimates. In a second step, we also show pointwise estimates on the derivatives of $\phi_{\zma}$.
\begin{proposition}\label{prop:uniqest}
    Let $(\tau_i)_i$, $(\a_i)_i$ be sequences of positive numbers such that $\a_i\to \infty$, $\tau_i \to 0$ as $i\to \infty$, and let $\pa_i = (\zm[i])\in \param[i]$ for all $i$. There exists a constant $\Lambda >0$ and $i_0 >0$ such that, for all $i\geq i_0$, the equation
    \begin{equation}\label{eq:Vphi}
        \proKi\bsq{\TBub[i]+ \phi - (\Dg+\a_i)^{-k}\bpr{\TBub[i]+ \phi}_+^{\deu-1}} = 0
    \end{equation} 
    has a unique solution $\phi_i$ in 
    \[  \Keri^\perp \cap \big\{\phi \in C^0(M) \such \abs{\phi(x)} \leq \Lambda F_{\zma[i]}(x)\Bub_{\zm[i]}(x)  \quad \forall~x\in M\big\},
        \]
    where $\proKi$ is the projection in $\Sob(M)$ onto 
    the orthogonal to $\Ker$ as defined in \eqref{def:kerzm}, and $F_{\zma[i]}$ is as in \eqref{def:Fa}. Moreover, this solution satisfies
    \[  \Snorm{k}{\phi_i} = o(1) \qquad \text{as } i\to \infty.
        \]
\end{proposition}
The proof follows again from a fixed-point argument, but this time in strong weighted spaces. We follow the strategy of proof of \cite[Proposition 4.2]{Pre24}, where a similar result is obtained in the case $k=1$ and for bounded coefficients. Instances where results of this type have been obtained in the case $k=1$ and with bounded coefficients can be found in \cite{Pre18,Pre22}.  
\begin{proof}
    Using Proposition \ref{prop:estRa}, there exists $C_1 >0$ and $i_0>0$ such that for all $x\in M$, for all $i\geq i_0$,
    \begin{equation}\label{tmp:estRa1}  
        \abs{(\Dg+\a_i)^k \TBub[i](x) -\TBub[i]^{\deu-1}(x)} \leq C_1 \Bub_{\zm[i]}(x) \Phi_{\zma[i]}(x),
        \end{equation}
    where we write
    \[   \Phi_{\zma[i]}(x):= \begin{cases}
        \a_i^{\sigma_{n,k}}(\mu_i +\dg{z_i,x})^{2\sigma_{n,k}-2k} & \text{when } \sqai\dg{z_i,x} \leq 1\\
        \a_i^k e^{-\sqai \dg{z_i,x}/2} & \text{when } \sqai\dg{z_i,x} \geq 1
    \end{cases}
        \]
    with $\sigma_{n,k}$ as in \eqref{def:snk}. 
    Define the set 
    \begin{multline}\label{def:Si}
        \Set_i := \big\{ \phi \in C^0(M) \such \abs{\phi(x)} \leq 2 C_0C_1 F_{\zma[i]}(x)\Bub_{\zm[i]}(x) \quad \forall~x\in M\\
            \text{and}\quad  \sum_{l=0}^k \intM{\phi \Dg^l \Zed[i]{j}} = 0 \quad \text{for } j=0,\ldots n\big\}, 
    \end{multline}
    where $C_0>0$ is given by Theorem \ref{prop:linest}. Endow $\Set_i$ with the norm 
    \[  \norm{\phi}_{*i} := \Lnorm[(M)]{\infty}{\frac{\phi}{F_{\zma[i]}\Bub_{\zm[i]}}}, \qquad \text{for } \phi \in \Set_i.
        \]
    Since $F_{\zma[i]}\Bub_{\zm[i]} \in L^{\infty}(M)$ is positive, $\Big(\Set_i,\norm{\cdot}_{*i}\Big)$ is a non-empty complete metric space for all $i\geq i_0$. For $\phi \in C^0(M)$, we define as in \eqref{def:EaGa},
    \[  G_{\zma[i]}(\phi) := (\TBub[i]+ \phi)_+^{\deu-1} - \TBub[i]^{\deu-1} -(\deu-1)\TBub[i]^{\deu-2}\phi,
    \]
    and we let
    \begin{equation}\label{def:Raphi}
        R_i(\phi) := -\err[i]{B} + G_{\zma[i]}(\phi).
    \end{equation}
    By Lemma \ref{prop:GaHk} and Proposition \ref{prop:estRa}, $R_i(\phi) \in \Sob[-k](M)$ and thus up to increasing $i_0$ using Proposition \ref{prop:lininv}, for all $i\geq i_0$ we can define $T_i(\phi) \in \Keri^\perp$ as the unique solution in $\Keri^\perp$ to 
    \begin{equation}\label{eq:phiTphi}
        (\Dg+\a_i)^k T_i(\phi) -(\deu-1)\TBub[i]^{\deu-2} T_i(\phi) = R_i(\phi) + \sum_{j=0}^n \lambda_i^j (\Dg+\a_i)^k \Zed[i]{j}
    \end{equation}
    for some unique $\lambda_i^j\in \R$, $j=0,\ldots n$. Then, for $\phi \in C^0(M)$, by standard elliptic theory, $T_i(\phi) \in C^{2k-1,\beta}(M)$ for $\beta \in (0,1)$. We now show that $T_i$ is a contraction on $\Set_i$. 
    \par We use \eqref{eq:estpoint1}, there is a constant $C>0$ such that, for all $i\geq i_0$,
    \[  \abs{G_{\zma[i]}(\phi)(x)} \leq C \abs{\phi(x)} \bpr{\abs{\phi(x)}^{\deu-2} + \TBub[i]^{\deu-2-\theta}(x) \abs{\phi(x)}^\theta} \qquad \forall ~x \in M
        \]
    for some $0< \theta < \min(1,\deu-2)$.
    For $\phi \in \Set_i$, this shows that when $x\in M$ is such that $\sqai\dg{z_i,x} \leq 1$, 
    \begin{equation}\label{tmp:estx}
        \abs{G_{\zma[i]}(\phi)(x)}\leq C(\a_i\mu_i^2)^{\Tilde{\theta}} \Bub_{\zm[i]}(x) \a_i^{\sigma_{n,k}}(\mu_i+\dg{z_i,x})^{2\sigma_{n,k}-2k} 
    \end{equation}
    for some $\Tilde{\theta}>0$, using the definition of $\Set_i$, $F_{\zma[i]}$ in \eqref{def:Fa} and $r_{\zma[i]}$ in \eqref{def:ra}. When $\sqai\dg{z_i,x} \geq 1$, and for $\phi \in \Set_i$, we have 
    \begin{equation}\label{tmp:estxx}
        \abs{G_{\zma[i]}(\phi)(x)} \leq C(\a_i\mu_i^2)^k \Bub_{\zm[i]}(x) \a_i^k e^{-\sqai\dg{z_i,x}/2}.
    \end{equation}
    Thus, with \eqref{tmp:estRa1}, there is a sequence $(\epsilon_i)_i$ with $\epsilon_i \to 0$ as $i\to \infty$, such that 
    \begin{equation}\label{eq:estRiphi}
        \abs{R_i(\phi)(x)} \leq (C_1 +\epsilon_i) \Bub_{\zm[i]}(x)\Phi_{\zma[i]}(x).
    \end{equation}
    Using Theorem \ref{prop:linest} for $R_i(\phi)$, by linearity of \eqref{eq:linphiR}, we obtain
    \[  \abs{T_i(\phi)(x)} \leq C_0(C_1 + \epsilon_i) F_{\zma[i]}(x)\Bub_{\zm[i]}(x) \qquad \forall~x\in M.
        \]
    Up to increasing $i_0$, this shows that $T_i(\Set_i) \subeq \Set_i$. Now let $\phi_1,\phi_2 \in \Set_i$, we have, by definition of $T_i$,
    \begin{multline*}  (\Dg+\a_i)^k \big(T_i(\phi_1)-T_i(\phi_2)\big) - (\deu-1) \TBub[i]^{\deu-2}\big(T_i(\phi_1)-T_i(\phi_2)\big)\\ = \big(G_{\zma[i]}(\phi_1)-G_{\zma[i]}(\phi_2)\big) + \sum_{j=0}^n \Tilde{\lambda}_i^j (\Dg + \a_i)^k \Zed[i]{j}
        \end{multline*}
    for some $\Tilde{\lambda}_i^j\in \R$, $j=0,\ldots n$, and for all $i\geq i_0$. Using \eqref{eq:estpoint2}, we get as before that
    \begin{equation}\label{tmp:estxxx}
        \abs{G_{\zma[i]}(\phi_1)(x)-G_{\zma[i]}(\phi_2)(x)} \leq \epsilon_i \norm{\phi_1-\phi_2}_{*i} \Bub_{\zm[i]}(x) \Phi_{\zma[i]}(x)
    \end{equation}
    for some sequence $(\epsilon_i)_i$ such that $\epsilon_i \to 0$ as $i\to \infty$, using \eqref{tmp:estx} and \eqref{tmp:estxx}. We observe that for $\phi_1,\phi_2 \in \Set_i$, $\norm{\phi_1-\phi_2}_{*i}$ is uniformly bounded, so that by Theorem \ref{prop:linest} we obtain for all $i\geq i_0$,
    \[  \norm{T_i(\phi_1)-T_i(\phi_2)}_{*i} \leq C_0\epsilon_i \norm{\phi_1 - \phi_2}_{*i}.
        \]
    Increasing again $i_0$ so that $C_0\epsilon_i < 1$ for all $i\geq i_0$, $T_i$ is a contraction on $\Set_i$. By Banach's fixed-point Theorem, for all $i\geq i_0$, there exists a unique $\phi_i \in \Set_i$ such that $T_i(\phi_i) = \phi_i$. Using \eqref{eq:phiTphi}, it satisfies
    \[  (\Dg+\a_i)^k \big(\TBub[i]+\phi_i\big) -\big(\TBub[i] + \phi_i\big)_+^{\deu-1}= \sum_{j=0}^{n}\lambda_i^j (\Dg+\a_i)^k \Zed[i]{j}
        \]
    in $M$. Now we have $\phi_i = T_i(\phi_i) \in C^k(M)$, so that by the definition of $\Set_i$ we obtain $\phi_i \in \Keri^\perp$. Finally, using Proposition \ref{prop:lininv}, \eqref{tmp:estxxx} and Lemma \ref{prop:estintRX},
    \[  \Snorm{k}{\phi_i} = \Snorm{k}{T_i(\phi_i)} \leq C \Snorm{-k}{R_i(\phi_i)} = o(1) \qquad \text{as } i \to \infty.
        \]
\end{proof}

\begin{corollary}\label{prop:estdphi}
    Let $(\tau_i)_i$, $(\a_i)_i$ and $(\pa_i)_i$ be sequences as in Proposition \ref{prop:uniqest}, and let $\phi_i$ be the unique solution to \eqref{eq:Vphi} in 
    \[  \Keri \cap \big\{\phi \in C^0(M) \such \abs{\phi(x)} \leq \Lambda F_{\zma[i]}(x)\Bub_{\zm[i]}(x)  \quad \forall~x\in M\big\},
        \]
    for $i \geq i_0$.
    There exists $C>0$ such that for $l=1,\ldots 2k-1$, we have
    \begin{equation}
        (\mu_i+\dg{z_i,x})^l \abs{\nabla_g^l \phi_i(x)}_g \leq C F_{\zma[i]}(x) \Bub_{\zm[i]}(x) \qquad \forall~x\in M.
    \end{equation}  
\end{corollary}
\begin{proof}
    Let $R_i(\phi_i)$ be defined as in \eqref{def:Raphi}, where $\phi_i$ is the unique solution to \eqref{eq:Vphi}. Equivalently, $\phi_i$ satisfies 
    \[  (\Dg+\a_i)^k \bsq{\phi_i - \sum_{j=0}^n \lambda_i^j \Zed[i]{j}} = (\deu-1)\TBub[i]^{\deu-2} \phi_i + R_i(\phi_i),
        \]
    and has the estimate
    \[  \abs{\phi_i(x)} \leq C F_{\zma[i]}(x) \Bub_{\zm[i]}(x) \qquad \forall~x\in M.
        \]  
    As in the proof of Lemma \ref{prop:1stestphi}, we use the representation formula \eqref{eq:reprforphi} and then differentiate it. We use the estimates on $\Gga[i]$, the Green's function for the operator $(\Dg+\a)^k$ in $M$ (see \cite[Proposition 3.9]{Car24}), namely that 
    \[  \abs{\nabla_g^l \Gga[i](x,y)}_g \leq C \dg{x,y}^{2k-n-l} \Psi(\sqai\dg{x,y}) \qquad l=1,\ldots 2k-1,
        \]
    where $\Psi$ is as in \eqref{def:Psi}.
    For $l=1,\ldots 2k-1$, we compute as in \eqref{eq:girRasqa} and \eqref{eq:girRaexp}, differentiating under the integral sign,
    \begin{multline*}
        \intM{\abs{\nabla_g^l \Gga[i](x,y)}_g\abs{R_i(\phi_i)(y)}}(y)\\
        \begin{aligned}
            &\leq C\begin{cases}
            \a_i^{\sigma_{n,k}}(\mu_i + \dg{z_i,x})^{2k-n +2\sigma_{n,k}-l} & \text{when } \sqai \dg{z_i,x} \leq 1\\
            \mu_i^{\frac{n-2k}{2}}\a_i^k \dg{z,x}^{4k-n-l} e^{-\sqai \dg{z_i,x}/2} & \text{when } \sqai \dg{z_i,x} \geq 1
        \end{cases}\\
        &\leq  C(\mu_i+\dg{z_i,x})^{-l} F_{\zma[i]}(x)\Bub_{\zm[i]}(x),
        \end{aligned}
    \end{multline*}
    using \eqref{eq:estRiphi} and Lemma \ref{prop:gir1}. We also compute, using Lemma \ref{prop:gir2} as in \eqref{tmp:2ndx}, that 
    \begin{multline*}  
        \intM{\abs{\nabla_g^l \Gga[i](x,y)}_g\TBub[i]^{\deu-2}(y)\abs{\phi_i(y)}}(y)\\ \begin{aligned}
            &\leq C \mu_i^{\frac{n-2k}{2}} (\mu_i + \dg{z_i,x})^{2k-n-l} \rad(\sqai\mu_i) \Psi\big(\sqai\dg{z_i,x}\big)\\
            &\leq C (\mu_i+\dg{z_i,x})^{-l} F_{\zma[i]}(x)\Bub_{\zm[i]}(x).
        \end{aligned}
        \end{multline*}
    Finally, using Lemma \ref{prop:estlami} and Proposition \ref{prop:estphiinfn}, together with Proposition \ref{prop:estdifX}, we obtain 
    \[  \Babs{\sum_{j=0}^n\lambda_i^j \nabla_g^l\Zed[i]{j}(x)}_g \leq C (\mu_i+\dg{z_i,x})^{-l} F_{\zma[i]}(x)\Bub_{\zm[i]}(x). 
        \]
    Putting these estimates together, we conclude.
\end{proof}

\subsection{Proof of Theorem \ref{prop:mainprinc}}
We conclude this section with the proof of Theorem \ref{prop:mainprinc}. We first show that, for a solution to \eqref{eq:critbub}, being close to $\TBuba$ in $\Sob(M)$ for some $\pa_\a$ automatically results in possessing strong pointwise estimates.
\begin{proposition}\label{prop:uniqsol}
    Let $(u_\a)_\a$ be a sequence in $\Sob(M)$ of positive solutions to \eqref{eq:critbub} for each $\a$. Assume that there exists, for all $\a \geq 1$, some weight $\mu_\a >0$ and point $z_\a \in M$, such that 
    \begin{equation}
        \begin{bigcases}
            &\a\mu_\a^2 \to 0\\
            &\Snorm{k}{u_\a - \TBuba} \to 0
        \end{bigcases}\qquad \text{as } \a \to \infty,
    \end{equation}
    where $\pa_\a = (z_\a,\mu_\a)$. Then, there exists a sequence of points $(\bz_\a)_\a$ in $M$, of positive numbers $(\bmu_\a)_\a$, such that 
    \begin{align}\label{tmp:pabpa}
        \frac{\bmu_\a}{\mu_\a} &\to 1,  & \frac{\dg{z_\a,\bz_\a}^2}{\mu_\a\bmu_\a} \to 0 & &\quad \text{as } \a \to \infty,
    \end{align}
    and there exists $\a_0\geq 1$ and a constant $C>0$ such that, for all $\a\geq \a_0$, the following holds.
    Write $\bpa_\a =(\bz_\a,\bmu_\a)$ and $\phi_\a := u_\a - \TBuba[\bpa_\a]$, for $l=0,\ldots 2k-1$, then
    \begin{equation}\label{eq:estphia}
        (\bmu_\a + \dg{\bz_\a,x})^l\abs{\nabla_g^l \phi_\a(x)}_g \leq C F_{\a,\bpa_\a}(x) \Bub_{\bz_\a,\bmu_\a}(x) \qquad \forall x~\in M.
    \end{equation}
\end{proposition}
\begin{proof}
    By hypothesis, there are $(\tau_\a)_\a$ and $(\epsilon_\a)_\a$ sequences of positive numbers such that $\tau_\a\to 0,\epsilon_\a \to 0$ as $\a \to \infty$, and such that 
    \begin{equation*}
        \begin{aligned}
            &(z_\a,\mu_\a)=\pa_\a \in \parama[\tau_\a]\\
            &\Snorm{k}{u_\a - \TBuba[\pa_\a]} \leq \epsilon_\a
        \end{aligned} \qquad \text{for all } \a.
    \end{equation*}
    Using Lemma \ref{prop:minprob}, there exists $\a_0\geq 1$ such that for all $\a \geq \a_0$, the problem 
    \[  \text{minimize}~ \Snorm{k}{u_\a - \TBub} \quad \text{for } \pa \in \parama[4\tau_\a] 
        \]
    is attained at some $\bpa_\a = (\bz_\a,\bmu_\a) \in \parama[2\tau_\a]$. Moreover, 
    \[  \Snorm{k}{\TBuba[\pa_\a]-\TBuba[\bpa_\a]} \leq 2\epsilon_\a,
        \]
    so that \eqref{tmp:pabpa} follows from Lemma \ref{prop:uniqbub}. Writing $\phi_\a := u_\a - \TBuba[\bpa_\a]$, by Lemma \ref{prop:phiainKa}, we have $\phi_\a \in \Kera[\bpa_\a]^\perp$. Observe that $\phi_\a$ satisfies, for all $\a$,
    \[  (\Dg+\a)^k \big(\TBuba[\bpa_\a]+\phi_\a\big) - \big(\TBuba[\bpa_\a] + \phi_\a\big)^{\deu} = 0.
        \]
    We also have $\Snorm{k}{\phi_\a}  \to 0$ as $\a \to \infty$. 
    \par We use Proposition \ref{prop:uniqest}: Up to increasing $\a_0\geq 1$, for all $\a\geq \a_0$ there is a unique 
    \[  \Tp_\a \in \Kera[\bpa_\a]^\perp \cap \big\{\phi \in C^0(M) \such \abs{\phi(x)}\leq \Lambda F_{\a,\bpa_\a}(x) \Bub_{\bz_\a,\bmu_\a}(x) \quad \forall~x\in M\big\}
        \]
    that solves 
    \begin{equation}\label{tmp:uniqx}
        \proKa[\bpa_\a] \bsq{\TBuba[\bpa_\a] + \Tp_\a -(\Dg+\a)^{-k}\big(\TBuba[\bpa_\a] + \Tp_\a\big)_+^{\deu-1}} = 0,
    \end{equation}
    and it satisfies $\Snorm{k}{\Tp_\a} \to 0$ as $\a\to \infty$. Using Proposition \ref{prop:uniqHk}, up to increasing again $\a_0$ so that $\Snorm{k}{\phi_\a} < \delta$, $\Snorm{k}{\Tp_\a} < \delta$, we conclude that $\phi_\a = \Tp_\a$. Finally, using Corollary \ref{prop:estdphi}, $\phi_\a$ satisfies, for $l=0,\ldots 2k-1$,
    \[  (\bmu_\a + \dg{\bz_\a,x})^{l} \abs{\nabla_g^l \phi_\a(x)}_g \leq CF_{\a,\bpa_\a}(x) \Bub_{\bz_\a,\bmu_\a}(x) \qquad \forall~x\in M.
        \]
\end{proof}

\begin{proof}[Proof of Theorem \ref{prop:mainprinc}]
    First observe that $u_\a \in C^\infty(M)$ for each $\a$, by elliptic regularity (see \cite[Proposition 8.3]{Maz16}). Define, for each $\a$, 
    \[  \mu_\a := \Lnorm[(M)]{\infty}{u_\a}^{-\frac{2}{n-2k}},
        \] 
    and let $z_\a\in M$ be the point where $u_\a$ reaches its maximum: $u_\a(z_\a) = \Lnorm[(M)]{\infty}{u_\a}$. Write 
    \[  \Tu_\a(y) = \mu_\a^{\frac{n-2k}{2}} u_\a(\exp_{z_\a}(\mu_\a y)) \qquad \text{for } y \in \Bal{0}{\varrho/\mu_\a} \sub \R^n,
        \]
    where $\inj/2 < \varrho < \inj$ is as in definition \ref{def:psizm}, and let $\Tg_\a := \exp_{z_\a}^* g(\mu_\a \cdot)$. Then $\Tu_\a$ satisfies 
    \begin{equation}\label{tmp:eqtua}
        (\Dg[\Tg_\a] + \a\mu_\a^2)^k \Tu_\a = \Tu_\a^{\deu-1} \qquad \text{in } \Bal{0}{\varrho/\mu_\a}.
    \end{equation}
    Observe that with \eqref{eq:critbub}, 
    \[  \Lnorm[(M)]{2}{u_\a}^2 \leq \frac{1}{\a^k}\dprod{(\Dg+\a)^k u_\a}{u_\a}_{\Sob[-k],\Sob} = \frac{1}{\a^k} \Lnorm[(M)]{\deu}{u_\a}^{\deu},
        \]
    so that using Hölder's inequality, we have 
    \[  1 = \frac{\Lnorm[(M)]{\deu}{u_\a}^{\deu}}{\Lnorm[(M)]{\deu}{u_\a}^{\deu}} \leq \frac{\Lnorm[(M)]{2}{u_\a}^{2}}{\Lnorm[(M)]{\deu}{u_\a}^{\deu}}\Lnorm[(M)]{\infty}{u_\a}^{\deu-2} \leq \frac{1}{\a^k\mu_\a^{2k}}.
        \]
    This shows that $\a\mu_\a^2$ is uniformly bounded with respect to $\a$, and $\mu_\a \to 0$ as $\a \to \infty$. Fix $R>0$, we have $\Tg_\a \to \xi$ in $C^\infty(\Bal{0}{R})$, and $(\Dg[\Tg_\a] + \a\mu_\a^2)^k$ is an elliptic operator on $\Bal{0}{R}$ with bounded coefficients, uniformly with respect to $\a$. Standard elliptic estimates now give that $(\Tu_\a)_\a$ is a bounded sequence in $C^{2k,\beta}(\cBal{0}{R})$ for $\beta \in (0,1)$. By compactness of the embedding $C^{2k,\beta}(\cBal{0}{R}) \sub C^{2k}(\cBal{0}{R})$, there exists $\Tu \in C^{2k}(\R^n)$ such that $\Tu_\a \to \Tu$ in $C^{2k}_{loc}(\R^n)$ up to a subsequence 
    , that we denote by $\a$ from now on. By definition, 
    \[  0< \Tu_\a(y) \leq 1 \qquad \text{ and } \qquad \Tu_\a(0)= 1.
        \]
    This passes to the limit, as $\a \to \infty$, so that $0\leq \Tu \leq 1$ and $\Tu \not\equiv 0$. We now show that $\Tu$ solves \eqref{eq:bub} in $\R^n$.
    \par We compute 
    \[  \Snorm{k-1}{u_\a}^2 \leq \frac{1}{\a}\dprod{(\Dg+\a)^k u_\a}{u_\a}_{\Sob[-k],\Sob} \leq \frac{1}{\a} \Lnorm[(M)]{\deu}{u_\a}^{\deu} \leq \frac{C}{\a},
        \]
    using Sobolev's inequality and \eqref{eq:assump}. Thanks to Mazumdar \cite{Maz16}, we know that for all $\epsilon>0$, there exists $B_\epsilon >0$ such that 
    \begin{equation}\label{eq:epsineq}
        \Lnorm[(M)]{\deu}{u_\a}^2 \leq (K_0^2 + \epsilon) \Lnorm[(M)]{2}{\Dg^{k/2}u_\a}^2 + B_\epsilon \Snorm{k-1}{u_\a}^2.
    \end{equation}  
    Let now $\a_0\geq 1$ be such that 
    \begin{equation}\label{tmp:maincas}
        \begin{cases}
            B_\epsilon \Snorm{k-1}{u_\a}^2 \leq \epsilon\\
            \Lnorm[(M)]{2}{\Dg^{k/2}u_\a}^2 = \Snorm{k}{u_\a}^2 - \Snorm{k-1}{u_\a}^2 \leq K_0^{-\frac{n}{k}} + \epsilon
        \end{cases}\hfill \forall~\a\geq \a_0,
    \end{equation}
    we have 
    \[  \Lnorm[(M)]{\deu}{u_\a}^2 \leq (K_0^2 + \epsilon)(K_0^{-\frac{n}{k}} + \epsilon) + \epsilon \leq K_0^{-\frac{n-2k}{k}} + C\epsilon.
        \]
    Since this holds for all $0<\epsilon\leq 1$, we obtain 
    \begin{equation}\label{tmp:mainx}
        \limsup_{\a\to \infty} \Lnorm[(M)]{\deu}{u_\a} \leq K_0^{-\frac{n-2k}{2k}}.
    \end{equation} 
    Using \eqref{eq:epsineq}, \eqref{tmp:maincas}, and \eqref{tmp:mainx}, for all $0<\epsilon\leq 1$ there exists $\a_0\geq 1$ such that for all $\a \geq \a_0$,
    \begin{multline*}
        \sum_{l=0}^k \tbinom{k}{l} \a^{k-l}\intM{\abs{\Dg^{l/2}u_\a}^2} = \intM{u_\a^\deu} \leq (K_0^{-2}+ \epsilon) \Lnorm[(M)]{\deu}{u_\a}^2\\
        \begin{aligned}
            &\leq (K_0^{-2}+\epsilon)(K_0^2 + \epsilon) \intM{\abs{\Dg^{k/2}u_\a}^2} + B_\epsilon \Snorm{k-1}{u_\a}^2\\
            &\leq \intM{\abs{\Dg^{k/2}u_\a}^2} + C\epsilon.
        \end{aligned}
    \end{multline*}
    Therefore, we have 
    \begin{equation}\label{tmp:bdDgal}
        \a^{k}\intM{u_\a^2} \leq C\epsilon \qquad \forall~\a\geq \a_0. 
    \end{equation}
    We also compute for all $\a \geq \a_0$, 
    \begin{multline}\label{tmp:bdDgalinv}
        \a^{k}\intM{u_\a^2} \geq \a^k \intM[\Bal{z_\a}{R\mu_\a}]{u_\a^2} = \a^k \mu_\a^{2k} \intMg{\Tg_\a}{\Bal{0}{R}}{\Tu_\a^2}\\
            =(\a\mu_\a^2)^k \bpr{\int_{\Bal{0}{R}} \Tu^2 dy + o(1)}.
    \end{multline}
    Since $\Tu \not\equiv 0$, \eqref{tmp:bdDgal} together with \eqref{tmp:bdDgalinv} give $\a\mu_\a^2\to 0$ as $\a \to \infty$. Equation \eqref{tmp:eqtua} then passes to the limit, and $\Tu$ solves 
    \[  \Dg[\xi]^k \Tu = \Tu^{\deu-1} \qquad \text{in }\R^n.
        \]
    Thanks to the classification of solutions of \cite{WeiXu99}, we  have $\Tu(y) = \Bub(y)$, where $\Bub$ is the standard Euclidean bubble defined in \eqref{def:bub}, and thus
    \[  \int_{\R^n}\Tu^{\deu} dy = K_0^{-\frac{n}{k}}.
        \]
    \par We now claim that 
    \begin{equation}\label{tmp:uaba}
        \Snorm{k}{u_\a - \TBuba} \to 0 \qquad \text{as } \a\to \infty,
    \end{equation}
    where $\pa_\a = (\zm[\a])$. To prove \eqref{tmp:uaba}, we compute, for all $R>0$,
    \[  \intM{u_\a^\deu} \geq \intM[\Bal{z_\a}{R\mu_\a}]{u_\a^\deu} = \int_{\Bal{0}{R}} \Tu^\deu dy + o(1).
        \]
    Since this holds for all $R>0$, we obtain 
    \begin{equation}\label{tmp:mainxx}
        \liminf_{\a\to \infty} \Lnorm[(M)]{\deu}{u_\a}^\deu \geq K_0^{-\frac{n}{k}}.
    \end{equation}
    Putting \eqref{tmp:mainx} and \eqref{tmp:mainxx} together, we have $\Lnorm[(M)]{\deu}{u_\a} \to K_0^{-\frac{n-2k}{2k}}$ as $\a \to \infty$, and then for all $R>0$, up to a subsequence,
    \begin{align*}
        \lim_{\a\to \infty} \intM[M\setminus \Bal{z_\a}{R\mu_\a}]{u_\a^\deu} &= \lim_{\a\to \infty} \intM{u_\a^\deu} - \lim_{\a\to\infty} \intM[\Bal{z_\a}{R\mu_\a}]{u_\a^\deu}\\
            &= K_0^{-\frac{n}{k}} - \int_{\Bal{0}{R}} \Tu^\deu dy = \epsilon(R),
    \end{align*}
    where $\epsilon(R) \to 0$ as $R\to \infty$. This gives, for all $R>0$, 
    \begin{multline*}
        \lim_{\a\to \infty} \intM{\abs{u_\a - \TBuba}^\deu}\\
        \begin{aligned}
            &=\lim_{\a\to \infty}\bsq{ \intM[\Bal{z_\a}{R\mu_\a}]{\abs{u_\a-\TBuba}^\deu} + \intM[M\setminus\Bal{z_\a}{R\mu_\a}]{\abs{u_\a-\TBuba}^\deu} }\\
            &= \lim_{\a\to \infty} \int_{\Bal{0}{R}} \abs{\Tu_\a(y) - \mu_\a^{\frac{n-2k}{2}}\TBuba(\exp_{z_\a}(\mu_\a y))}^\deu dy + \epsilon(R)\\
            &= \epsilon(R),
        \end{aligned}
    \end{multline*}
    up to a subsequence, since $\Tu_\a\to \Bub$ in $C^{2k}(\Bal{0}{R})$, $\Tg_\a \to \xi$ in $C^\infty_{loc}(\R^n)$, and $\mu_\a^{\frac{n-2k}{2}}\TBuba(\exp_{z_\a}(\mu_\a \cdot)) \to \Bub$ in $C^\infty_{loc}(\R^n)$ by Proposition \ref{prop:Tconc}. Thus, we have up to a subsequence $u_\a - \TBuba \to 0$ in $L^\deu(M)$, and as a consequence
    \begin{multline*}
        \Snorm{k}{u_\a - \TBuba}^2  \leq \dprod{(\Dg+\a)^k (u_\a - \TBuba)}{u_\a- \TBuba}_{\Sob[-k],\Sob} \\
            = \intM{\bpr{u_\a^{\deu-1} - \TBuba^{\deu-1}}(u_\a-\TBuba)} - \dprod{\erra{B}}{u_\a-\TBuba}_{\Sob[-k],\Sob},
    \end{multline*}
    where $\erra{B}\to 0$ in $\Sob[-k](M)$ as showed in Proposition \ref{prop:estRa}. Observe that $u_\a-\TBuba$ is bounded in $\Sob(M)$, and
    \begin{multline*}  
        \abs{\intM{\bpr{u_\a^{\deu-1} - \TBuba^{\deu-1}}(u_\a-\TBuba)}}\\ 
        \begin{aligned}
            &\leq \bpr{\Lnorm[(M)]{\deu}{u_\a}^{\deu-1}+\Lnorm[(M)]{\deu}{\TBuba}^{\deu-1}} \Lnorm[(M)]{\deu}{u_\a-\TBuba}\\
            &= o(1),
        \end{aligned}
    \end{multline*}
    so that we obtain $\Snorm{k}{u_\a-\TBuba}\to 0$ as $\a \to \infty$, up to a subsequence.
    \par We can now apply Proposition \ref{prop:uniqsol}, and there exists $\bpa_\a = (\bz_\a,\bmu_\a)$ such that $\a\bmu_\a^2 \to 0$, and 
    \[  \Snorm{k}{u_\a - \TBuba[\bpa_\a]} \to 0 \qquad \text{as } \a \to \infty,
        \]
    up to a subsequence. Moreover, writing $\phi_\a := u_\a - \TBuba[\bpa_\a]$, we have $\phi_\a \in C^{2k-1}(M)\cap \Kera[\bpa_\a]^\perp$ and $\phi_\a$ satisfies \eqref{eq:estphia}, which concludes the proof of Theorem  \ref{prop:mainprinc}.
\end{proof}

\section{The optimal constant for the Sobolev embedding}\label{sec:proof1}
In this section, we prove Theorem \ref{prop:mainest}. We use a contradiction argument, similar to the strategy of proof found in \cite{HebVau96, Heb03}. Our proof, however, is based on the pointwise blow-up description given by Theorem \ref{prop:mainprinc}.
\par We define the following functional, for any $\Lambda>0$, and for $u \in \Sob(M) \setminus\{0\}$,
    \[  I_\Lambda(u) := \frac{\Lnorm[(M)]{2}{\Dg^{k/2}u}^2 + \Lambda \Snorm{k-1}{u}^2}{\bpr{\intM{\abs{u}^\deu}}^{\frac{2}{\deu}}}. 
        \]
We then observe that the fact that any constant $A$ in \eqref{eq:sobineq} has to satisfy $A\geq K_0^2$ is equivalent to $\inf_{u\neq 0} I_\Lambda(u) \leq K_0^{-2}$, for all $\Lambda>0$. Moreover, if there is $\Lambda>0$ such that $\inf_{u\neq 0} I_\Lambda(u) = K_0^{-2}$, then Theorem \ref{prop:mainest} is true.
By contradiction, we will assume that 
\[  \inf_{u\neq 0} I_\Lambda(u) < \frac{1}{K_0^2}
    \]
for all $\Lambda >0$, where the infimum is taken among all $u \in \Sob(M)$, $u\neq0$. This is equivalent to the statement that
\begin{equation}\label{eq:contra}
    \forall~\a>0, \quad \inf_{u \in \mathcal{N}} \intM{(\Dg+\a)^k u ~u} < \frac{1}{K_0^2},
\end{equation}
where  
\[  \mathcal{N} := \{u \in \Sob(M) \such \Lnorm[(M)]{\deu}{u}=1\}.
    \]
\subsection{Proof of Theorem \ref{prop:mainest}}\label{sec:decom}
We now prove Theorem \ref{prop:mainest}. Assume that \eqref{eq:contra} holds, then by \cite[Theorem 3]{Maz16}, there exists for each $\a>0$ a positive minimizer 
\begin{equation}\label{def:ua}
    u_\a \in \mathcal{N},
\end{equation}
and a constant $\lambda_\a>0$ such that 
\begin{equation}\label{eq:ua}
    (\Dg+\a)^k u_\a = \lambda_\a u_\a^{\deu-1}.
\end{equation}
By elliptic regularity, as before we have $u_\a \in C^\infty(M)$ for all $\a >0$. Moreover,  
\[  \lambda_\a = \inf_{u\in \mathcal{N}} \intM{(\Dg+\a)^k u_\a ~u_\a},
    \]
and $\lambda_\a < 1/K_0^2$ for all $\a$, by assumption. It follows that 
\begin{equation}\label{eq:defbui}
    \bu_\a (x) := \lambda_\a^{\frac{n-2k}{4k}} u_\a(x)
    \end{equation}
is a positive solution to \eqref{eq:critbub}, and using \eqref{eq:ua}, 
\[  \Snorm{k}{\bu_\a}^2 = \lambda_\a^{\frac{n-2k}{2k}} \Snorm{k}{u_\a}^2 \leq \lambda_\a^{\frac{n}{2k}} \Lnorm[(M)]{\deu}{u_\a}^2 \leq K_0^{-\frac{n}{k}}.
    \]
We apply Theorem \ref{prop:mainprinc}: Up to a subsequence $\a_i \to \infty$, There exists a sequence $\pa_i := (z_i,\mu_i) \in M\times (0,+\infty)$ such that $\a_i \mu_i^2 \to 0$ as $i\to \infty$, and 
\[  \Snorm{k}{\bu_i - \TBub[i]} \to 0 \qquad \text{as } i \to \infty,
    \]
where we write $\bu_i :=\bu_{\a_i}$. From now on, we index the sequences by $i\in \N$ instead of $\a_i$ for simplicity. Thus, by Sobolev embedding, $\bu_i - \TBub[i] \to 0$ in $L^\deu(M)$, and using \eqref{eq:ii} we get
\begin{equation*}  
    \lambda_i^{\frac{n-2k}{4k}} = \Lnorm[(M)]{\deu}{\bu_i} = \Lnorm[(M)]{\deu}{\TBub[i]} + o(1) = K_0^{-\frac{n-2k}{2k}} + o(1),
\end{equation*}
so that $\lambda_i \to K_0^{-2}$ as $i\to \infty$. Moreover, we have that $\phi_i := \bu_i - \TBub[i]$ satisfies
\begin{align}  
    \notag&\Snorm{k}{\phi_i} \to 0, \\
    \label{eq:estimphi}&(\mu_i + \dg{z_i,x})^l \abs{\nabla_g^l \phi_i(x)}_g \leq C F_{\zma[i]}(x) \Bub_{\zm[i]}(x) \qquad \forall~x\in M.
\end{align}
We prove Theorem \ref{prop:mainest} by showing that \eqref{eq:estimphi} contradicts the sharp Sobolev inequality in $\R^n$, integrating \eqref{eq:ua} against $u_i$ and letting $i\to \infty$. 
\begin{lemma}
    Let $u_i := u_{\a_i}$ be given in \eqref{def:ua}, and let 
    \begin{equation}\label{def:gama}
        \gamma_i := \begin{cases}
            \sqai\mu_i & \text{if } n =2k+1\\
            \a_i \mu_i^2 \big(1+\abs{\log\sqai\mu_i}\big) & \text{if } n=2k+2\\
            \a_i \mu_i^2 & \text{if } n \geq 2k+3
        \end{cases}.
    \end{equation}
    There exists $C>0$ and $i_0 \geq 1$ such that for all $i\geq i_0$, 
    \begin{equation}\label{eq:1stclaim}
        \a_i \intM{\abs{\Dg^{\frac{k-1}{2}}u_i}^2} \geq C\gamma_i.
    \end{equation}
\end{lemma}
\begin{proof}
    By straightforward computations, one sees that for $l=0,\ldots k-1$,
    \[  \abs{\Dg[\xi]^{l/2}\Bub(y)} = \frac{\abs{P_l(\abs{y})}}{(1+\pk \abs{y}^2)^{-\frac{n-2k+2l}{2}}},
        \]
    where $P_l$ is a polynomial of degree $l$. Therefore, there exists a constant $C>0$ depending only on $n$ and $k$ such that for $i$ big enough,
    \begin{equation}\label{eq:controlintB}
        \frac{1}{C}\frac{\gamma_i}{\a_i\mu_i^2} \leq \int_{\Bal{0}{\frac{1}{\sqai\mu_i}}} \abs{\Dg[\xi]^{\frac{k-1}{2}}\Bub(y)}^2 dy \leq C\frac{\gamma_i}{\a_i\mu_i^2}
    \end{equation}  
    for all $n > 2k$. By the decomposition of $\Dg$ in local coordinates, as in \eqref{eq:nabexp}, we get for all $y\in \Bal{0}{1/\sqa}\sub \R^n$, 
    \[  \big(\Dg^{\frac{k-1}{2}} \TBub[i]\big)(\exp_{z_i}(y)) = \mu_i^{\frac{-n+2}{2}}\big(\Dg[\xi]^{\frac{k-1}{2}} \Bub\big) \big(\tfrac{y}{\mu_i}\big) + \bigO\bpr{\sum_{m=1}^{k-2} \mu_i^{k-m} \abs{\big(\nabla_\xi^m\Bub\big)\big(\tfrac{y}{\mu_i}\big)}}.
        \]  
    Moreover, we compute, for $m=0,\ldots k-1$, that
    \[  \int_{\Bal{0}{\frac{1}{\sqai\mu_i}}}\abs{\nabla_\xi^m \Bub(y)}^2 dy \leq C\begin{cases}
        (\sqai\mu_i)^{-2(k-m)+n-2k} & \text{if } m < k- \frac{n-2k}{2}\\
        1 + \abs{\log\sqai\mu_i}  & \text{if } m= k-\frac{n-2k}{2}\\
        1 & \text{if } m > k - \frac{n-2k}{2}.
    \end{cases}
        \]
    We start by showing \eqref{eq:1stclaim} in the case $n\geq 2k+2$. Thanks to \eqref{eq:controlintB}, and using the expansion of the metric in local coordinates and Lemma \ref{prop:intcomp}, we obtain when $i\geq i_0$ is big enough so that $1/\sqai < \varrho_0$,
    \begin{multline}\label{tmp:estintx}
        \a_i \intM[\Bal{z_i}{1/\sqai}]{\abs{\Dg^{\frac{k-1}{2}}\TBub[i]}^2} \geq C \a_i \mu_i^2 \int_{\Bal{0}{\frac{1}{\sqai\mu_i}}} \abs{\Dg[\xi]^{\frac{k-1}{2}}\Bub(y)}^2 dy\\
        + \bigO\bpr{\int_{\Bal{0}{\frac{1}{\sqai\mu_i}}}\a_i\mu_i^4 \abs{y}^2 \abs{\nabla_\xi^{k-1}\Bub(y)}^2 dy}\\ \hfill + \bigO\bpr{\a_i\sum_{m=1}^{k-2} \mu_i^{2(k-m)}\int_{\Bal{0}{\frac{1}{\sqai\mu_i}}} \abs{\nabla_\xi^m \Bub(y)}^2 dy}\\
        \qquad \qquad\geq C \gamma_i (1+ o(1)). \hfill
    \end{multline}
    Moreover, using \eqref{eq:estimphi}, writing as in \eqref{def:snk}
    \[  \sigma_{n,k} = \begin{cases}
        3/4 &\begin{aligned}[t]
            &\text{when } k=1, \,n\geq 4\\
            &~\text{or } k\geq 2,\, n=2k+2
        \end{aligned}\\
        1 & \text{when $k\geq 2,\, n \geq 2k+3$}
    \end{cases}, 
        \]
    we have for all $n\geq 2k+2$
    \begin{multline}\label{tmp:estintxx}
        \abs{\a_i \intM[\Bal{z_i}{1/\sqai}]{\vprod{\Dg^{\frac{k-1}{2}}\TBub[i]}{\Dg^{\frac{k-1}{2}}\phi_i}}}\\
        \begin{aligned}
            &\leq C \a_i \a_i^{\sigma_{n,k}}\int_{\Bal{0}{1/\sqai}} (\mu_i+\abs{y})^{2\sigma_{n,k}} \mu_i^{n-2k} (\mu_i + \abs{y})^{-2(n-k-1)} dy\\
            &\leq C (\a_i \mu_i^2)^{1+\sigma_{n,k}} \int_{\Bal{0}{\frac{1}{\sqai\mu_i}}} (1+\abs{y})^{-2n+2k+2+2\sigma_{n,k}} dy\\
            &= o(\gamma_i).
        \end{aligned}
    \end{multline} 
    Putting \eqref{tmp:estintx} and \eqref{tmp:estintxx} together, we obtain with \eqref{eq:defbui} that
    \begin{multline*}
        \lambda_i^{\frac{n-2k}{2k}}\a_i \intM{\abs{\Dg^{\frac{k-1}{2}}u_i}^2}\\
        \begin{aligned}
            &\geq \a_i \intM[\Bal{z_i}{1/\sqai}]{\vprod{\Dg^{\frac{k-1}{2}}\big(\TBub[i]+\phi_i\big)}{\Dg^{\frac{k-1}{2}}\big(\TBub[i]+\phi_i\big)}}\\
            &\geq \a_i \intM[\Bal{z_i}{1/\sqai}]{\abs{\Dg^{\frac{k-1}{2}}\TBub[i]}^2} + \a_i \intM[\Bal{z_i}{1/\sqai}]{\vprod{\Dg^{\frac{k-1}{2}} \TBub[i]}{\Dg^{\frac{k-1}{2}}\phi_i}} \\
            &\geq C\gamma_i (1+o(1))
        \end{aligned}
    \end{multline*}
    as $i\to \infty$, which proves \eqref{eq:1stclaim}. 
    \par In the case $n= 2k+1$, let $0<\epsilon< 1$ to be fixed later, we have as in \eqref{tmp:estintx} 
    \begin{multline*}
        \a_i\intM{\abs{\Dg^{\frac{k-1}{2}}\TBub[i]}^2} \geq C \a_i\mu_i^2 \int_{\Bal{0}{\frac{\epsilon}{\sqai\mu_i}}} \abs{\Dg[\xi]^{\frac{k-1}{2}}\Bub(y)}^2 dy\\
            + \bigO\bpr{\int_{\Bal{0}{\frac{\epsilon}{\sqai\mu_i}}}\a_i\mu_i^4 \abs{y}\abs{\nabla_\xi^{k-1} \Bub(y)}^2 dy}\\
            \hfill + \bigO\bpr{\int_{\Bal{0}{\frac{\epsilon}{\sqai\mu_i}}} \a_i \sum_{m=1}^{k-2}\mu^{2(k-m)} \int_{\Bal{0}{\frac{\epsilon}{\sqai\mu_i}}}\abs{\nabla_\xi^m \Bub(y)}^2}\\
            \qquad \geq C \sqai\mu_i (1+o(1)) \geq C_0 \sqa_i\mu_i
    \end{multline*} 
    for some $C_0> 0$ independent of $\epsilon$, for all $i\geq i_0$ big enough. Now, as in \eqref{tmp:estintxx}, we obtain
    \begin{multline*}
        \abs{\a_i \intM[\Bal{z_i}{\epsilon/\sqai}]{\vprod{\Dg^{\frac{k-1}{2}}\TBub[i]}{\Dg^{\frac{k-1}{2}}\phi_i}}} \\
        \begin{aligned}
            &\leq C \a_i\a_i^{1/2} \int_{\Bal{0}{\epsilon/\sqai}} (\mu_i + \abs{y})\mu_i (\mu_i + \abs{y})^{-n+1} (1+\abs{\log\sqai(\mu_i+\abs{y})}) dy\\
            &\leq C_1 \epsilon^2 \big(1+ \abs{\log\epsilon}\big) \sqai\mu_i,
        \end{aligned}
    \end{multline*}
    where $C_1>0$ does not depend on $\epsilon$ either. Choose $\epsilon>0$ small enough so that $1-\frac{C_1}{C_0}\epsilon\big(1+\abs{\log\epsilon}\big) \geq \frac{1}{2}$, we get finally
    \begin{multline*}
        \lambda_i^{\frac{n-2k}{2k}}\a_i \intM{\abs{\Dg^{\frac{k-1}{2}}u_i}^2}\\
        \begin{aligned}
            &\geq \a_i \intM[\Bal{z_i}{\epsilon/\sqai}]{\vprod{\Dg^{\frac{k-1}{2}}\big(\TBub[i] + \phi_i\big)}{\Dg^{\frac{k-1}{2}}\big(\TBub[i] + \phi_i\big)}}\\
            &\geq \frac{C_0}{2}\sqai \mu_i.
        \end{aligned}
    \end{multline*}
\end{proof}
\begin{lemma}
    Let $u_i = u_{\a_i}$ be given in \eqref{def:ua}, $\gamma_i$ be as in \eqref{def:gama}, and write
    \begin{equation}\label{def:U}
        \U(y) := u_i(\exp_{z_i}(y)) \qquad \text{for } y \in \Bal{0}{\varrho} \sub \R^n,
    \end{equation}
    where $\inj/2 < \varrho < \inj$ is as in definition \ref{def:psizm}. We have, as $i\to \infty$, that
    \begin{align}
        \label{eq:2ndclaim} \intM{\abs{\Dg^{k/2}u_i}^2} &= \int_{\Bal{0}{\varrho}} \abs{\Dg[\xi]^{k/2}\U}^2 dy + o(\gamma_i)\\
        \label{eq:3rdclaim} \intM{u_i^{\deu}} &= \int_{\Bal{0}{\varrho_0}} \U^\deu dy + o(\gamma_i),
    \end{align}
    where $0<\varrho_0<\inj/2$ is given by Lemma \ref{prop:intcomp}.
\end{lemma}
\begin{proof}
    First, observe that with \eqref{eq:estimphi} and \eqref{eq:estdifXsqa}, \eqref{eq:estdifXexp}, for $l= 0,\ldots 2k-1$ and $y \in \Bal{0}{\varrho}$, we have
    \begin{equation}\label{tmp:estint1}
        \abs{\nabla_g^l u_i}_g(\exp_{z_i}(y)) \leq C\begin{cases}
            \mu_i^{\frac{n-2k}{2}} (\mu_i+\abs{y})^{2k-n-l} & \text{when } \sqai\abs{y} \leq 1\\
            \a_i^{l/2} \mu_i^{\frac{n-2k}{2}} \abs{y}^{2k-n} e^{-\sqai\abs{y}/2} & \text{when } \sqai\abs{y} \geq 1,
        \end{cases}
    \end{equation}
    and for all $x\in M\setminus \Bal{z_i}{\varrho}$,
    \[  \abs{\Dg^{k/2}u_i(x)}^2 \leq C \a_i^{2k} \mu_i^{n-2k} e^{-\sqai\varrho} = o(\gamma_i) 
        \]
    as $i\to \infty$. For $x\in \Bal{z_i}{\varrho}$, we use again the decomposition of $\Dg$ and $dv_g$ in local coordinates, we have
    \begin{multline*}
        \intM[\Bal{z_i}{\varrho}]{\abs{\Dg^{k/2}u_i}^2}
            = \int_{\Bal{0}{\varrho}} \abs{\Dg[\xi]^{k/2} \U(y)}^2 dy + \bigO\bpr{\int_{\Bal{0}{\varrho}}\abs{y}^2 \abs{\nabla_\xi^k \U(y)}^2 dy }\\ + \bigO\bpr{\sum_{m=1}^{k-1}\int_{\Bal{0}{\varrho}}\abs{\nabla_\xi^m \U(y)}^2 dy}.
    \end{multline*}
    We now estimate these quantities, using \eqref{tmp:estint1}, we obtain
    \begin{align*}
        &\int_{\Bal{0}{1/\sqai}} \abs{y}^2 \abs{\nabla_\xi^k \U(y)}^2 dy \leq C \mu_i^2 \int_{\Bal{0}{\frac{1}{\sqai\mu_i}}} \abs{y}^2 (1+\abs{y})^{-2(n-k)}dy \leq \frac{C}{\a_i}\gamma_i \\
        &\int_{\Bal{0}{\varrho}\setminus\Bal{0}{1/\sqai}}\abs{y}^2 \abs{\nabla_\xi^k \U(y)}^2 dy \leq \frac{C}{\a_i}(\sqai\mu_i)^{n-2k} = o(\gamma_i).
    \end{align*}
    When $k\geq 2$, we also need to compute the following integrals: for $m=1,\ldots k-1$,
    \begin{equation}\label{tmp:Ua}
    \begin{aligned}
        &\begin{aligned}
            \int_{\Bal{0}{1/\sqai}} \abs{\nabla_\xi^m \U}^2 dy &\leq C\mu_i^{2(k-m)}\int_{\Bal{0}{\frac{1}{\sqai\mu_i}}} (1+\abs{y})^{-2(n-2k+m)}\\
                &\leq C \begin{cases}
                    (\sqai\mu_i)^{n-2k}\a_i^{-k+m} & \text{if } k-m>\frac{n-2k}{2}\\
                    \mu_i^{n-2k} \big(1+\abs{\log\sqai\mu_i}\big) & \text{if } k-m=\frac{n-2k}{2}\\
                    \mu_i^2 & \text{if } k-m<\frac{n-2k}{2}
                \end{cases} \\
                &= o(\gamma_i),
            \end{aligned}\\
        &\int_{\Bal{0}{\varrho}\setminus\Bal{0}{1/\sqai}} \abs{\nabla_\xi^m \U}^2 dy \leq C\a_i^{m-l}(\sqai\mu_i)^{n-2k} = o(\gamma_i).
    \end{aligned}
    \end{equation}
    Putting everything together, we have finally
    \begin{align*}
        \intM{\abs{\Dg^{k/2}u_i}^2} &= \intM[\Bal{z_i}{\varrho}]{\abs{\Dg^{k/2} u_i}^2} +o(\gamma_i)\\
            &= \int_{\Bal{0}{\varrho}} \abs{\Dg[\xi]^{k/2} \U}^2 dy + o(\gamma_i)
    \end{align*}
    as $i \to \infty$, which proves \eqref{eq:2ndclaim}. 
    \par To prove \eqref{eq:3rdclaim}, we compute as before using \eqref{tmp:estint1},
    \begin{align*}
        &\int_{\Bal{0}{1/\sqai}} \abs{y}^2 \U^\deu(y) dy \leq C \mu_i^2 \int_{\Bal{0}{\frac{1}{\sqai\mu_i}}} \abs{y}^2 (1+\abs{y})^{-2n} = o(\gamma_i)\\
        &\int_{\Bal{0}{\varrho_0}\setminus\Bal{0}{1/\sqai}} \abs{y}^2 \U^\deu(y) dy \leq \frac{C}{\a_i}(\sqai\mu_i)^n = o(\gamma_i),
    \end{align*}
    and we conclude with the same arguments.
\end{proof}

We have now everything we need to conclude the proof of Theorem \ref{prop:mainest}. 
\begin{proof}[End of the proof of Theorem \ref{prop:mainest}]
    Integrate \eqref{eq:ua} against $u_i \in \Sob(M)$, we get
    \begin{equation}\label{tmp:intuai}
        \sum_{l=0}^k \tbinom{k}{l} \a_i^{k-l} \intM{\abs{\Dg^{l/2}u_i}^2} = \lambda_i \intM{u_i^\deu}.
    \end{equation}
    Thanks to \eqref{eq:1stclaim}, there exists $C>0$ such that \eqref{tmp:intuai} gives
    \begin{align*}  
        \intM{\abs{\Dg^{k/2}u_i}^2} + C\gamma_i &\leq \lambda_i \intM{u_i^\deu}\\
            &\leq \frac{1}{K_0^2} \intM{u_i^\deu}.
    \end{align*}
    It then follows from \eqref{eq:2ndclaim} and \eqref{eq:3rdclaim} that
    \begin{equation}\label{eq:mainprofI}
        \int_{\Bal{0}{\varrho}} \abs{\Dg[\xi]^{k/2} \U}^2 dy + C\gamma_i \leq \frac{1}{K_0^2} \bpr{\int_{\Bal{0}{\varrho_0}} \U^\deu dy}^{2/\deu} + o(\gamma_i),
    \end{equation}
    since $\int_{\Bal{0}{\varrho_0}}\U^\deu dy \leq \Lnorm[(M)]{\deu}{u_i}^\deu = 1$. Recall that $\varrho_0 < \inj/2 < \varrho$, and let $\Tilde{\chi} \in \Cct(\R^n)$ be a cut-off function such that $\Tilde{\chi} \equiv 1$ in $\Bal{0}{\varrho_0}$ and $\Tilde{\chi}\equiv 0$ in $\R^n \setminus \Bal{0}{\varrho}$. We have 
    \begin{equation*}
        \bpr{\int_{\Bal{0}{\varrho_0}} \U^\deu dy}^{2/\deu} \leq \bpr{\int_{\R^n} \big(\Tilde{\chi}\U\big)^\deu dy}^{2/\deu},
    \end{equation*}
    and $\Tilde{\chi} \U \in \Cct(\R^n)$. We use the optimal Sobolev's inequality on $\R^n$, and obtain
    \[  \bpr{\int_{\R^n} \big(\Tilde{\chi}\U\big)^\deu dy}^{2/\deu} \leq K_0^2 \int_{\R^n} \abs{\Dg[\xi]^{k/2} \big(\Tilde{\chi}\U\big)}^2 dy = K_0^2 \int_{\Bal{0}{\varrho}} \abs{\Dg[\xi]^{k/2} \big(\Tilde{\chi}\U\big)}^2 dy.
        \]
    Observe that 
    \[\begin{aligned}
        \int_{\Bal{0}{\varrho}}\abs{\Dg[\xi]^{k/2} \big(\Tilde{\chi}\U\big)}^2 dy &\leq \int_{\Bal{0}{\varrho}} \Tilde{\chi}^2 \abs{\Dg[\xi]^{k/2}\U}^2 dy + \bigO\bpr{\sum_{m=0}^{k-1}\int_{\Bal{0}{\varrho}}\abs{\nabla_\xi^m \U}^2 dy}\\
            &\leq \int_{\Bal{0}{\varrho}} \abs{\Dg[\xi]^{k/2}\U}^2 dy + o(\gamma_i)
    \end{aligned}\]
    by \eqref{tmp:Ua}, so that finally 
    \begin{equation}\label{eq:mainprofII}
        \bpr{\int_{\Bal{0}{\varrho_0}} \U^\deu dy}^{2/\deu} \leq K_0^2 \int_{\Bal{0}{\varrho}} \abs{\Dg[\xi]^{k/2}\U}^2 dy + o(\gamma_i)
    \end{equation}
    as $i\to \infty$.
    Putting \eqref{eq:mainprofI} and \eqref{eq:mainprofII} together, we have 
    \[  C\gamma_i \leq o(\gamma_i)
        \]
    for some $C>0$ which is independent of $i$. This is a contradiction and concludes the proof of Theorem \ref{prop:mainest}.
\end{proof}
\begin{remark}
    Since we have proved that \eqref{eq:SoboptHk} is optimal with respect with the first constant, we can now optimize the constant corresponding to the lower-order terms. We set
    \[  B_0(g) := \inf \{B>0 \text{ such that \eqref{eq:SoboptHk} holds with $B$}\}, 
        \]
    then we have for all $u\in \Sob(M)$,
    \begin{equation}\label{tmp:sobnorm}
        \Lnorm[(M)]{\deu}{u}^2 \leq K_0^2 \Lnorm[(M)]{2}{\Dg^{k/2}u}^2 + B_0(g) \Snorm{k-1}{u}^2.
    \end{equation}
    This inequality is optimal with respect to both constants, none of them can be improved. Testing \eqref{tmp:sobnorm} on the constant functions, we have immediately
    \[  B_0(g) \geq \bpr{\Vol(M)}^{-\frac{n}{2k}}.
        \]
    We refer to \cite{DjaDru01}, and to \cite{DruHeb02}, for a more precise characterization of the optimal second constant $B_0(g)$ and the corresponding extremals for \eqref{tmp:sobnorm} on a compact manifold, for the case $k=1$. 
\end{remark}
    The contradiction argument developed here to prove Theorem \ref{prop:mainest} shows the following result along the way.
\begin{corollary}
    There exists $\a_0>0$ such that for all $\a\geq \a_0$, any positive solution $u_\a \in \Sob(M)$ to \eqref{eq:critbub} satisfies
    \[  \Snorm{k}{u_\a} > K_0^{-\frac{n}{2k}}.
        \]
\end{corollary}

\appendix
\section{Technical results}
\subsection{Integration computations}\label{sec:intconc}
We show here standard results about the pull-back of functions onto $(M,g)$ using the local coordinates. One may refer for instance to \cite{Maz16} for similar computations.
\begin{lemma}\label{prop:intcomp}
    Let $(M,g)$ be a smooth compact manifold of dimension $n$, and let $l$ be a positive integer such that $n>2l$. For all $0<\varrho<\inj$, there exists $C>0$ such that for all $z\in M$ and $f\in \Cct(\Bal{0}{\varrho})$, we have
    \begin{equation}\label{tmp:intx1}
        \intM{\abs{\nabla_g^l \big(f\circ \exp_z^{-1}\big)}_g^2} \leq C \int_{\R^n} \abs{\nabla_\xi^l f}^2 dy.
    \end{equation}
    Moreover, there exists $0<\varrho_0<\inj$ and $C_0>0$ such that for all $z\in M$ and $f \in \Cct(\Bal{0}{\varrho_0})$,
    \begin{equation}\label{tmp:intx2}
        \int_{\R^n} \abs{\nabla_\xi^l f}^2 dy \leq C_0 \intM{\abs{\nabla_g^l \big(f\circ \exp_z^{-1}\big)}_g^2}.
    \end{equation}
\end{lemma}
\begin{proof}
    We use the properties of the local coordinates: writing $\Tg := \exp_z^* g$, we have
    \begin{align*}
        \abs{\Tg_{ij}(y) - \delta_{ij}} &\leq C \abs{y}^2\\
        \abs{\nabla_{\Tg}^l f(y)}_{\Tg} &= \abs{\nabla_\xi^l f(y)} + \bigO\Big(\abs{y}^2 \tabs{\nabla_\xi^l f(y)}\Big)\\ 
            &\qquad+ \bigO\Big(\abs{y}\tabs{\nabla_\xi^{l-1}f(y)}\Big) + \bigO\Big(\tsum_{m=1}^{l-2}\tabs{\nabla_\xi^m f(y)}\Big) \qquad l\geq 2
    \end{align*}
    for all $y \in \Bal{0}{\varrho} \sub \R^n$. We also have, for all $s>0$, 
    \begin{equation}\label{tmp:inta}
        \abs{\intM[\Bal{z}{\varrho}]{\abs{f\circ \exp_z^{-1}}^s} - \int_{\Bal{0}{\varrho}} \abs{f}^s dy} \leq C \varrho^{2}\int_{\Bal{0}{\varrho}}\abs{f}^s dy.
    \end{equation}
    Then, we compute 
    \begin{multline}\label{tmp:intb}
        \abs{\int_{\Bal{0}{\varrho}} \abs{\nabla_{\Tg}^l f}_{\Tg}^2 dy - \int_{\Bal{0}{\varrho}} \abs{\nabla_\xi^l f}^2 dy}\\ \leq C \varrho^2 \int_{\Bal{0}{\varrho}} \abs{\nabla_\xi^l f}^2 dy + C \sum_{m=1}^{l-1} \int_{\Bal{0}{\varrho}} \abs{\nabla_\xi^m f}^2 dy.
    \end{multline}
    For $m\leq l-1$, we use the continuous embedding $\hSob[l,2](\R^n) \emb \hSob[m,\deuk{l-m}](\R^n)$, where $\deuk{l} := \frac{2n}{n-2l}$, and get with Hölder's inequality
    \begin{equation}\label{tmp:intc}
        \begin{aligned}
            \int_{\Bal{0}{\varrho}} \abs{\nabla_\xi^m f}^2 dy &\leq C \big(\varrho^n\big)^{\frac{2(l-m)}{n}} \bpr{\int_{\Bal{0}{\varrho}} \abs{\nabla_\xi^m f}^{\deuk{l-m}}}^{2/\deuk{l-m}}\\
                &\leq C \varrho^2 \int_{\Bal{0}{\varrho}} \abs{\nabla_\xi^l f}^2 dy.
        \end{aligned}
    \end{equation}
    Putting \eqref{tmp:inta}, \eqref{tmp:intb}, \eqref{tmp:intc} together, this concludes the proof of \eqref{tmp:intx1} for all $\varrho < \inj$. Finally, taking a small enough $0< \varrho_0< \inj$, there is $C_0>0$ such that
    \[  \intM{\abs{\nabla_g^l \big(f \circ \exp_z^{-1}\big)}_g^2} \geq (1-C\varrho_0^2)\int_{\R^n} \abs{\nabla_\xi^l f}^2 dy \geq \frac{1}{C_0} \int_{\R^n} \abs{\nabla_\xi^l f}^2 dy.
        \]
\end{proof}

We are now in position to prove Lemma \ref{prop:psizmbded}.
\begin{proof}[Proof of Lemma \ref{prop:psizmbded}]
    Let $\mu_0 < \varrho$ be given by Lemma \ref{prop:intcomp}, $z\in M$, and write 
    \[  \Hps_{\zm} (x) := \mu^{-\frac{n-2k}{2}} \psi\big(\tfrac{1}{\mu}\exp_z^{-1}(x)\big) \qquad \forall~ x \in \Bal{z}{\varrho}, 
        \]
    so that $\psi_{\zm}(x) = \chi_\varrho(\dg{z,x})\Hps_{\zm}(x)$. We have 
    \[  \Snorm{k}{\psi_{\zm}}^2 = \sum_{l=0}^k \intM[\Bal{z}{\varrho}]{\abs{\Dg^{l/2} \psi_{\zm}}^2}, 
        \]
    and for $l = 0,\ldots k$, 
    \begin{align*}  
        \intM[\Bal{z}{\varrho}]{\abs{\Dg^{l/2} \psi_{\zm}}^2} &\leq C \sum_{m=0}^l \intM[\Bal{z}{\varrho}]{\abs{\nabla_g^{l-m}\chi_\varrho(\dg{z,x})}_g^2 \abs{\nabla_g^m \Hps_{\zm}(x)}_g^2}(x)\\
            &\leq C \sum_{m=0}^l \intM[\Bal{z}{\varrho}]{\abs{\nabla_g^m \Hps_{\zm}}_g^2}.
    \end{align*}
    We use \eqref{tmp:intx1} and get 
    \[  \intM[\Bal{z}{\varrho}]{\abs{\nabla_g^m\Hps_{\zm}}_g^2} \leq C \int_{\Bal{0}{\varrho}}\abs{\nabla_\xi^m \Tps_\mu(y)}^2 dy,
        \]
    where we write $\Tps_\mu(y) := \mu^{-\frac{n-2k}{2}} \psi\big(\tfrac{y}{\mu}\big)$. By Hölder's inequality and a change of variables, we obtain for $m=0,\dots k$, 
    \begin{align*}
        \int_{\Bal{0}{\varrho}} \abs{\nabla_\xi^m \Tps_\mu}^2 dy &\leq C \big(\varrho^n\big)^{\frac{2(k-m)}{n}} \bpr{\int_{\Bal{0}{\varrho}} \abs{\nabla_\xi^m \Tps_\mu}^{\deuk{k-m}}dy}^{2/\deuk{k-m}}\\
            &\leq C \bpr{\int_{\Bal{0}{\varrho/\mu}} \abs{\nabla_\xi^m \psi}^{\deuk{k-m}} dy}^{2/\deuk{k-m}}\\
            &\leq C \Hnorm{k,2}{\psi}^2
    \end{align*}
    using the continuous embedding $\hSob(\R^n) \emb \hSob[m,\deuk{k-m}](\R^n)$, since $\psi \in \hSob(\R^n)$. In the end, we have thus showed
    \[  \Snorm{k}{\psi_{\zm}}^2 \leq C \Hnorm{k,2}{\psi}^2.
        \]
\end{proof}

In the rest of the appendix, we use the notations introduced in section \ref{sec:prelim} of this paper. 
\begin{lemma}\label{prop:Phito0}
    For all $\epsilon>0$, there exists $\a_0\geq 1$ and $\tau_0>0$ such that for all $\a\geq \a_0$, $\tau\leq \tau_0$, and $\pa = (\zm)\in \param$, 
    \[  \a^{k-l}\intM{\abs{\Dg^{l/2}\bpr{\Zdif}(x)}^2}(x) < \epsilon.
        \]
\end{lemma}
\begin{proof}
    Start by choosing $\a_0\geq 1$ such that $1/\sqrt{\a_0} < \inj/2$. By the product rule, for $l=0, \ldots 2k$, we have
    \begin{multline*}  
        \abs{\nabla_g^l \bpr{\Zdif}(x)}_g\\ \leq C \mu\sum_{m=0}^l \abs{\nabla_g^m \bpr{\dsur{}{z_j}\Theta_\a(z,\cdot)}(x)}_g\abs{\nabla_g^{l-m} \Bub_{\zm}(x)}_g.
        \end{multline*}
    Using \eqref{eq:estdThet}, we then obtain that when $\sqa\dg{z,x} \leq 1$ or $\dg{z,x} > \varrho$,
    \[  \nabla_g^l \bpr{\Zdif}(x) = 0, 
        \]
    and when $1/\sqa \leq \dg{z,x} \leq \varrho$,
    \begin{multline}\label{eq:estdPhi}
        \abs{\nabla_g^l \bpr{\Zdif}(x)}_g\\ \leq C \sum_{m=0}^l \mu \a^{\frac{m+1}{2}}e^{-\sqa\dg{z,x}/2} \mu^{\frac{n-2k}{2}}\dg{z,x}^{2k-n-l+m}.
    \end{multline}
    Thus, we obtain, for $l=0,\ldots k$,
    \begin{multline*}  
        \a^{k-l}\intM{\abs{\Dg^{l/2}\bpr{\Zdif}(x)}^2}(x)\\
        \begin{aligned}
            &= \a^{k-l}\intM[\Bal{z}{\varrho}\setminus\Bal{z}{1/\sqa}]{\abs{\Dg^{l/2}\bpr{\Zdif}(x)}^2}(x)\\
            &\leq C\a\mu^2 \sum_{m=0}^l \int_{\Bal{0}{\varrho}\setminus\Bal{0}{1/\sqa}} \a^{k-l+m} \mu^{n-2k}\abs{y}^{4k-2n-2l+2m}e^{-\sqa \abs{y}/2} dy\\
            &\leq C(\a\mu^2)^{\frac{n-2k}{2}+1}.
        \end{aligned}
    \end{multline*}
    Taking $\tau_0\leq 1$ small enough, and since $\a\mu^2< \tau$, we conclude.
\end{proof}
\begin{lemma}\label{prop:estintRX}
    Let $\a\geq 1$ be such that $1/\sqa < \inj/2$, $\tau\leq 1$, and $\pa = (\zm) \in \param$, and let $R \in L^{\frac{2n}{n+2k}}(M)$ be such that
    \[  \abs{R(x)} \leq \mu^{\frac{n-2k}{2}}(\mu+\dg{x,z})^{2k-n} \begin{cases}
        \a^\sigma (\mu + \dg{z,x})^{2\sigma - 2k} & \text{when } \sqa\dg{z,x} \leq 1\\
        \a^k e^{-\sqa\dg{z,x}/2} & \text{when } \sqa\dg{z,x} \geq 1
    \end{cases},
        \]
    for some $\sigma \in [1/2,1]$. Then, $R \in \Sob[-k](M)$ and 
    \[  \Snorm{-k}{R} \leq C (\sqa\mu)^{1/2}.
        \]
\end{lemma}
\begin{proof}
    We use the continuous embedding $L^{\frac{2n}{n+2k}}(M) \emb \Sob[-k](M)$, so that $R\in \Sob[-k](M)$, and we compute
    \[  \intM{\abs{R(x)}^{\frac{2n}{n+2k}}}(x) \leq I_1^{\frac{2n}{n+2k}} + I_2^{\frac{2n}{n+2k}} + I_3^{\frac{2n}{n+2k}},
        \]
    defining
    \begin{align*}
        I_1 &:= \bpr{\intM[\Bal{z}{1/\sqa}]{\abs{R}^{\frac{2n}{n+2k}}}}^{\frac{n+2k}{2n}}, ~
        I_2 &:= \bpr{\intM[\Bal{z}{\varrho} \setminus\Bal{z}{1/\sqa}]{\abs{R}^{\frac{2n}{n+2k}}}}^{\frac{n+2k}{2n}},\\
        I_3 &:= \bpr{\intM[M \setminus \Bal{z}{\varrho}]{\abs{R}^{\frac{2n}{n+2k}}}}^{\frac{n+2k}{2n}}.
    \end{align*}
    We first estimate $I_1$, by a change of variables we have 
    \begin{align*}
        I_1 &\leq C (\a\mu^2)^\sigma \bpr{\int_{\Bal{0}{\frac{1}{\sqa\mu}}} (1+\abs{y})^{(2\sigma-n)\frac{2n}{n+2k}} dy}^{\frac{n+2k}{2n}}\\
            &\leq C (\a\mu^2)^\sigma \begin{cases}
                (\sqa \mu)^{\frac{n-2k-4\sigma}{2}} & \text{if } n-2k < 4\sigma\\
                1 + \abs{\log\sqa\mu} & \text{if } n-2k = 4\sigma\\
                1 & \text{if } n-2k > 4\sigma
            \end{cases}\\
            &\leq C (\sqa\mu)^{1/2}
    \end{align*}
    since $\sqa\mu < 1$, and for all $n>2k$, $\sigma\geq 1/2$. Now for $I_2$, with a change of variables we have
    \begin{align*}
        I_2 &\leq C (\sqa\mu)^{\frac{n-2k}{2}} \bpr{\int_{\R^n \setminus \Bal{0}{1}} \abs{y}^{(2k-n)\frac{2n}{n+2k}} e^{-\abs{y}}dy}^{\frac{n+2k}{2n}}\\
            &\leq C(\sqa\mu)^{\frac{n-2k}{2}} \leq C (\sqa\mu)^{1/2}
    \end{align*}
    for all $n>2k$. Finally, for $I_3$, see that when $\dg{z,x}\geq \varrho$, we have
    \[  \abs{R(x)} \leq C \a^k \mu^{\frac{n-2k}{2}} e^{-\sqa\varrho/2} \leq C \mu^{\frac{n-2k}{2}},
        \]
    so that 
    \[  I_3 \leq C \mu^{\frac{n-2k}{2}} \leq (\sqa\mu)^{1/2}
        \]
    for all $n>2k$.
\end{proof}

\subsection{Giraud's Lemmas}\label{sec:gir} 
Write, for $\epsilon \in (0,1)$, $t\geq 0$, 
\[  \Psi_\epsilon(t) := \begin{cases}
    1 & \text{when } t < 1\\
    e^{-(1-\epsilon)t} & \text{when } t\geq 1.
\end{cases}
    \]
\begin{lemma}\label{prop:gir1}
    Let $X \in C^0(M\times M)$ and $Y\in C^0(M\times M \setminus \{(x,x)\such x\in M\})$. Assume that there exists $\epsilon\in (0,1)$, $\gamma \in (0,n]$, $\beta \in (0,n]$, and $-\gamma < \rho < n-\gamma$, such that for all $\a \geq 1$, $\mu >0$,
    \[  \abs{X(x,y)} \leq (\mu + \dg{x,y})^{-\gamma} \begin{cases}
        (\mu + \dg{x,y})^{-\rho} & \text{when } \sqa\dg{x,y} \leq 1\\
        \a^{\rho/2} e^{-(1-\epsilon)\sqa\dg{x,y}} & \text{when } \sqa\dg{x,y} \geq 1
    \end{cases}
        \]  
    for all $x,y \in M$, and 
    \[  \abs{Y(x,y)} \leq \dg{x,y}^{\beta-n} \Psi_\epsilon(\sqa \dg{x,y})
        \]
    for all $x\neq y$ in $M$. Let $Z(x,y) := \intM{X(x,z)Y(z,y)}(z)$ for all $x,y \in M$, then $Z\in C^0(M\times M)$. There exists $\a_0\geq 1$, $\mu_0 >0$ and $C>0$ such that for all $\a\geq \a_0$, $\mu\leq \mu_0$, we have the following:
    \begin{itemize}
        \item If $\sqa\dg{x,y} \leq 1$, 
            \[  \abs{Z(x,y)} \leq C \begin{cases}
                \a^{\frac{\rho+\gamma-\beta}{2}} & \text{if } \beta-\rho > \gamma\\
                1+ \abs{\log \sqa(\mu+\dg{x,y})} & \text{if } \beta-\rho = \gamma\\
                (\mu + \dg{x,y})^{\beta-\rho -\gamma} & \text{if } \beta-\rho < \gamma
            \end{cases}.
                \]
        \item If $\sqa\dg{x,y} \geq 1$,
            \[  \abs{Z(x,y)} \leq C \a^{\rho/2} \dg{x,y}^{\beta-\gamma}e^{-(1-\epsilon)\sqa\dg{x,y}}.
                \]
    \end{itemize}
\end{lemma}
\begin{lemma}\label{prop:gir2}
    Let $X \in C^0(M\times M)$ and $Y\in C^0(M\times M \setminus \{(x,x)\such x\in M\})$. Assume that there exists $\epsilon\in (0,1)$, $\beta \in (0,n]$ and $\gamma > \beta$, such that for all $\a \geq 1$, $\mu >0$,
    \begin{align*}  
        \abs{X(x,y)} &\leq (\mu + \dg{x,y})^{-\gamma} \Psi_\epsilon(\sqa\dg{x,y}) & &\forall~x,y \in M, \text{ and }\\
        \abs{Y(x,y)} &\leq \dg{x,y}^{\beta-n} \Psi_\epsilon(\sqa \dg{x,y}) &  &\forall~x\neq y \text{ in } M.
    \end{align*}   
    Let $Z(x,y) := \intM{X(x,z)Y(z,y)}(z)$ for all $x,y \in M$, then $Z\in C^0(M\times M)$. There exists $\a_0\geq 1$, $\mu_0 >0$ and $C>0$ such that for all $\a\geq \a_0$, $\mu\leq \mu_0$, we have 
    \[  \abs{Z(x,y)} \leq C\Psi_\epsilon(\sqa\dg{x,y}) \begin{cases}
        \mu^{n-\gamma} (\mu+\dg{x,y})^{\beta-n} & \text{if } \gamma>n\\
        (\mu+\dg{x,y})^{\beta-n} \big(1 + \abs{\log\frac{\mu+\dg{x,y}}{\mu}}\big) & \text{if } \gamma = n\\
        (\mu+\dg{x,y})^{\beta-\gamma} & \text{if } \gamma< n
    \end{cases}.
        \]
\end{lemma}

The proof of both Lemmas \ref{prop:gir1} and \ref{prop:gir2} is very similar to \cite[Lemma A.2]{Car24}, and is based on the standard proof for the classical Giraud's Lemma, found for instance in \cite[Proposition 4.12]{Aub82}.

\bibliographystyle{amsplain}
\bibliography{reference}

\providecommand{\bysame}{\leavevmode\hbox to3em{\hrulefill}\thinspace}
\providecommand{\MR}{\relax\ifhmode\unskip\space\fi MR }
\providecommand{\MRhref}[2]{%
  \href{http://www.ams.org/mathscinet-getitem?mr=#1}{#2}
}
\providecommand{\href}[2]{#2}
\begin{thebibliography}{10}

\bibitem{AdPaYa95}
Adimurthi, Filomena Pacella, and S.~L. Yadava, \emph{Characterization of concentration points and {$L^\infty$}-estimates for solutions of a semilinear {N}eumann problem involving the critical {S}obolev exponent}, Differential Integral Equations \textbf{8} (1995), no.~1, 41--68.

\bibitem{Aub75}
Thierry Aubin, \emph{Probl\`emes isop\'{e}rim\'{e}triques et espaces de {S}obolev}, C. R. Acad. Sci. Paris S\'{e}r. A-B \textbf{280} (1975), no.~5, Aii, A279--A281.

\bibitem{Aub82}
\bysame, \emph{Nonlinear analysis on manifolds. {M}onge-{A}mp\`ere equations}, Grundlehren der mathematischen Wissenschaften [Fundamental Principles of Mathematical Sciences], vol. 252, Springer-Verlag, New York, 1982.

\bibitem{BahCor88}
A.~Bahri and J.-M. Coron, \emph{On a nonlinear elliptic equation involving the critical {S}obolev exponent: the effect of the topology of the domain}, Comm. Pure Appl. Math. \textbf{41} (1988), no.~3, 253--294.

\bibitem{BarWetWil03}
Thomas Bartsch, Tobias Weth, and Michel Willem, \emph{A {S}obolev inequality with remainder term and critical equations on domains with topology for the polyharmonic operator}, Calc. Var. Partial Differential Equations \textbf{18} (2003), no.~3, 253--268.

\bibitem{Car24}
Lorenzo Carletti, \emph{The {G}reen's function of polyharmonic operators with diverging coefficients: Construction and sharp asymptotics}, Journal of Differential Equations \textbf{419} (2025), 370--417.

\bibitem{DjaDru01}
Zindine Djadli and Olivier Druet, \emph{Extremal functions for optimal {S}obolev inequalities on compact manifolds}, Calc. Var. Partial Differential Equations \textbf{12} (2001), no.~1, 59--84.

\bibitem{DjHeLe00}
Zindine Djadli, Emmanuel Hebey, and Michel Ledoux, \emph{Paneitz-type operators and applications}, Duke Math. J. \textbf{104} (2000), no.~1, 129--169.

\bibitem{Dru99}
Olivier Druet, \emph{The best constants problem in {S}obolev inequalities}, Math. Ann. \textbf{314} (1999), no.~2, 327--346.

\bibitem{DruHeb02}
Olivier Druet and Emmanuel Hebey, \emph{The {$AB$} program in geometric analysis: sharp {S}obolev inequalities and related problems}, Mem. Amer. Math. Soc. \textbf{160} (2002), no.~761, viii+98.

\bibitem{EsPiVe14}
Pierpaolo Esposito, Angela Pistoia, and J\'{e}r\^{o}me V\'{e}tois, \emph{The effect of linear perturbations on the {Y}amabe problem}, Math. Ann. \textbf{358} (2014), no.~1-2, 511--560.

\bibitem{FeHeRo05}
Veronica Felli, Emmanuel Hebey, and Fr\'ed\'eric Robert, \emph{Fourth order equations of critical {S}obolev growth. {E}nergy function and solutions of bounded energy in the conformally flat case}, NoDEA Nonlinear Differential Equations Appl. \textbf{12} (2005), no.~2, 171--213.

\bibitem{GazGruSw10}
Filippo Gazzola, Hans-Christoph Grunau, and Guido Sweers, \emph{Polyharmonic boundary value problems}, Lecture Notes in Mathematics, vol. 1991, Springer-Verlag, Berlin, 2010, Positivity preserving and nonlinear higher order elliptic equations in bounded domains.

\bibitem{Heb99}
Emmanuel Hebey, \emph{Nonlinear analysis on manifolds: {S}obolev spaces and inequalities}, Courant Lecture Notes in Mathematics, vol.~5, New York University, Courant Institute of Mathematical Sciences, New York; American Mathematical Society, Providence, RI, 1999.

\bibitem{Heb03}
\bysame, \emph{Sharp {S}obolev inequalities of second order}, J. Geom. Anal. \textbf{13} (2003), no.~1, 145--162.

\bibitem{Heb14}
\bysame, \emph{Compactness and stability for nonlinear elliptic equations}, Zurich Lectures in Advanced Mathematics, European Mathematical Society (EMS), Z\"{u}rich, 2014.

\bibitem{HebRob01}
Emmanuel Hebey and Fr\'{e}d\'{e}ric Robert, \emph{Coercivity and {S}truwe's compactness for {P}aneitz type operators with constant coefficients}, Calc. Var. Partial Differential Equations \textbf{13} (2001), no.~4, 491--517.

\bibitem{HebVau95}
Emmanuel Hebey and Michel Vaugon, \emph{The best constant problem in the {S}obolev embedding theorem for complete {R}iemannian manifolds}, Duke Math. J. \textbf{79} (1995), no.~1, 235--279.

\bibitem{HebVau96}
\bysame, \emph{Meilleures constantes dans le th\'{e}or\`eme d'inclusion de {S}obolev}, Ann. Inst. H. Poincar\'{e} C Anal. Non Lin\'{e}aire \textbf{13} (1996), no.~1, 57--93.

\bibitem{Maz16}
Saikat Mazumdar, \emph{G{JMS}-type operators on a compact {R}iemannian manifold: best constants and {C}oron-type solutions}, J. Differential Equations \textbf{261} (2016), no.~9, 4997--5034.

\bibitem{Maz17}
\bysame, \emph{Struwe's decomposition for a polyharmonic operator on a compact {R}iemannian manifold with or without boundary}, Commun. Pure Appl. Anal. \textbf{16} (2017), no.~1, 311--330.

\bibitem{MusPis02}
Monica Musso and Angela Pistoia, \emph{Multispike solutions for a nonlinear elliptic problem involving the critical {S}obolev exponent}, Indiana Univ. Math. J. \textbf{51} (2002), no.~3, 541--579.

\bibitem{Pre18}
Bruno Premoselli, \emph{A pointwise finite-dimensional reduction method for a fully coupled system of {E}instein-{L}ichnerowicz type}, Commun. Contemp. Math. \textbf{20} (2018), no.~6, 1750076, 72.

\bibitem{Pre22}
\bysame, \emph{Towers of bubbles for {Y}amabe-type equations and for the {B}r\'ezis-{N}irenberg problem in dimensions {$n\geq 7$}}, J. Geom. Anal. \textbf{32} (2022), no.~3, Paper No. 73, 65.

\bibitem{Pre24}
\bysame, \emph{A {P}riori {E}stimates for {F}inite-{E}nergy {S}ign-{C}hanging {B}lowing-{U}p {S}olutions of {C}ritical {E}lliptic {E}quations}, Int. Math. Res. Not. IMRN (2024), no.~6, 5212--5273.

\bibitem{Rey02}
Olivier Rey, \emph{The question of interior blow-up-points for an elliptic {N}eumann problem: the critical case}, J. Math. Pures Appl. (9) \textbf{81} (2002), no.~7, 655--696.

\bibitem{Rob11}
Fr\'{e}d\'{e}ric Robert, \emph{Admissible {$Q$}-curvatures under isometries for the conformal {GJMS} operators}, Nonlinear elliptic partial differential equations, Contemp. Math., vol. 540, Amer. Math. Soc., Providence, RI, 2011, pp.~241--259.

\bibitem{RobVet13}
Fr\'{e}d\'{e}ric Robert and J\'{e}r\^{o}me V\'{e}tois, \emph{A general theorem for the construction of blowing-up solutions to some elliptic nonlinear equations via {L}yapunov-{S}chmidt's finite-dimensional reduction}, Concentration analysis and applications to {PDE}, Trends Math., Birkh\"{a}user/Springer, Basel, 2013, pp.~85--116.

\bibitem{Swa92}
Charles~A. Swanson, \emph{The best {S}obolev constant}, Appl. Anal. \textbf{47} (1992), no.~4, 227--239.

\bibitem{Tal76}
Giorgio Talenti, \emph{Best constant in {S}obolev inequality}, Ann. Mat. Pura Appl. (4) \textbf{110} (1976), 353--372.

\bibitem{Wan95}
Zhi~Qiang Wang, \emph{High-energy and multi-peaked solutions for a nonlinear {N}eumann problem with critical exponents}, Proc. Roy. Soc. Edinburgh Sect. A \textbf{125} (1995), no.~5, 1003--1029.

\bibitem{WeiXu99}
Juncheng Wei and Xingwang Xu, \emph{Classification of solutions of higher order conformally invariant equations}, Math. Ann. \textbf{313} (1999), no.~2, 207--228.

\bibitem{Zei24}
Samuel Zeitler, \emph{A sharp higher order sobolev inequality on riemannian manifolds}, arXiv preprint 2409.08920, 2024.

\end{thebibliography}

\end{document}